\documentclass[12pt]{amsart}
\usepackage{hyperref}
\usepackage{amsmath,amsfonts,amsthm,amscd,amssymb,graphicx}
\usepackage{fullpage}
\usepackage{bm,bbm}
\theoremstyle{plain}
\newtheorem{theorem}{Theorem}[section]
\newtheorem{lemma}[theorem]{Lemma}
\newtheorem{proposition}[theorem]{Proposition}
\newtheorem{corollary}[theorem]{Corollary}

\theoremstyle{definition}
\newtheorem{definition}[theorem]{Definition}

\theoremstyle{remark}
\newtheorem{remark}[theorem]{Remark}

\newcommand{\beq}{\begin{equation}}
\newcommand{\eeq}{\end{equation}}
\newcommand{\bpm}{\begin{pmatrix}}
\newcommand{\epm}{\end{pmatrix}}

\DeclareMathOperator{\tr}{trace}

\DeclareMathOperator{\re}{Re}
\DeclareMathOperator{\im}{Im}
\renewcommand{\Re}{\re}
\renewcommand{\Im}{\im}

\renewcommand{\vec}[1]{\ensuremath{\mathbf{#1}}}
\newcommand{\mat}[1]{\ensuremath{\mathsf{#1}}}
\newcommand{\norm}[1]{\ensuremath{\|#1\|}}
\newcommand{\mi}{\ensuremath{\mathrm{i}}}
\newcommand{\me}{\ensuremath{\mathrm{e}}}
\newcommand{\dif}{\ensuremath{\mathrm{d}}}

\newcommand{\spm}{{\ensuremath{{\scriptscriptstyle\pm}}}}
\renewcommand{\sp}{{\ensuremath{{\scriptscriptstyle +}}}}
\newcommand{\sm}{{\ensuremath{{\scriptscriptstyle -}}}}

\def\bb1{{1\!\!1}}

\renewcommand{\epsilon}{\varepsilon}
\numberwithin{equation}{section}

\date{Last Updated:  \today}

\title[Stability of Navier--Stokes shocks]{Multidimensional stability of large-amplitude Navier--Stokes shocks}
\author[Humpherys, Lyng, and Zumbrun]{Jeffrey Humpherys, Gregory Lyng, and Kevin Zumbrun}

\thanks{J.H. was partially supported by NSF grant DMS-CAREER-0847074. 
G.L. was partially supported by NSF grant DMS-CAREER-0845127.
K.Z. was partially supported  by NSF grants DMS-0300487 and DMS-0801745.}

\address{Department of Mathematics, Brigham Young University, Provo, UT 84602}
\email{jeffh@math.byu.edu}
\address{Department of Mathematics, University of Wyoming, Laramie, WY 82071}
\email{glyng@uwyo.edu}
\address{Department of Mathematics, Indiana University, Bloomington, IN 47402}
\email{kzumbrun@indiana.edu}

\begin{document}
\begin{abstract}
Extending results of Humpherys-Lyng-Zumbrun in the one-dimensional case, we use a combination of asymptotic ODE estimates and numerical Evans-function computations to examine the multidimensional stability of planar Navier--Stokes shocks across the full range of shock amplitudes, including the infinite-amplitude limit, for monatomic or diatomic ideal gas equations of state and viscosity and heat conduction coefficients $\mu$, $\mu +\eta$, and $\nu=\kappa/c_v$ in the physical ratios predicted by statistical mechanics, with Mach number $M>1.035$.
Our results indicate unconditional stability within the parameter range considered; this agrees with the results of Erpenbeck and Majda for the corresponding inviscid case of Euler shocks. 
Notably, this study includes the first successful numerical computation of an Evans function associated with the multidimensional stability of a viscous shock wave. The methods introduced may be used in principle to decide stability for any $\gamma$-law gas, $\gamma>1$, and arbitrary $\mu>|\eta|\ge 0$, $\nu>0$ or, indeed, for shocks of much more general models and equations, including in particular viscoelasticity, combustion, and magnetohydrodynamics (MHD).
\end{abstract}
\maketitle

\tableofcontents


\section{Introduction}\label{sec:intro}

In this paper---extending to multiple space dimensions results of Humpherys-Lafitte-Zumbrun \cite{HLZ} and Humpherys-Lyng-Zumbrun \cite{HLyZ} in the one-dimensional case---we address by a combination of numerical and analytical methods the long-standing question of stability of Navier--Stokes shock waves for a polytropic equation of state with shock amplitude not necessarily small.  Remarkably, by a rescaling/compactification argument as in \cite{HLyZ}, we are able to obtain uniform estimates independent of shock amplitude, including the infinite-amplitude limit, obtaining convincing numerical evidence that for a monatomic or diatomic ideal gas equation of state and viscosity and heat conduction coefficients in the physical ratios predicted by statistical mechanics, \emph{Navier--Stokes shocks are multidimensionally stable, independent of amplitude or other parameters}.

This confirms the predictions of the inviscid Euler case, for which the corresponding result was established by Erpenbeck and Majda \cite{Er,Maj4}.  The treatment of the multidimensional case involves substantial new technical difficulties, both numerical and analytical, beyond that of the one-dimensional case.  Indeed, this study includes the first successful numerical Evans-function computations for multidimensional stability of a viscous shock wave of any type.


\subsection{Background and description of main results}
\label{ssec:back}
Stability is a fundamental issue in the study of shock waves.  Indeed, one of their defining features is the ability to transmit large amounts of energy over large distances in the form of a coherent wave front.  The mathematical study of shock stability, connected with properties of compressibility, initially lagged behind the study of other forms of hydrodynamic stability connected with incompressible flow, partly perhaps due to complexity of the problem and partly perhaps because shocks under most circumstances are extremely stable \cite{BE}.  However, since its rigorous formulation by Erpenbeck \cite{Er} and others in the 1960's, the problem of shock stability has been the object of intensive investigation by a number of different authors.  We mention in particular the ``inviscid'' studies of Erpenbeck, Majda, and Metivier \cite{Er, Maj1, Maj2, Maj3, Me}, largely settling the question in the case where second-order transport effects of viscosity, heat conduction, magnetic resistivity, etc. are neglected, and their striking corollary, {\bf ``Majda's Theorem:}'' {\it for a polytropic ideal gas equation of state, inviscid shock waves are spectrally and nonlinearly stable, independent of shock amplitude,}
in one and multiple space dimensions.\footnote{Shown by Erpenbeck \cite{Er} at the level of spectral stability;
see also the one-dimensional investigations of Lax \cite{La}.}

The above-mentioned inviscid studies concern discontinuous  solutions of systems of first-order quasilinear hyperbolic conservation laws, which are constant to either side of the shock front. When transport effects are included, the structure becomes more complicated, with the shock discontinuity replaced by a smooth but rapidly varying ``shock layer'' joining asymptotically constant states at infinity.  In most cases, the profile of this shock layer is not explicitly known, but only guaranteed by abstract theory to exist. For treatments of the shock structure problem, see for example the studies of Bethe \cite{Be}  in the inviscid setting, and Weyl \cite{W} and Gilbarg \cite{Gi} in the ``viscous,'' or Navier--Stokes case.\footnote{For further history, going back to work of Stokes, Rayleigh, and Gel'fand, see \cite{W,Gi,BE}.}

The study of stability of viscous shock layers, was initiated at the one-dimensional scalar level by Hopf \cite{Ho} and Il'in-Ole\u inik \cite{IO}.  For one-dimensional systems, it was begun in the 1980's by Kawashima-Matsumura, Kawashima-Matsumua-Nishihara, Liu, and Goodman \cite{KM,KMN,L1,Go1,Go2}, and essentially concluded in \cite{SX,L2,GZ,ZH,MZ1,MZ2,MZ3,HuZ1,HLZ,HLyZ,HRZ,RZ}.  We note in particular the proof by Mascia-Zumbrun and Humpherys-Zumbrun \cite{MZ2,HuZ1} for the first time of small-amplitude (one-dimensional) stability of ordinary gas-dynamical and Lax-type magnetohydrodynamic Navier--Stokes shocks with general equation of state, and the proof by Mascia-Zumbrun and Raoofi-Zumbrun \cite{MZ3,RZ} of nonlinear (one-dimensional) stability of large-amplitude shock solutions of arbitrary type for a class of systems generalizing the Kawashima class \cite{K,KSh}, including gas dynamics, viscoelasticity, and magnetohydrodynamics (MHD), assuming a numerically verifiable \emph{Evans-function condition} encoding spectral stability in an appropriate sense; that is, the Evans-function condition accounts for the lack of spectral gap/accumulating essential spectrum at the origin that is an fundamental feature of the shock stability problem.\footnote{For a discussion of the Evans function and ``effective spectra,'' see \cite{ZH,MZ1}.}  Finally, we note the analytical/numerical studies in \cite{HLZ} and \cite{HLyZ}, Humpherys-Lafitte-Zumbrun and Humpherys-Lyng-Zumbrun, respectively, {\it verifying spectral and nonlinear stability of arbitrary amplitude polytropic ideal gas shock layers for the isentropic and non-isentropic Navier--Stokes equations} with gas constant $\gamma \in [1.2,3]$.

The study of multidimensional viscous stability for scalar equations was initiated by Goodman \cite{Go3} for small-amplitude shocks,  and extended to large amplitudes by Goodman-Miller and Hoff-Zumbrun \cite{GoM,HoZ1,HoZ2}.  The study of multidimensional stability for systems, requiring substantially new ideas beyond those of the scalar case, was initiated by Zumbrun-Serre and Zumbrun in \cite{ZS,Z2} with various important extensions carried out by Freist\"uhler-Szmolyan, Zumbrun, and Nguyen, respectively, in \cite{FS2,Z3,Z5,N}.
The current state of the art as regards multidimensional stability for viscous shock layers of the Navier-Stokes equations and related physical systems is that generalized spectral stability, as represented by an appropriate \emph{Evans-function condition}, is sufficient (and with further elaboration, essentially necessary as well \cite{Z6}) for nonlinear asymptotic stability with explicit time-algebraic rates of decay in $L^p$, $p\geq 2$; see \cite{Z3,N}.  For analogous results on the associated inviscid limit problem, see \cite{GMWZ}.

At an abstract level, this is completely parallel to the inviscid theory, where nonlinear stability reduces to a generalized spectral stability condition phrased in terms of a {\it Lopatinski determinant} analogous to the Evans function.  Indeed, the Lopatinski condition of inviscid theory may be shown to be the low-frequency limit of the Evans function of the viscous theory \cite{ZS,Z2}.  However, unlike the Lopatinski determinant, which, being associated with a constant-coefficient problem, is explicitly computable in terms of the endstates of the shock, the Evans function, determined by solution of a variable-coefficient system of ordinary differential equations (ODEs), is almost never explicitly computable, except in certain asymptotic limits.  {In fact, up to the time of this writing, the multidimensional Evans condition had not been verified to our knowledge either analytically or numerically for any single viscous shock wave.}

The question we consider here is whether Majda's Theorem---evidently a condition on low-frequency stability---extends to a corresponding result of full multidimensional viscous stability.  Our main result is to show, by numerically well-conditioned and analytically justified computations that in fact an analogous result does hold for the viscous case, at least for a $\gamma$-law gas with $\gamma=5/3$ or $\gamma=7/5$, corresponding to a monatomic or diatomic gas, and Mach number $M\geq 1.035$ not too close to the small-amplitude limit $M=1.0$.

In passing, we develop a number of new techniques to deal with difficulties special to the multidimensional case, in particular: $(i)$ a flexible yet sharp ``high-frequency tracking'' algorithm allowing us to obtain reasonable bounds on the size of possible unstable frequencies/eigenvalues while maintaining reasonable computational expense; $(ii)$ a ``modified balanced flux'' coordinatization of the Evans function analogous to the integrated coordinates of one-dimensional theory, namely, factoring out zeros at the origin while maintaining analyticity in the temporal frequency $\lambda$; and $(iii)$ a ``pseudo-Lagrangian'' change of  independent variable factoring out excessive variation arising in the Eulerian coordinates natural for the multidimensional problem.  These appear to be necessary even to carry out successful Evans-function computations for a single wave of moderate amplitude; see \cite{BHLyZ1,BHLyZ2} for further discussion.


\subsection{Discusion and open problems}
\label{ssec:discussion}
Our results extend and complete the earlier one-dimensional investigations of \cite{HLyZ}.  A remarkable feature of both analyses is that we are able to treat simultaneously all possible shock amplitudes, making use of invariance of the  polytropic Navier--Stokes equations, by a {\it rescaling/compactification argument.}
Specifically, rescaling so that the both the density at the left endstate $\rho_\sm$ and the shock speed $s$ in Lagrangian coordinates (resp., mass flux $m$ in Eulerian coordinates) be one, we find that all gas-dynamical quantities in the profile and eigenvalue ODEs remain bounded and bounded away from zero, save for internal energy at the left endstate $e_\sm$, which approaches the nonphysical value zero in the infinite-(non-rescaled) amplitude limit.

Noting that all coefficients remain smooth (indeed, analytic) for $e=0$, and checking by special structure of the equations that stable/unstable subspaces of the limiting eigenvalue ODEs at spatial infinity have continuous limits as $e_\sm\to 0$, so that we can make sense of the notion of extended eigenvalue there in a way that perturbs smoothly, we are able to adjoin the limiting value $e_\sm=0$ corresponding to infinite shock amplitude to the computational domain, thus obtaining a compact parameter space on which we can carry out uniformly well-conditioned computations.

A substantial difficulty, however, is that for $e=0$ the first-order hyperbolic (Euler) part of the Navier--Stokes equations is no longer symmetrizable.  As a result, the standard ``Kawashima-type'' hyperbolic--parabolic energy estimates that have been used in \cite{Z5} and elsewhere to bound the size of possible unstable frequencies/eigenvalues no longer apply, and we must appeal to more special and/or complicated methods to bound uniformly the range of frequencies under consideration.  Here, we follow a ``tracking'' argument similar to that used in \cite{HLyZ}, based on invariant cones and Ricatti-type energy estimates generalizing those of \cite{GZ,ZH} in the strictly parabolic case,\footnote{ See also \cite{GMWZ2} for links to a type of Kreiss symmetrizer estimate introduced in \cite{MeZ2}.} in which the equations are conjugated by a series of transformations to an equivalent partially symmetrized form in which stable and unstable subspaces have not only a spectral gap but a gap in numerical range.

A new difficulty in the multidimensional case is that there are multiple scaling regimes for the eigenstructure of the asymptotic eigenvalue problem, and so these transformations do not lend themselves to the type of series expansion that was carried out in \cite{HLyZ}, nor is the symmetrization step readily carried out by hand.  We replace these steps therefore with a more general approach in which these steps are replaced by numerical computations; the resulting algorithm, described in Appendix \ref{tracking}, both clarifies and greatly extends that of \cite{HLyZ}.

A further difficulty of the multidimensional case is the unavailability of the ``integrated'' coordinates of the one-dimensional case, which have the desirable property of factoring out eigenvalues at the origin while maintaining analyticity of the eigenvalue ODE (hence also the associated Evans function) with respect to the spectral parameter $\lambda$.  As noted earlier, we overcome this by the introduction of a new ``modified balanced flux'' coordinatization that shares these properties.

A final new aspect is practical unavailability of Lagrangian coordinates in the multidimensional stability problem.  Though Lagrangian coordinates may be introduced, at the expense of carrying $d^2-1$ additional variables, where $d$ the spatial dimension, corresponding to the $d^2$ entries of the stress tensor, they introduce a massive ambiguity in the equations due to the infinite-dimensional family of invariances corresponding to volume-preserving maps of the spatial variable; see \cite{PYZ} for further discussion.  This results in an infinite-dimensional family of spurious pure imaginary essential spectra complicating both the stability problem and potential construction of an Evans function.  On the other hand, it turns out that the usual construction of the Evans function, in Eulerian coordinates, has asymptotic behavior yielding unacceptably large variation at high frequencies.  Indeed, we find it practically incomputable, except at small frequencies.  We remedy this by a change of the independent spatial variable $x_1$ to its corresponding value in the one-dimensional Lagrangian representation.  Remarkably, by this simple ``pseudo-Lagrangian'' coordinatization we recover the favorable asymptotics of the one-dimensional Lagrangian case, and practically feasible
computations \cite{BHLyZ1}.

These innovations not only make possible the uniform treatment of large-amplitude Navier--Stokes shocks, but also for the first time practical multidimensional stability computations in a variety of different settings.  Two particularly interesting directions for further investigation are viscous MHD shocks and detonations, for both of which instabilities are known to occur, with interesting associated bifurcations, and for which the effects of viscosity are as yet unclear \cite{Mo,PYZ,TZ2,TZ3,SS,Z6,BHLyZ3}.  A related direction for study is on the possible relation between types of instabilities and an associated convex entropy; see \cite{BFZ,LV,TZ4} for related one-dimensional investigations.  The study of large-amplitude stability for temperature-dependent transport coefficients, or kinetic models, are further interesting directions.

Another interesting open problem is stability of Navier--Stokes shocks in the {\it small-amplitude limit}: this ``characteristic'' limit is not accessible to direct numerical investigation, but should be possible (both analytically and numerically) by the singular perturbation/asymptotic ODE techniques used here and in \cite{FS1,FS2,PZ,HLZ,HLyZ,BHZ1}.  For a treatment of the multidimensional problem in the artificial (Laplacian) viscosity case, see \cite{FS2}.  See \cite{PZ} for a treatment of the partially parabolic and related relaxation case in one dimension.

Finally, we emphasize that, though our computations here are numerically well-conditioned and based on rigorous theory, they constitute convincing numerical evidence and not rigorous numerical proof.  To convert the present analysis to numerical proof based on interval arithmetic is an interesting open problem, involving separate and interesting issues of feasible computation, and, given the fundamental nature of the problem, certainly seems warranted. The successful work of Barker \cite{B} in a related context serves as a proof of concept in this regard.


\subsection{Plan of the Paper}
\label{ssec:plan}

In Sections \ref{sec:equations} and \ref{sec:viscous}, we present the Navier--Stokes and shock profile equations, then, in Subsections \ref{rescaled}, \ref{rhcond}, and \ref{hypcheck}, our choice of rescaled coordinates and uniform existence/decay of shock profiles.  In Sections \ref{lineig} and \ref{evansform} we derive the linearized eigenvalue ODE, and carry out the construction of the Evans function in balanced flux coordinates.

In Section \ref{sec:hfb}, we carry out the key high-frequency tracking analysis uniformly bounding the frequency domain under consideration.  In Section \ref{numerical}, we carry out a low-frequency Evans function analysis using the balanced flux formulation, establishing stability for frequencies sufficiently near zero.  Finally, in Section \ref{sec:intermediate}, we carry out a large-scale intermediate frequency Evans function analysis using a combination of standard and modified balanced flux formulations together with pseudo-Lagrangian coordinates, verifying stability on the remaining frequency regimes.  We summarize the argument in Section \ref{sec:summary}, concluding that Navier--Stokes shocks are stable for monatomic and diatomic ideal gas equations of state and Mach number $M>1.035$.  Various useful tools and auxiliary materials are collected in the appendices.

\medskip{\bf Acknowledgment.}  We thank Blake Barker for his generous help, both in checking independently intermediate-frequency computations and in numerous useful discussions.


\section{Compressible Navier--Stokes equations}
\label{sec:equations}

Before starting, we make the standard observation \cite{Er,Maj1} that, by rotational invariance of the Navier--Stokes equations, together with the assumed rotational symmetry (since constant) of a planar shock wave in transverse directions (i.e., transverse to the direction of propagation), it is sufficient in the study of spectral stability to restrict to the case of dimension two.  Specifically, taking the Fourier transform in transverse directions, and denoting the resulting frequency as $\xi$, one finds that the eigenvalue equations depend only on $|\xi|$, hence reduce to the two-dimensional case. Thus, there is no loss of generality in restricting to two dimensions, as we hereafter do. In Eulerian coordinates, the Navier--Stokes equations for compressible gas dynamics in two spatial dimensions can be written in the form
\begin{subequations}\label{eq:ns}
\begin{equation}
\rho_t+ (\rho u)_{x_1} +(\rho v)_{x_2}=0,\label{eq:mass}
\end{equation}
\begin{equation}
(\rho u)_t+ (\rho u^2+p)_{x_1} + (\rho uv)_{x_2}=(2\mu +\eta) u_{x_1x_1}+ \mu u_{x_2x_2} +(\mu+ \eta )v_{x_1x_2},\label{eq:momentumx}
\end{equation}
\begin{equation}
(\rho v)_t+ (\rho uv)_{x_1} + (\rho v^2+p)_{x_2}=\mu v_{x_1x_1}+ (2\mu +\eta) v_{x_2x_2} + (\mu+\eta) u_{x_2x_1},\label{eq:momentumy}
\end{equation}
\begin{multline}
\label{eq:energy}
(\rho E)_t+ (\rho uE+up)_{x_1} + (\rho vE+vp)_{x_2}\\
=\Big( \kappa T_{x_1} + (2\mu+\eta)uu_{x_1} + \mu v(v_{x_1}+u_{x_2}) + \eta uv_{x_2}\Big)_{x_1} \\
+\Big( \kappa T_{x_2}+ (2\mu+\eta)vv_{x_2} + \mu u(v_{x_1}+u_{x_2}) + \eta vu_{x_1}\Big)_{x_2}, 
\end{multline}
\end{subequations}
where $\rho$ is density, $u$ and $v$ are the fluid velocities in $x_1$ and $x_2$
directions respectively, $p$ is pressure, and the specific energy $E$ is made up of the specific internal  
energy $e$ and kinetic energy:
\begin{equation}
E=e+\frac{u^2}{2} +\frac{v^2}{2}.\label{eq:internal_kinetic}
\end{equation}
The constants $\mu>|\eta|\ge0$ and $\kappa>0$ are coefficients of first (``dynamic'') and second viscosity and heat conductivity.  Finally, $T$ is the temperature,  and we assume that the specific internal energy $e$ and the pressure $p$ are known functions of density and temperature:
\begin{equation}
\label{eq:general_eos}
p=p_0(\rho,T),\quad e=e_0(\rho,T).
\end{equation}
\begin{remark}
In our use of \eqref{eq:general_eos} to close the system \eqref{eq:ns}, we are assuming, following the terminology of Menikoff and Plohr \cite{MP} (see also the discussion in \cite{B-GS}), that the fluid is endowed with a \emph{complete equation of state}. In their framework, the specific internal energy is given everywhere as a function of the specific volume $\tau=1/\rho$ (we use $\tau$ to avoid a notational conflict with $v$ which we are using to denote velocity) and the specific entropy $S$:
\[
e=e(\tau,S).
\]
Then, the pressure $p$ and the temperature $T$ are obtained from the state equation via 
\[
p=-\left.\frac{\partial e}{\partial \tau}\right|_S,\quad T=\left.\frac{\partial e}{\partial S}\right|_\tau.
\]
Other important quantities in our analysis are the specific heats at constant volume and pressure (these are measurable)
\[
c_v=T\left.\frac{\partial S}{\partial T}\right|_\tau,\quad c_p=T\left.\frac{\partial S}{\partial T}\right|_p,
\]
and the (dimensionless) adiabatic exponent and Gruneisen coefficient
\[
\gamma=\frac{\tau}{p}\left.\frac{\partial^2 e}{\partial\tau^2}\right|_S,
\quad
\Gamma=-\frac{\tau}{T}\frac{\partial^2e}{\partial S\partial\tau}.
\]
\end{remark}

An important special case occurs when we consider a polytropic ideal gas.  In this case the energy and pressure functions take the specific form
\begin{equation}
\label{eq:ideal_gas}
p_0(\rho,T)=R\rho T,\quad e_0(\rho,T)=c_vT,
\end{equation}
where the gas constant, $R> 0$, and the specific heat at constant volume, $c_v>0$, are constants that depend on the gas \cite{Sm}.

Alternatively, we may use \eqref{eq:ideal_gas} to write the pressure in terms of the density and the specific internal energy as  
\begin{equation}
\label{eq:Gamma}
p=\Gamma \rho e,
\end{equation}
where the \emph{Gruneisen coefficient} $\Gamma$ is given by $\Gamma:= \frac{R}{c_v }=\gamma -1> 0$ and $\gamma =c_p/c_v \ge 1$ is the \emph{adiabatic exponent}, which, in the polytropic setting, is the ratio of specific heat at constant pressure $c_p$ to specific heat at constant volume $c_v$.  Equivalently, $p(\rho,S)=a\me^{S/c_v}\rho^{\gamma},\;a=\text{constant}$, where $S$ is thermodynamical entropy, or $p(\rho)=a\rho^{\gamma}$ in the isentropic approximation; see \cite{Sm,BHRZ,HLZ}.

\begin{remark}
In the literature, a gas satisfying the pressure law
\[
p=(\gamma-1)\rho e
\]
is commonly referred to as a ``$\gamma$-law'' gas. The main calculations of this paper will be in the setting of polytropic ideal gas. That is, we use \eqref{eq:ideal_gas} to close the system \eqref{eq:ns}.
\end{remark}

In the thermodynamical rarified gas approximation, $\gamma>1$ is the average over constituent particles of $\gamma=(N+2)/N$, where $N$ is the number of internal degrees of freedom of an individual particle, or, for molecules with ``tree'' (as opposed to ring, or other more complicated) structure,
\begin{equation}
\label{gammaformula}
\gamma=\frac{2n+3}{2n + 1},
\end{equation}
where $n$ is the number of constituent atoms \cite{Ba}: $\gamma= 5/3$ for monatomic, $\gamma= 7/5$ for  diatomic gas.  For dense fluids, $\gamma$ is typically determined phenomenologically \cite{H}. In general, $\gamma$ is usually taken within $1 \leq \gamma \leq 3$ in models of gas- or fluid-dynamical flow, whether phenomenological or derived by statistical mechanics \cite{Sm,Se1,Se2}.

The dynamic viscosity $\mu$ is the constant of proportionality asserted in Newton's law of viscosity between shear stress and velocity gradient of a shear flow.  This is readily measured, and is the value usually reported physical tables.  For an ideal gas, values of the second viscosity $\eta$ are less clear \cite{Ro}.  For incompressible flows, a common assumption is $\eta=0$.  For compressible flows, a common assumption is
\begin{equation}
\label{firstsecond}
\eta=-(2/3)\mu,
\end{equation}
which amounts to the assumption that pressure is equal to ``mean pressure'' defined as one-third the trace of the fluid-dynamical stress tensor,
and which seems to agree well with experiment at least for monatomic and diatomic gases; in particular, $\eta$ is typically negative \cite{Ba, Ro}.

Heat conductivity $\kappa$ is related to $\mu$ by the  dimensionless ratio 
\begin{equation}
{\rm Pr}:=c_p \mu/\kappa,
\end{equation}
or {\it Prandtl number}, which is predicted (somewhat less successfully than $\Gamma$, $\gamma$) in the statistical mechanical ideal gas theory by {\it Eucken's formula} \cite{Bro}
\begin{equation}
\label{Euc}
{\rm Pr}=\frac{4\gamma}{9\gamma-5},
\end{equation}
The key parameter $\nu:=\kappa/c_v$ arising in our analysis thus satisfies
\begin{equation}
\label{nu_mu}
\nu/\mu= \gamma/ {\rm Pr},
\end{equation}
with a theoretically predicted value (from \eqref{Euc}) of 
\begin{equation}
\label{simplest}
\nu/\mu= \frac{9\gamma -5}{4}.
\end{equation}
Typical values for dry air at normal (e.g., room) temperatures, expressed in dimensionless constants, are
\[
\gamma\approx 1.4, \quad \nu/\mu=\frac{\kappa}{c_v\mu}\approx 1.96, 
\]
with $\eta/\mu\approx -.666$ according to \eqref{firstsecond}; see Appendix A, \cite{HLyZ} for further discussion.


\section{Viscous shock profiles}\label{sec:viscous}
\subsection{Traveling-wave equation}
A viscous shock profile of \eqref{eq:ns}  
is a traveling-wave solution,
\begin{equation}
(\rho,u,v,T)(x,t)=(\hat \rho, \hat u, \hat v, \hat T)(x_1-st)
\label{eq:tw_ansatz}
\end{equation}
moving with speed $s$ and connecting constant states
$(\rho_\spm,u_\spm,v_\spm, T_\spm)$.
By Galilean invariance, we may take without loss of generality a  
standing
shock profile, $s=0$, as we shall do hereafter, to obtain
the system of standing-wave ODEs (primes denote differentiation with respect to $x_1$)
\begin{subequations}\label{ode}
\begin{align}
(\rho u)'&=0\,,\\
(\rho u^2)' + p'&=(2\mu +\eta) u''\,,\\
(\rho uv)' &=\mu v''\,,\\
(\rho u E)' + (pu)'&=(2\mu+\eta )(uu')'+\kappa T''+\mu(vv')'\,.
\end{align}
\end{subequations}
From the first equation $m:=\rho u \equiv {\rm constant}$, hence
the third equation becomes $mv'=\mu v''$, for which the only bounded
solutions are $v\equiv {\rm constant}$.
Without loss of generality (by Galilean coordinate change if  
necessary),
we take $v\equiv 0$, yielding
\begin{subequations}\label{ode2}
\begin{align}
m u' + p'&=(2\mu+\eta) u''\,,\\
m E' + (pu)'&=(2\mu+\eta)( uu')'+\kappa T''\,.
\end{align}
\end{subequations}
Integrating from $-\infty$ to $x_1$ and rearranging, we obtain
the first-order system
\begin{subequations}\label{odefirst}
\begin{align}
u'& = (2\mu+\eta)^{-1}\Big(m (u-u_\sm) + (p(\rho,T)-p_\sm)\Big)\,,\\
T'& = \kappa^{-1}\Big( m \big(e(\rho, T)-e_\sm\big)
- \frac{m(u-u_\sm)^2}{2} +(u-u_\sm)p_\sm\Big)\,,
\end{align}
\end{subequations}
which may be closed through the relation $\rho=m u^{-1}$.
Specialized to the ideal gas case, equation \eqref{odefirst} becomes
\begin{subequations}\label{idealode}
\begin{align}
u'& = (2\mu+\eta)^{-1}\Big(m (u-u_\sm) + R(\rho T- \rho_\sm T_\sm)\Big)\,,\\
T'&= \kappa^{-1}\Big( m c_v \big(T- T_\sm\big)
- \frac{m(u-u_\sm)^2}{2} +(u-u_\sm)R\rho_\sm T_\sm\Big)\,.
\end{align}
\end{subequations}


\subsection{Rescaled coordinates}
\label{rescaled}

The Navier--Stokes equations are invariant under the rescaling
\begin{equation}\label{scales}
(x_1,x_2,t;\,\rho,u,v,T)\to 
\left(m x_1, mx_2,\epsilon m^2t;\, \epsilon \rho, \frac{u}{\epsilon m}, \frac{v}{\epsilon m},  
\frac{T}{\epsilon^2 m^2}\right),
\end{equation}
where the pressure and internal energy in the (new) rescaled  
variables are given by
\begin{equation}
p(\rho,T)=\epsilon^{-1}m^{-2}p_0(\epsilon^{-1} \rho,\epsilon^2m^2T)
\label{eq:new_pressure}
\end{equation}
and
\begin{equation}
e(\rho,T)=\epsilon^{-2}m^{-2}e_0(\epsilon^{-1} \rho,\epsilon^2m^2 T)\,;
\label{eq:new_energy}
\end{equation}
in the ideal gas case, \emph{the pressure and internal energy laws
remain unchanged}
\begin{equation}
p(\rho,T)=R \rho T,
\qquad
e(\rho,T)=c_vT,
\label{eq:ideal_gas_scaling1}
\end{equation}
with the same constants $R$, $c_v$.  
Likewise, $\Gamma$ remains unchanged in \eqref{eq:Gamma},
for any $\epsilon, m > 0$.

From now on, we restrict to the ideal gas case, and we
choose $\epsilon$ and $m$ so that $\rho_\sm=1$ and $u_-=1$.
(To put things another way, we choose
$m=\hat m:=\hat \rho \hat u$ as the constant of integration for the (non-rescaled)
profile ODE, thus normalizing to $m=1$ in rescaled coordinates.)
Setting $\nu:= \kappa/c_v$, changing to $e$-coordinates, and noting that $m=1$ after rescaling, we find that system \eqref{idealode} becomes, simply,
\begin{subequations}
\label{midealode}
\begin{align}
u'=& \frac{1}{2\mu+\eta}\left( (u-1) + \Gamma \left(\frac{e}{u}- e_\sm\right)\right)\,,\\
e'=& \nu^{-1}\left((e- e_\sm)
- \frac{(u-1)^2}{2} +(u-1)\Gamma e_\sm\right)\,.
\end{align}
\end{subequations}

\subsection{Rankine--Hugoniot conditions}
\label{rhcond}

The Rankine--Hugoniot conditions are
\begin{subequations}
\label{eq:rh}
\begin{align}
[\rho u] &= 0,\label{eq:rh_mass}\\
[\rho u^2]+ [p(\rho, e)]&=0,\label{eq:rh_momentum}\\
\left[\rho u (e+\frac{u^2}{2})\right]+ [p(\rho, e)u]&=0.\label{eq:rh_energy}
\end{align}
\end{subequations}
In \eqref{eq:rh}, the square brackets denote the difference (jump) between end states. That is, if $h$ is some function of $\rho$, $u$ and $e$, then 
\[
[h(\rho, u, e)]=h(\rho_\sp, u_\sp, e_\sp)-h(\rho_\sm, u_\sm, e_\sm).
\]
By the chosen rescaling, $m:=(\rho u)_\spm=1$, $\rho_-=u_-=1$.
Fixing $\Gamma > 0$, and letting $u_\sp$
vary in the physical range $1\ge u_\sp\ge u_*(\Gamma):= \Gamma/(\Gamma+2)$
(we will show below that this is the physical range),
we use \eqref{eq:rh_mass}--\eqref{eq:rh_energy} to solve for
\[
\rho_\sp, e_\sp\,,\;\text{and}\; e_\sm.
\]
Our assumptions reduce \eqref{eq:rh_mass}--\eqref{eq:rh_energy} to $\rho_\sp= u_\sp^{-1}$ and 
\begin{subequations}
\label{eq:new_rh}
\begin{align}
(u_\sp-1)&=-(p_\sp-p_\sm)=-\Gamma\Big(e_\sp \rho_\sp- e_\sm\Big),
\label{eq:new_rh_momentum}\\
(e_\sp-e_\sm)+\Big(\frac{u_\sp^2}{2}- \frac{1}{2}\Big)
&=-(p_\sp u_\sp -p_\sm) = -\Gamma(e_\sp- e_\sm ).\label{eq:new_rh_energy}
\end{align}
\end{subequations}
Subtracting $\frac{u_\sp+1}{2}$ times \eqref{eq:new_rh_momentum} from
\eqref{eq:new_rh_energy} and rearranging, we obtain
\begin{equation}
\label{eq:R}
e_\sp= e_\sm
\frac{1+ \frac{\Gamma}{2}(1-u_\sp)}
{1-\frac{\Gamma}{2u_\sp}(1-u_\sp)}=\frac{ e_\sm u_\sp}{(\Gamma +2)}
\frac{(\Gamma +2 -\Gamma u_\sp))}{(u_\sp-u_*)},
\end{equation}
$u_*:=\frac{\Gamma}{\Gamma+2}$,
from which we obtain the physicality condition
\begin{equation}
\label{phys}
u_\sp> u_*= \frac{\Gamma}{\Gamma + 2},
\end{equation}
corresponding to positivity of the denominator, with
$\frac{e_\sp}{e_\sm}\to +\infty$ as $u\to u_*$.
Finally, substituting into \eqref{eq:new_rh_momentum} and rearranging,
we obtain a complete description of the endstates in terms of $u_\sp$:
\begin{align}
	e_\sp &= \dfrac{u_\sp(\Gamma + 2 - \Gamma u_\sp)}{2\Gamma(\Gamma+1)}\,, \label{e+}\\
	e_\sm &= 
	\dfrac{(\Gamma + 2)u_\sp - \Gamma}{2\Gamma(\Gamma+1)} \,,\label{e-} \\
\rho_\sp&=1/u_\sp\label{rho+}\,.
\end{align}

We see from this analysis that the strong shock limit corresponds,
for fixed $\Gamma$, to the limit $u_\sp\to u_*$, with all other
parameters functions of $u_\sp$.
In this limit,
\begin{equation}
\label{asymptotics1}
\rho_\sm=1, \quad u_\sm=1,  \quad e_\sm \sim (u_\sp-u_*)\to 0,
\end{equation}
and
\begin{equation}
\label{asymptotics2}
\rho_\sp
\to \frac{\Gamma+ 2}{\Gamma}, \quad
u_\sp \to \frac{\Gamma}{\Gamma +2}, \quad
e_\sp \to \frac{1+ \big(\frac{\Gamma}{\Gamma+1}\big)^2}{2(\Gamma +1)}=:e_\mathrm{max}.
\end{equation}
\emph{In particular, for $\gamma$ bounded away from one
(equivalently, $\Gamma$ bounded from zero)
$u$ remains bounded from zero, so that profile equations 
\eqref{midealode} remain smooth with respect to all variables.}\footnote{ This is essentially different from the situation of the isentropic case, which is singular in the strong shock limit \cite{BHRZ,HLZ,BHLRZ}.}


\subsection{Existence and decay of profiles}
\label{hypcheck}

We recall from \cite{HLyZ} the following lemma describing the nature of the shock profile. 

\begin{lemma}[Humpherys et al.\ \cite{HLyZ}]
\label{profdecay}

For $\Gamma$ bounded and bounded away from the nonphysical limit $\Gamma=0$, $\mu, \mu+ \eta, \nu$ bounded and bounded from zero, and $u_\sp$ bounded away from the characteristic limit $u_\sp=u_\sm=1$, profiles $(\hat u,\hat e)$ of  the rescaled equations \eqref{midealode} exist for all $1\ge u_\sp\ge u_*$, decaying exponentially to their end states $( u_\spm, e_\spm)$ as $x\to \pm \infty$, uniformly in $\Gamma, u_\sp, \mu, \eta, \nu$.
\end{lemma}

\begin{proof}
As observed in \cite{HLyZ}, the argument of Gilbarg \cite{Gi}  for finite-amplitude shocks, $1\ge u_\sp>u_*$, applies also in the limiting case $u_\sp=u_*$, to yield existence on the whole parameter range with density $\hat \rho$ bounded uniformly above and below.  Uniform exponential decay then follows from the corresponding result established in Lagrangian coordinates in \cite{HLyZ}, together with the fact that the Lagrangian spatial variable $\tilde x$ is related to the Eulerian variable $x$ by
$d\tilde x/dx= \hat \rho^{-1}$, hence $|\tilde x|/C\le |x|\le C|\tilde x|$ for uniform $C>0$.  
\end{proof}


\section{Linearized eigenvalue equations}
\label{lineig}

For compactness of notation, set
\begin{equation}\label{mus}
\tilde \mu=(2\mu + \eta), \quad \tilde \eta = (\mu + \eta),
\quad \nu=\kappa/c_v, \quad \hat p = \Gamma \hat\rho\hat e.
\end{equation}
Specializing to the ideal gas case, linearizing about the stationary solution $(\hat \rho, \hat u, \hat v, \hat e)$ and taking the Fourier transform in the $x_2$ direction, we obtain (recalling that $\hat v\equiv 0$) the generalized eigenvalue equations
\begin{subequations}
\label{linearized}
\begin{align}
\lambda \rho
+(\hat \rho u+\hat u\rho)'
+\mi\xi \hat \rho v 
&=0\,, \\
\lambda (\hat \rho u + \hat u \rho)
+ \Big(2u + \hat u^2 \rho+\Gamma (\hat e\rho + \hat \rho e)\Big)'
& =\tilde\mu u''-\xi^2\mu u+\mi\xi(\tilde\eta v'-v)\,, \\
\lambda \hat \rho v +v' 
+ \mi\xi\Gamma (\hat e\rho + \hat \rho e)
& =\mu v''-\xi^2\tilde \mu v+\mi\xi\tilde \eta u',
\end{align}
\begin{multline}
\lambda\left(\rho(\hat e+\hat u^2/2)+u+\hat\rho e\right)
+\left(\gamma(\rho\hat e\hat u+ e+\hat\rho \hat e u)+\frac{1}{2}\rho\hat u^3+\frac{3}{2}\hat u u\right)'\\
+ \mi\xi\left(\gamma\hat\rho\hat ev+\frac{1}{2}v\hat u\right) 
=\tilde\mu(\hat uu'+u\hat u_{x_1})'+\mi\xi\eta(\hat uv)'-\xi^2\mu\hat u u\\
+\mi\xi\mu\hat u v'+\mi\xi\eta v\hat u_{x_1} + \nu e''-\xi^2\nu e.
\end{multline}
\end{subequations}
where $\xi\in \mathbb{R}$ is the Fourier frequency in the $x_2$ direction.

Defining flux variables
\begin{subequations}\label{flux}
\begin{align}
w&:= -\hat \rho u- \hat u\rho,\\
x&:= \tilde \mu u' - (2u + \hat u^2 \rho) -\Gamma (\hat e\rho + \hat \rho e) +\mi\xi \tilde \eta v,\\
y&:=  \mu v' - v +\mi\xi\tilde \eta u,\\
z&:= \tilde \mu (u\hat u_{x_1}+\hat uu') + \nu e' -\gamma( e + \hat u \hat e \rho + \hat e \hat \rho u) - \left(\frac{3\hat u}{2}u+ \frac{\hat u^3}{2}\rho \right) +\mi\xi\tilde \eta \hat uv
\end{align}
\end{subequations}
and modifying by
\begin{subequations}\label{tflux}
\begin{align}
\tilde x&:= x-\hat u w = \tilde \mu u'
- u  -\Gamma (\hat e\rho + \hat  \rho e) +\mi\xi \tilde \eta v,
\\
\tilde z&:= z- \hat u\tilde x -\hat E w =
\nu e' + \tilde \mu \hat u_{x_1} u
-e -\Gamma\hat\rho\hat eu,
\end{align}
\end{subequations}
we may write \eqref{linearized} as a first-order system
\begin{subequations}\label{tfirst}
\begin{align}
w' &= -\lambda \hat\rho w - \lambda \hat\rho^2 u + i\xi\hat\rho v,\\
\tilde x' &= -\hat u_{x_1} w + (\lambda \hat \rho  +\mu \xi^2) u,\\
y' &= -i \xi \hat p w - i \xi \hat p \hat\rho u + (\lambda\hat\rho + \xi^2\tilde\mu) v + i\xi\Gamma\hat\rho e,\\
\tilde z' &= -\hat e_{x_1} w - \hat u_{x_1} \tilde x + \mi\xi f(\hat U,\hat U_{x_1}) v + (\lambda\hat\rho + \xi^2\nu) e,\\
\tilde\mu u' &= -\hat p w + \tilde x + (1-\hat p \hat\rho) u - i\xi\tilde\eta v + \Gamma\hat\rho e,\\
\mu v' &= y - i\xi\tilde\eta u + v,\\
\nu e' &= \tilde z + g(\hat U,\hat U_{x_1}) u + e,
\end{align}
\end{subequations}
where
\begin{subequations}\label{eq:matrix_funcs}
\begin{align} \label{eq:g_eq}
f(\hat U,\hat U_{x_1})&:= \hat p + (\mu-\eta)\hat u_{x_1} =  \Gamma\hat\rho\hat e+(\mu-\eta)\hat u_{x_1}, \\
\intertext{and}
g(\hat U,\hat U_{x_1}) &:=\hat p - \tilde\mu\hat u_{x_1} = \Gamma\hat\rho\hat e-\tilde\mu\hat u_{x_1}.
\end{align}
\end{subequations}

\section{Evans function formulation}
\label{evansform}

Defining $\vec{W}:=(w,\tilde x,y,\tilde z ,u,v,e)^T$, using $\hat u^{-1}\equiv \hat \rho$, and rearranging, we may express \eqref{tfirst} in the basic form
\begin{equation}
\label{firstorder}
\vec{W}' = \mat{A}(x_1;\lambda, \xi) \vec{W},\quad '=\frac{\dif}{\dif x_1},
\end{equation}
with
\begin{equation}
\label{evans_ode}
\mat{A}(x_1;\lambda,\xi):= \\
\begin{pmatrix}
-\lambda\hat \rho  & 0 & 0&  0 & 
-\lambda\hat \rho^2  & \mi\xi \hat \rho  & 0\\
-\hat u_{x_1} & 0&  0& 0  & \lambda \hat \rho +\mu \xi^2 & 0 & 0\\
-\mi\xi\hat p & 0 & 0 & 0 & -\mi\xi\hat p \hat\rho & \lambda\hat\rho+\tilde\mu\xi^2 & \mi\xi\Gamma\hat\rho \\
- \hat e_{x_1} & -\hat u_{x_1} & 0&  0 & 
0& \mi\xi f(\hat U,\hat U_{x_1}) & \lambda\hat\rho+\xi^2\nu\\
-\tilde\mu^{-1}\hat p & \tilde\mu^{-1} & 0 & 0 & \tilde\mu^{-1}(1-\hat p \hat\rho) & -\mi\xi\tilde\mu^{-1}\tilde\eta & \tilde\mu^{-1}\Gamma\hat\rho \\
0 & 0 & \mu^{-1} & 0 & -\mi\xi\mu^{-1}\tilde\eta & \mu^{-1} & 0 \\ 
0 & 0 & 0 & \nu^{-1} & \nu^{-1} g(\hat U,
\hat U_{x_1}) & 0 & \nu^{-1}
\end{pmatrix}.
\end{equation}

\subsection{Balanced flux form}\label{balanced}

Next, loosely following \cite{PZ}, we reformulate the system of eigenvalue ODEs in ``balanced flux'' form. To do so, we rescale the flux variables via
\begin{equation}
\label{eq:flux_rescale}
\check r\check w=w,\;
\check r\check x=\tilde x,\;
\check r\check y=y,\;\text{and}\; 
\check r\check z=\tilde z\,,
\end{equation}
where $\check r$ is given by 
\begin{equation}
\label{eq:scale_factor}
\check r(\lambda,\xi):=\sqrt{|\xi|^2 + |\lambda|^2}.
\end{equation}
Then, writing $\check{\vec{W}}:=(\check w,\check x,\check y,\check z,u,v,e)^T$ and $\check r\check \lambda=\lambda,\;r\check\xi=\xi$, we may re-express \eqref{firstorder} in the form  
\begin{equation}\label{bfirstorder}
\check{\vec{W}}'=\check{\mat{A}}(x_1;\check \lambda, \check \xi, \check r)\check{\vec{W}},
\end{equation}
where
\begin{multline} \label{bevans_ode}
\check{\mat{A}}(x_1;\check \lambda, \check \xi, \check r) := \\
\begin{pmatrix}
-\check r\check\lambda\hat \rho  & 0 & 0&  0 & 
-\check\lambda\hat \rho^2  & \mi\check\xi \hat \rho  & 0\\
-\hat u_{x_1} & 0&  0& 0  & \check\lambda \hat \rho +\check r\mu \check\xi^2 & 0 & 0\\
-\mi\check  r\check\xi\hat p & 0 & 0 & 0 & -\mi\check\xi\hat p\hat\rho & \check\lambda\hat\rho+\check r\tilde\mu\check\xi^2 & \mi \check\xi\Gamma\hat\rho \\
- \hat e_{x_1} & -\hat u_{x_1} & 0&  0 & 
0& \mi\check\xi f(\hat U,\hat U_{x_1}) & \check\lambda\hat\rho+\check{r}\check\xi^2\nu\\
-\tilde\mu^{-1}\check r\hat p & \tilde\mu^{-1}\check r & 0 & 0 & \tilde\mu^{-1}(1-\hat p\hat\rho) & -\mi \check r\check\xi\tilde\mu^{-1}\tilde\eta & \tilde\mu^{-1}\Gamma\hat\rho \\
0 & 0 & \mu^{-1}\check r & 0 & -\mi \check r\check\xi\mu^{-1}\tilde\eta & \mu^{-1} & 0 \\ 
0 & 0 & 0 & \nu^{-1}\check r &  \nu^{-1} g(\hat U,\hat U_{x_1}) & 0 & \nu^{-1}
\end{pmatrix}.
\end{multline}

In the one-dimensional case $\xi=0$, this reduces to the integrated equations of \cite{HLyZ}.  Thus, the balanced flux form may be viewed as a generalization to the multidimensional case of the integrated equations commonly used in one dimension.  Like the integrated equations, equations \eqref{bevans_ode} have the advantage of removing the zero eigenvalue at the origin of the original system \eqref{evans_ode}, hence are preferable for numerical stability computations \cite{BHLyZ2}.


\subsection{Consistent splitting}\label{splitting}
Denote by 
\[
\check{\mat{A}}_\spm(\check \lambda, \check \xi, \check r):= \lim_{x_1\to \pm \infty} \check{\mat{A}}(x_1;\check \lambda, \check \xi, \check r)
\]
the limiting coefficient matrices at $x_1=\pm \infty$.  (These limits exist by exponential convergence of profiles, Lemma \ref{profdecay}.) Denote by $S_\spm=S_\spm(\check \lambda, \check \xi,\check r) $ and  $U_\spm=U_\spm(\check \lambda, \check \xi,\check r)$ the stable and unstable  subspaces of $\check{\mat{A}}_\spm$.

\begin{definition}
Following \cite{AGJ}, we say that \eqref{firstorder} exhibits consistent splitting on a given $(\check\lambda, \check \xi, r)$-domain if $\check{\mat{A}}_\spm$ are hyperbolic, with $\dim S_\sp$ and $\dim U_\sm$ constant and summing to the dimension of the full space (in this case $7$).
\end{definition}

By continuous dependence of $\check{\mat{A}}$ on $\check \lambda, \check \xi, \check r$ and  standard matrix perturbation theory, $S_\sp$ and $U_\sm$ are continuous on any domain for which consistent splitting holds.

\begin{lemma}
\label{consistent}
For all $\Gamma >0$, $\mu, \nu>0$, $\eta \ge 0$, $1 \ge u_\sp>u_*$, \eqref{firstorder}--\eqref{evans_ode} exhibit consistent splitting on $\{ \Re \check \lambda \ge 0, \check \xi \in \mathbb{R}, \, r> 0\}$, with $\dim S_+=4$ and $\dim U_-=3$.  Moreover, for $1>u_\sp\ge u_*$, subspaces $S_\sp$ and $U_\sm$, along with their associated spectral projections, extend continuously to
\[
\{ \Re \check \lambda \ge 0, \check \xi \in \mathbb{R}, \, \check r \ge 0\}\,,
\]
with respect to all arguments $\check \lambda, \check \xi, \check r$ and parameters $\Gamma, \mu, \eta, \nu, u_\sp$.
\end{lemma}

\begin{proof}
For finite-amplitude shocks, $u_\sp>u_*$, this follows by
the general results of \cite{GMWZ}, which apply
in particular to the compressible Navier--Stokes equations
with either standard or van der Waals equation of state.
Though carried out for form \eqref{evans_ode},
the analysis of \cite{GMWZ} applies equally to the form \eqref{bfirstorder}.
(Away from $\check r=0$, this follows immediately by invariance of eigenvalues
under a nonsingular change of coordinates; near $\check r=0$, it follows by 
repeating the analysis of \cite{GMWZ} in the new coordinates.)

As observed in \cite{HLyZ}, the same argument applies to $\check{\mat{A}}_\sp$
also in the limiting case $u_\sp=u_*$, since the hypotheses of \cite{GMWZ}
remain satisfied.
Thus, as in the one-dimensional analysis of \cite{HLyZ}, 
it remains only to verify the claim for $\check{\mat{A}}_\sm$ in the 
limiting case $u_\sp\to u_*$, for which $e_\sm, u_\sm, \rho_\sm$ converge to 
$0,1,1$.
For $\check r>0$, this is most conveniently verified by working in
the original coordinates of \eqref{evans_ode}
and recalling the correspondence between triples
$(\lambda, \alpha, \mi \xi)$ with $\alpha$ an eigenvalue
of $\mat{A}_\sm(\lambda, \xi)$ and solutions $(\lambda, \mi\zeta_1,\mi\zeta_2)$
of the dispersion relation 
\[
\lambda \in \sigma \Big(-\sum_{j=1}^2 \mi\zeta_j a_j -\sum_{j,k=1}^2 \zeta_j\zeta_k b_{jk}\Big)
\]
of the symbol of the limiting eigenvalue equation written as a second-order system \cite{Z3,Z5}. 
In coordinates $\mathcal{U}:=(\rho, u,v,e)^T$, we find, consulting equations (1.45)--(1.47) of \cite{Z5}, that, for $\rho=u=1$, $v=e=0$,
\begin{equation}
\label{symbol}
\begin{aligned}
\sum_{j=1}^2 \mi\zeta_j a_j  &=
\begin{pmatrix}
\mi\zeta_1 &  \mi\vec{\zeta}^T  & 0\\
0&  \mi\zeta_1 \mat{I}_2  & \mi R \vec{\zeta} \\
0 & 0 &  \mi\zeta_1 
\end{pmatrix}\\
\sum_{j,k=1}^2 \zeta_j\zeta_k b_{jk}&=
\begin{pmatrix}
0 & 0 & 0\\
0 & \mu \mat{I}_2 + (\mu+\eta)\vec{\zeta} \vec{\zeta}^T & 0\\
0 & 0 &  \nu |\vec{\zeta}|^2
\end{pmatrix},
\end{aligned}
\end{equation}
from which, by upper block-triangular form, we may read off
the solutions
\begin{equation}\label{special}
\lambda\equiv  -\mi\zeta_1, 
\end{equation}
\[
\lambda= -\mi\zeta_1 -\nu |\vec{\zeta}|^2,
\]
and 
\[
\lambda \in -\mi\zeta_1 - \sigma\Big( \mu \mat{I}_2 + (\mu+\eta)\vec{\zeta}\vec{\zeta}^T \Big).
\]
Evidently, the only pure imaginary $\alpha=\mi\zeta_1$ corresponding to $\Re \lambda \ge 0$ is the single root $\alpha \equiv -\lambda$ for $\lambda$ pure imaginary, associated with the special solution \eqref{special}, arising through the uncoupled hyperbolic mode represented by the first row of symbol $-\sum_j \mi\zeta a_j-\sum_{jk}\zeta_j\zeta_k b_{jk}$, with all other roots $\alpha$ strictly stable (negative real part) or unstable (positive real part) on $\xi\in \mathbb{R}$, $\Re \lambda\ge 0$.  The root $-\lambda$ is strictly stable for $\Re \lambda>0$, from which we see that consistent splitting holds for $\xi\in \mathbb{R}$ and $\Re \lambda>0$.

Moreover, observe that the total eigenspace associated with
$-\lambda$ along with the remaining, strictly stable eigenvalues
of $\mat{A}_\sm$, corresponding for $\Re \lambda$ to the stable subspace $S_\sm$, 
converges at the boundary $\Re \lambda \to 0$ to the center-stable
subspace $CS_\sm$ of $\mat{A}_\sm$, which remains strictly spectrally separated 
from the unstable subspace $U_-$.
Redefining $S_\sm$ as $CS_\sm$ on the whole of $\Re \lambda \ge 0$, we thus
obtain the claimed continuity by standard matrix perturbation theory,
for all $\check r>0$.

Finally, we consider the case $\check r\to 0$.
From \eqref{symbol}, we see that the shock remains noncharacteristic
at $x=-\infty$ as $u_\sp\to u_*$; moreover, all modes are {\it incoming}
to the shock (positive hyperbolic characteristics $c_j\equiv \zeta_1$),
corresponding to its nature as a ``one-shock'', 
or left-moving wave relative to the ambient fluid flow $\hat u>0$.
Recalling the low-frequency analysis of \cite{Z3,GMWZ}, we
may thus conclude that the $4$ center-stable modes at $x=-\infty$
exactly correspond to the $4$ ``slow'', or ``hyperbolic'' modes
whose eigenvalues converge to zero as $r\to 0$, with the other
$3$ unstable modes corresponding to  ``fast'', or ``parabolic''
modes whose eigenvalues are bounded away from zero as $\check r\to 0$.
From this spectral separation, we obtain the claimed continuity
by standard matrix perturbation theory.
\end{proof}

\begin{remark}\label{nokaw}
From \eqref{symbol}, we see that system \eqref{eq:mass}--\eqref{eq:energy} leaves the class of Kawashima-type symmetric hyperbolic--parabolic systems \cite{K} in the nonphysical limit $e\to 0$, losing the properties of {\rm symmetrizability} and {\rm genuine coupling} (see \cite{Z5,Z3,KSh} for definitions of these terms); indeed, the first-order part of the symbol is no longer hyperbolic, featuring a Jordan block.  Thus, the symbolic analyses of \cite{GMWZ, Z5,Z3} do not apply in this case.
\end{remark}


\subsection{Construction of the Evans function}
\label{construction}
We now construct the Evans function associated with 
\eqref{bfirstorder}, following the approach of \cite{MZ1,PZ}.

\begin{lemma}
\label{cbasis}
There exist bases 
\[
V^-=(V_1^-, V_2^-, V_3^-)(\check \lambda, \check \xi, \check r), 
\quad
V^+=(V_4^+, V_5^+,V_6^+,V_7^+)(\check \lambda, \check \xi, \check r)
\]
 of $U_\sm(\check \lambda, \check \xi, \check r)$ and $S_\sp(\check \lambda, \check \xi, \check r)$, 
extending continuously in $\check \lambda, \check \xi, \check r$
and $\Gamma, \nu, \mu, \eta, u_\sp$ to
$\{\Re \lambda \ge 0, \, \xi\in \mathbb{R} \}$ for $\Gamma>0$, $\nu, \mu>0$, 
$\eta\ge 0$, and $1>u_\sp\ge u_*$.
\end{lemma}

\begin{proof}
Standard matrix perturbation theory; see, e.g., \cite{Kato}.
\end{proof}

\begin{lemma}\label{basis}
There exist bases 
\[
\check W^\sm=(\check W_1^\sm, \check W_2^\sm, \check W_3^\sm)(\check \lambda, \check \xi, \check r), \quad \check W^\sp=(\check W_4^\sp, \check W_5^\sp,\check W_6^\sp,\check W_7^\sp)(\check \lambda, \check \xi, \check r)
\]
of the unstable manifold at $x_1=-\infty$ and the stable manifold at $x_1=+\infty$
of \eqref{bfirstorder} asymptotic to $\me^{\check{\mat{A}}_\sm x_1}V^\sm$ and $\me^{\check{\mat{A}}_\sp x_1}V^\sp$, respectively,
as $x\to \mp\infty$,
extending continuously in $\check \lambda, \check \xi, \check r$
and $\Gamma, \nu, \mu, \eta, u_\sp$ to
$\{\Re \check \lambda \ge 0, \, \check \xi\in \mathbb{R}, \, \check r\ge 0 \}$ 
for $\Gamma>0$, $\mu>|\eta|\geq0$, $\nu> 0$, and $1>u_\sp\ge u_*$.
\end{lemma}

\begin{proof}
This follows
by uniform exponential convergence of $\check{\mat{A}}$ to $\check{\mat{B}}_\spm$ as $x_1\to \pm \infty$, Lemma \ref{profdecay},
using the conjugation lemma of \cite{MeZ}
to construct solutions of the variable-coefficient system \eqref{bfirstorder} 
in terms of solutions of its constant-coefficient limits at $x=\pm \infty$.
\end{proof}

\begin{definition}\label{evansdef}
The Evans function associated with \eqref{bfirstorder}
is defined as
\begin{equation}\label{evanseq}
D(\lambda,\xi)=\check D(\check \lambda,\check\xi,\check r):=\det(\check W^\sp, \check W^\sm)|_{x_1=0}.
\end{equation}
\end{definition}

\begin{proposition}\label{evanscont}
The Evans function $\check D$ is continuous
in $\check \lambda, \check \xi, \check r$ and 
$\Gamma, \nu, \mu, \eta, u_\sp$ on 
$\{\Re \check \lambda \ge 0, \, \check \xi\in \mathbb{R}, \, \check r\ge 0\}$
for $\Gamma >0$, $\mu>|\eta|\ge 0$, $\nu> 0$, and $1>u_\sp\ge u_*$.  
Moreover, on $\{\Re \check \lambda \ge 0, \, \check \xi\in \mathbb{R}, \, \check r\ge 0\}
\setminus \{\check r=0\}$, for $u_\sp>u_*$,
its zeros correspond with eigenvalues of the 
integrated linearized operator about the shock layer, i.e., with
solutions of \eqref{linearized} decaying at $x_1=\pm \infty$.
\end{proposition}

\begin{proof}
The first statement follows by Lemma \ref{basis}.
The second follows by the asymptotic description 
$W^\spm_j \sim \me^{\check{\mat{A}}_\spm x_1}V_j^\spm$ together 
with the consistent splitting property, which implies decay at $x_1=\pm \infty$
of the constant-coefficient solutions $\me^{\check{\mat{A}}_\spm x_1}V_j^\spm$.
\end{proof}

\begin{remark}
\label{stronglimit}
Proposition \ref{evanscont} includes in particular the key information that the Evans function converges in the strong shock limit $u_\sp\to u_*$ to the Evans function for the limiting system at $u_\sp=u_*$, uniformly on compact subsets of $\{\Re \lambda \ge 0\}$.
\end{remark}

\begin{remark}
\label{nonunique}
The Evans function as so far described is highly non-unique, due to the many choices for continuous prolongation of  subspaces $U_\sm(\check \lambda, \check \xi, \check r)$ and  $S_\sp (\check \lambda, \check \xi, \check r)$.  This non-uniqueness will be removed in Sections \ref{numerical}--\ref{sec:intermediate} by specification of a numerical continuation algorithm.
\end{remark}

\begin{remark}
\label{nonvanishing}
It is readily shown by a low-frequency calculation similar to those of \cite{ZS,Z3,Z5,GMWZ} that $D(\cdot, \cdot, 0)$ does not vanish at $\check r=0$ for $\Re \check \lambda \ge 0$, $\check \xi\in \mathbb{R}$, so long as the associated discontinuous shock satisfies the hyperbolic stability condition, as has been verified for ideal gas equation of state by Erpenbeck, Majda, and others \cite{Er,Maj1,Maj2,Maj3,Z5}.  Thus, the zero at the origin is indeed removed by the choice of balanced flux coordinates, as we verify by direct numerical computation in Section \ref{numerical}.  For further details, see \cite{BHLyZ2}.
\end{remark}


\section{High-frequency bounds}
\label{sec:hfb}
Adapting to the multidimensional case the approach of Humpherys et al.\ \cite{HLyZ}, we start by identifying a bounded set in $(\lambda,\xi)$-space outside of which there are no unstable eigenvalues. This will reduce our computational  domain to a compact set in frequency space, making the problem feasible for numerical Evans function computation.  Following \cite{GMWZ}, we introduce the parabolic coordinates
\begin{equation}
\label{prescale}
\lambda=\breve r\breve \lambda, \quad \xi=\breve r^{1/2}\breve \xi,
\end{equation}
\begin{equation}
\label{pradius}
\breve r:=|\xi|^2+|\lambda|,\quad 
|\breve \xi|^2 + |\breve \lambda|=1.
\end{equation}
(Not to be confused with the standard polar coordinates of Section \ref{balanced}.)  To simplify calculations, we restrict to the case $\breve \lambda$ {\it pure imaginary} that we shall ultimately need.\footnote{General $\breve \lambda$ may be treated as in \cite{GMWZ} by splitting into the two cases $\Re \breve \lambda \le \breve r^{-1/2}/C$ and $\Re \breve \lambda \ge \breve r^{-1/2}/C$,  or equivalently $\Re \lambda \ll |\xi|$ and $\Re \lambda \sim |\xi|$, the former treated as in the pure imaginary case and the latter by a separate, somewhat simpler argument.}  By symmetry of the eigenvalue equations under complex conjugation, it is sufficient to take $\breve \lambda=i\breve \tau$, $\breve \tau\ge 0$, or $\breve \tau=1-\breve \xi^2$, $-1\le \breve \xi \le 1$.

\subsection{Symbolic preparations}
For this part of the analysis, it is most convenient to go back to the basic form \eqref{firstorder}--\eqref{evans_ode}. Introducing parabolic coordinates \eqref{prescale}, we expand $\mat{A}$ in powers of $\breve r^{1/2}$ as $\mat{A}=\breve r \mat{A}_1 + \breve r^{1/2}\mat{A}_{1/2} + \mat{A}_0$, where 
\begin{equation}
\label{A1}
\mat{A}_1 =
\left(\begin{array}{c|cccccc}
-\hat \rho \breve \lambda & 
0 & 0 & 0 & -\hat \rho^2 \breve \lambda & 0 & 0 \\ \hline
0& 0& 0 & 0 & \breve \lambda \hat \rho + \mu \breve \xi^2 & 0 & 0 \\
0& 0& 0 & 0 & 0 & \breve \lambda \hat \rho + \tilde \mu \breve \xi^2  & 0\\
0& 0& 0 & 0 & 0 & 0& \breve \lambda \hat \rho + \nu \breve \xi^2 \\
0& 0& 0 & 0 & 0 & 0 & 0\\
0& 0& 0 & 0 & 0 & 0 & 0\\
0& 0& 0 & 0 & 0 & 0 & 0\\
\end{array}\right),
\end{equation}
\begin{equation} \label{A1/2}
\mat{A}_{1/2}  =
\begin{pmatrix}
0& 0& 0 & 0 & 0 & i\breve \xi\hat \rho & 0\\
0& 0& 0 & 0 & 0 & 0 & 0\\
\mi\breve\xi\hat p& 0& 0 & 0 & -\mi\breve \xi\hat\rho\hat p & 0 & \mi\breve \xi \Gamma \hat \rho\\
0& 0& 0 & 0 & 0 &\mi\breve \xi f(\hat U,\hat U_{x_1}) & 0\\
0& 0& 0 & 0 & 0 & -\mi\breve \xi \tilde \eta \tilde \mu^{-1} & 0\\
0& 0& 0 & 0 &-\mi\breve \xi \tilde \eta \mu^{-1} & 0 & 0\\
0& 0& 0 & 0 & 0 & 0 & 0\\
\end{pmatrix},
\end{equation}
and
\begin{equation} \label{A0}
\mat{A}_{0}  =
\begin{pmatrix}
0& 0& 0 & 0 & 0 & 0 & 0\\
-\hat u_{x_1}& 0& 0 & 0 & 0 & 0 & 0\\
0 & 0& 0 & 0 & 0 & 0 & 0\\
-\hat e_{x_1}& -\hat u_{x_1}& 0 & 0 & 0 & 0 & 0\\
-\tilde\mu^{-1}\hat p&  \tilde \mu^{-1}& 0 & 0 &  \tilde \mu^{-1}(1-\hat\rho\hat p) & 0 & \tilde\mu^{-1}\Gamma\hat\rho\\
0& 0&  \mu^{-1} & 0 &  0 & \mu^{-1} & 0\\
0& 0& 0 & \nu^{-1} & \nu^{-1} g(\hat U,\hat U_{x_1}) & 0 & \nu^{-1}\\
\end{pmatrix}.
\end{equation}
Noting that $\mat{A}_1$ is upper block-triangular, with ($1\times 1$)
upper diagonal block $-\hat \rho \breve \lambda$ 
and ($6\times 6$) lower diagonal block 
\[
\alpha:=
\begin{pmatrix}
0& 0 & 0 & \breve \lambda \hat \rho + \mu \breve \xi^2 & 0 & 0 \\
0& 0 & 0 & 0 & \breve \lambda \hat \rho + \tilde \mu \breve \xi^2  & 0\\
0& 0 & 0 & 0 & 0& \breve \lambda \hat \rho + \nu \breve \xi^2 \\
0& 0 & 0 & 0 & 0 & 0\\
0& 0 & 0 & 0 & 0 & 0\\
0& 0 & 0 & 0 & 0 & 0\\
\end{pmatrix},
\]
having all zero eigenvalues, we block-diagonalize by the upper block-diagonal transformation $\vec{W}:=\mat{S}\vec{X}$, where $\mat{S}=\begin{pmatrix}1 & \theta \\ 0 & \mat{I}_6\\ \end{pmatrix}$, 
$ \mat{S}^{-1}:=\begin{pmatrix} 1 & -\theta \\ 0 & \mat{I}_6\\ \end{pmatrix}$, with
$\theta:= -\breve \lambda (0, 0,0,\hat \rho^2, 0, 0) (\breve \lambda\hat\rho \mat{I}_6 + \alpha)^{-1}$.
Since $\alpha^2=0$, we have $(\breve \lambda\hat\rho \mat{I}_6+\alpha)^{-1}= \breve \lambda^{-1}\hat\rho^{-1}(\mat{I}_6 - \breve \lambda^{-1}\breve\rho^{-1} \alpha)$, and that $(0, 0,0,\hat \rho^2, 0, 0)\alpha=0$, and hence we obtain $\theta= (0,0,0, -\hat \rho, 0, 0)$ and
\[
\mat{S}^{-1}\mat{S}'=
\begin{pmatrix}
0 & \theta' \\
0 & 0\\
\end{pmatrix}
=
\left(\begin{array}{c|cccccc}
0 & 0 & 0 & 0 & - \hat \rho_{x_1} & 0 & 0 \\ \hline
0 & 0 & 0 & 0 & 0 & 0 & 0 \\
0 & 0 & 0 & 0 & 0 & 0 & 0 \\
0 & 0 & 0 & 0 & 0 & 0 & 0 \\
0 & 0 & 0 & 0 & 0 & 0 & 0 \\
0 & 0 & 0 & 0 & 0 & 0 & 0 \\
0 & 0 & 0 & 0 & 0 & 0 & 0 \\
\end{array}\right)\,,
\]
and thus $\vec{X}'=\mat{B}\vec{X}$, 
$
\mat{B}= \mat{S}^{-1}\mat{A}\mat{S} - \mat{S}^{-1}\mat{S}'= \breve r \mat{B}_1 + \breve r^{1/2}\mat{B}_{1/2}+ \mat{B}_0,
$
where
\begin{equation} \label{B1}
\mat{B}_1 =
\left(\begin{array}{c|cccccc}
-\hat \rho \breve \lambda & 
0 & 0 & 0 &  0 & 0 & 0 \\ \hline
0& 0& 0 & 0 & \breve \lambda \hat \rho + \mu \breve \xi^2 & 0 & 0 \\
0& 0& 0 & 0 & 0 & \breve \lambda \hat \rho + \tilde \mu \breve \xi^2  & 0\\
0& 0& 0 & 0 & 0 & 0& \breve \lambda \hat \rho + \nu \breve \xi^2 \\
0& 0& 0 & 0 & 0 & 0 & 0\\
0& 0& 0 & 0 & 0 & 0 & 0\\
0& 0& 0 & 0 & 0 & 0 & 0\\
\end{array}\right),
\end{equation}
\begin{equation} \label{B1/2}
\mat{B}_{1/2} =
\begin{pmatrix}
0& 0& 0 & 0 & 0 & \mi\breve \xi\hat\rho(1-\tilde\mu^{-1}\tilde\eta) & 0\\
0& 0& 0 & 0 & 0 & 0 & 0\\
-\mi\breve\xi\hat p& 0& 0 & 0 & 0 & 0 &\mi\breve \xi \Gamma \hat \rho\\
0& 0& 0 & 0 & 0 & \mi\breve \xi f(\hat U,\hat U_{x_1}) & 0\\
0& 0& 0 & 0 & 0 & -\mi\breve \xi \tilde \eta \tilde \mu^{-1} & 0\\
0& 0& 0 & 0 &-\mi\breve \xi \tilde \eta \mu^{-1} & 0 & 0\\
0& 0& 0 & 0 & 0 & 0 & 0\\
\end{pmatrix},
\end{equation}
and
\begin{equation}
\label{B0}
\mat{B}_{0} =
\begin{pmatrix}
-\tilde\mu^{-1}\hat p\hat\rho & \tilde\mu^{-1}\hat\rho & 0 & 0 & \tilde\mu^{-1}\hat\rho + \hat\rho_{x_1}& 0 & \tilde\mu^{-1}\Gamma\hat\rho^2 \\
-\hat u_{x_1} & 0 & 0 & 0& \hat\rho\hat u_{x_1} & 0 & 0 \\
0 & 0 & 0 & 0 & 0 & 0 & 0 \\
-\hat e_{x_1} & -\hat u_{x_1} & 0 & 0 & \hat\rho\hat{e}_{x_1} & 0 & 0 \\
-\tilde\mu^{-1}\hat p& \tilde\mu^{-1}& 0 & 0 & \tilde\mu^{-1} & 0 & \tilde\mu^{-1}\Gamma\hat \rho \\
0 & 0 & \mu^{-1} & 0 & 0 & \mu^{-1} & 0 \\
0 & 0 & 0 & \nu^{-1} & \nu^{-1} g(\hat U,\hat U_{x_1}) & 0 & \nu^{-1}
\end{pmatrix}\,.
\end{equation}
Making the ``balancing transformation'' $\vec{X}=\mat{T}\vec{Y}$,
$\mat{T}=\begin{pmatrix} \mat{I}_4 & 0 \\
0 &  \breve r^{-1/2} \mat{I}_3  \end{pmatrix}$,
we obtain 
$
\vec{Y}'=\mat{C}\vec{Y}\,,
$
$ \mat{C}= \mat{T}^{-1}\mat{B}\mat{T} $,
$\mat{C}=\breve r^{1/2}\mat{C}_{1/2}+\mat{C}_0+\breve r^{-1/2}\mat{C}_{-1/2}$, with
\begin{equation}
\label{cC}
\mat{C}_{1/2}(x_1, \breve\xi,\breve r) = \left(\begin{array}{c|cccccc}
-\hat \rho \breve \lambda \breve r^{1/2} & 0 & 0 & 0 &  0 & 0 & 0 \\ \hline
0& 0& 0 & 0 & \breve \lambda \hat \rho + \mu \breve \xi^2 & 0 & 0 \\
-\mi\breve\xi\hat p& 0& 0 & 0 & 0 & \breve \lambda \hat \rho + \tilde \mu \breve \xi^2  & 0\\
0& 0& 0 & 0 & 0 & 0& \breve \lambda \hat \rho + \nu \breve \xi^2 \\
-\tilde\mu^{-1}\hat p&  \tilde \mu^{-1} & 0 & 0 
& 0 & -\mi\breve \xi \tilde \eta \tilde \mu^{-1} & 0\\
0& 0&  \mu^{-1} & 0 & -\mi\breve \xi \tilde \eta \mu^{-1} & 0 & 0\\
0 & 0& 0 & \nu^{-1} & 0 & 0 & 0\\
\end{array}\right)\,,
\end{equation}
\begin{equation} \label{C0}
\mat{C}_0(x_1, \breve\xi) =
\begin{pmatrix}
-\tilde\mu^{-1}\hat p\hat\rho & \tilde\mu^{-1}\hat\rho & 0 & 0 & 0 & \mi\breve\xi\hat\rho(1-\tilde\mu^{-1}\tilde\eta) & 0 \\
-\hat u_{x_1} & 0 & 0 & 0 & 0 & 0 & 0 \\ 
0 & 0 & 0 & 0 & 0 & 0& \mi\breve\xi\Gamma\hat\rho \\
-\hat e_{x_1} & -\hat u_{x_1} & 0 & 0 &  0 & \mi\breve \xi f(\hat U,\hat U_{x_1}) & 0 \\
0 & 0 & 0 & 0 & \tilde\mu^{-1} & 0 & \tilde\mu^{-1}\Gamma\hat\rho \\
0 & 0 & 0 & 0 & 0 & \mu^{-1} & 0 \\ 
0 & 0 & 0 & 0 & \nu^{-1} g(\hat U,\hat U_{x_1}) & 0 & \nu^{-1}  
\end{pmatrix},
\end{equation}
\begin{equation}
\mat{C}_{-1/2}(x_1, \breve\xi) =
\begin{pmatrix}
0  & 0 & 0 & 0 & \tilde\mu^{-1}\hat\rho+\hat\rho_{x_1} & 0 & \tilde\mu^{-1}\Gamma\hat\rho^2 \\
0 & 0 & 0 & 0 & \hat\rho\hat u_{x_1} & 0 & 0 \\
0&0&0&0&0&0&0\\
0 & 0 & 0 & 0 & \hat\rho\hat{e}_{x_1} & 0 & 0 \\
0&0&0&0&0&0&0\\
0&0&0&0&0&0&0\\ 
0&0&0&0&0&0&0
\end{pmatrix}.
\end{equation}

Making the final transformation $\vec{Y}=\mat{U}\vec{Z}$,
\begin{equation}
\label{psi}
\mat{U}=\begin{pmatrix} 1 & 0 \\
\psi &   \mat{I}_6  \end{pmatrix}, \quad
\psi=-(\hat \rho \breve \lambda \breve r^{1/2}\mat{I}_6 + \beta)^{-1}q,
\end{equation}
where
$q:= ( 0, -\mi\breve\xi\hat p, 0, -\tilde\mu^{-1}\hat p, 0, 0)^T$ and
\begin{equation}\label{beta}
\beta( \breve\lambda,  \breve \xi):=
\begin{pmatrix}
 0& 0 & 0 & \breve \lambda \hat \rho + \mu \breve \xi^2 & 0 & 0 \\
 0& 0 & 0 & 0 & \breve \lambda \hat \rho + \tilde \mu \breve \xi^2  & 0\\
 0& 0 & 0 & 0 & 0& \breve \lambda \hat \rho + \nu \breve \xi^2 \\
  \tilde \mu^{-1}& 0 & 0 & 0 & -\mi\breve \xi \tilde \eta \tilde \mu^{-1} & 0\\
 0&  \mu^{-1} & 0 &  -\mi\breve \xi \tilde \eta \mu^{-1} & 0 & 0\\
 0& 0 & \nu^{-1} & 0 & 0 & 0\\
\end{pmatrix},
\end{equation}
we obtain $\vec{Z}'=\mat{D}\vec{Z}$, with $ \mat{D}= \mat{U}^{-1}\mat{C}\mat{U} - \mat{U}^{-1}\mat{U}' $. 
Note, $\mat{U}^{-1}\mat{U}'=\begin{pmatrix} 0 & 0\\ \psi' & 0\end{pmatrix}$, where 
\beq
\label{psi'}
\psi'=-( \hat \rho \breve \lambda \breve r^{1/2}\mat{I}_6 + \beta)^{-1}q_{x_1}
+(\hat \rho \breve \lambda \breve r^{1/2}\mat{I}_6 + \beta)^{-1}
(\hat \rho_{x_1} \breve \lambda \breve r^{1/2}\mat{I}_6 + \beta_{x_1})
(\hat \rho \breve \lambda \breve r^{1/2}\mat{I}_6 + \beta)^{-1} q.
\eeq
Thus, $\mat{D}=\breve r^{1/2}\mat{D}_{1/2}+\mat{D}_0+\breve r^{-1/2}\mat{D}_{-1/2}$, where
\begin{equation}
\label{cD}
\mat{D}_{1/2}(x_1, \breve\xi,\breve r) = \left(\begin{array}{c|cccccc}
-\hat \rho \breve \lambda \breve r^{1/2} & 0 & 0 & 0 &  0 & 0 & 0 \\ \hline
0& 0& 0 & 0 & \breve \lambda \hat \rho + \mu \breve \xi^2 & 0 & 0 \\
0& 0& 0 & 0 & 0 & \breve \lambda \hat \rho + \tilde \mu \breve \xi^2 & 0\\
0& 0& 0 & 0 & 0 & 0& \breve \lambda \hat \rho + \nu \breve \xi^2 \\
0&  \tilde \mu^{-1}& 0 & 0 & 0 & -\mi\breve \xi \tilde \eta \tilde \mu^{-1} & 0\\
0& 0&  \mu^{-1} & 0 &  -\mi\breve \xi \tilde \eta \mu^{-1} & 0 & 0\\
0& 0& 0 & \nu^{-1} & 0 & 0 & 0\\
\end{array}\right)\,,\\
\end{equation}
and 
\begin{align} \label{D0}
\mat{D}_0(x_1, \breve\xi,\breve r) & =
\mat{U}^{-1}\mat{C}_0(x_1,\breve \xi)\mat{U} -\mat{U}^{-1}\mat{U}'\,, \\
\mat{D}_{-1/2}(x_1, \breve\xi,\breve r) & =
\mat{U}^{-1}\mat{C}_{-1/2}(x_1,\breve \xi)\mat{U}\,.
\end{align}

As established in Lemma \ref{betalem} below, the eigenvalues of $\beta$ have real parts uniformly bounded from zero, hence, recalling that $\breve \lambda=\mi\breve \tau$ is pure imaginary, we also find that the matrix $( \hat \rho \breve \lambda \breve r^{1/2}I_6 + \beta)^{-1}$ is uniformly bounded, hence $\psi$ and $\psi'$ are uniformly bounded.  In practice we will estimate these numerically.  Likewise, we see that the  center-stable subspace $\Sigma_\sm$ and the unstable subspace $\Sigma_\sp$ of $\mat{D}_{1/2}$ must have a uniform spectral gap bounded from zero, given by the minimum real part of the eigenvalues associated with the unstable subspace of $\beta$.


\subsection{Numerical implementation}

With these preparations, we are now ready to obtain bounds using the ``tracking'' procedure of Appendix \ref{tracking}.  We describe the plan of attack: By block diagonalizing $\mat{D}_{1/2}$ and noting an $\breve r^{1/2}$-dependent separation between the numerical ranges of the stable and unstable blocks, we can use Lemma \ref{tracklem} to obtain high-frequency bounds in $\breve r$ for each fixed $\breve \xi$.


\subsubsection{Tracking Setup}  Since the vector space $\mathbb{C}^6$ an be split into complementary $\beta$-invariant subspaces corresponding to its stable and unstable subspaces, we can form a smooth block matrix $R$ of column vectors (see Appendix \ref{blockdiag}) corresponding to bases of those complementary subspaces together with the $(1,1)$ block of $\mat{D}_{1/2}$, which contributes a purely imaginary eigenvalue to the spectrum of $\mat{D}_{1/2}$ (since we are restricting to the case that $\breve\lambda=\mi\breve\tau$).  By associating this center space with $\mat{N}_-$ and setting $L:= R^{-1}$, we have the block form
\begin{equation}
\label{eq:block-form}
L \, \mat{D}_{1/2} \, R = \mat{N} := \begin{pmatrix} \mat{N}_- & 0\\ 0 & \mat{N}_+\end{pmatrix},
\end{equation}
where $\dim \mat{N}_- = 4$ and $\dim \mat{N}_+ = 3$.

Furthermore, by setting $\vec{Z} = R \vec{V}$, we have that $\vec{V}' = (L D R - L R') \vec{V}$, where
\[
L D R - L R' = \breve r^{1/2} \mat{N} + \underbrace{L \, \mat{U}^{-1} (\mat{C}_0 + \breve r^{-1/2} \mat{C}_{-1/2}) \mat{U} \, R - L\, \mat{U}^{-1}\mat{U}'\,R - L R'.}_{\displaystyle\Theta(\breve r,\breve \xi)}
\]
By separating out the asymptotic rates, we can write
\[
\Theta(\breve r,\breve \xi) = \Theta^{(0)}(\breve r,\breve \xi) + \breve r^{-1/2}\Theta^{(-1/2)}(\breve r,\breve \xi),
\]
where
\begin{align*}
\Theta^{(0)}(\breve r,\breve \xi) &= L \, \mat{U}^{-1} \mat{C}_0 \mat{U} \, R - L\, \mat{U}^{-1}\mat{U}'\,R - L R' \\
\Theta^{(-1/2)}(\breve r,\breve \xi) &= L \, \mat{U}^{-1} \mat{C}_{-1/2} \mat{U} \, R.
\end{align*}
Following the block structure of $\mat{N}$, each $\Theta^{(j)}(\breve r,\breve \xi)$ can be further written as
\begin{equation}
\label{theta-bits}
\Theta^{(j)}(\breve r,\breve \xi) = \begin{pmatrix} \Theta^{(j)}_{--}(\breve r,\breve \xi) &  \Theta^{(j)}_{-+}(\breve r,\breve \xi) \\  \Theta^{(j)}_{+-}(\breve r,\breve \xi) &  \Theta^{(j)}_{++}(\breve r,\breve \xi)\end{pmatrix}.
\end{equation}
The presence of the $\breve r^{1/2}$ coefficient in front of $\mat{N}$ results in a growing numerical gap $\delta=\breve \delta \breve r^{1/2}$ as $\breve r$ gets larger.  As we show below, each of the blocks \eqref{theta-bits} are uniformly bounded in $\breve r$ for each fixed $\breve\xi\in[0,1]$, and thus we are able to take $\breve r$ large enough that tracking condition \eqref{rootcond} given in Lemma \ref{tracklem} is satisfied and high-frequency bounds are established.

\subsubsection{Uniform Tracking Bounds}  Assume that $\breve\xi\in[0,1]$ is fixed.  The tracking condition \eqref{rootcond} given in Lemma \ref{tracklem} takes the form
\begin{equation}
\label{eq:track1}
\breve r^{1/2}>\frac{\norm{\Theta^{(0)}_{\sm\sm}}+\norm{\Theta^{(0)}_{\sp\sp}}+\breve r^{-1/2}(\norm{\Theta^{(-1/2)}_{\sm\sm}}+\norm{\Theta^{(-1/2)}_{\sp\sp}})+2\sqrt{H(\breve r,\breve \xi)}}{\breve\delta},
\end{equation}
where
\[
H(\breve r,\breve \xi)=(\norm{\Theta^{(0)}_{\sp\sm}}+\breve r^{-1/2}\norm{\Theta^{(-1/2)}_{\sp\sm}})(\norm{\Theta^{(0)}_{\sm\sp}}+\breve r^{-1/2}\norm{\Theta^{(-1/2)}_{\sp\sm}}),
\]
and separation between the numerical ranges of the blocks as described in \eqref{numrange} reduces to
\[
\breve\delta = \min \sigma(\Re \mat{N_+})-\max \sigma(\Re \mat{N_-}).
\]
We can show that the right-hand size of \eqref{eq:track1} is uniformly bounded in $\breve r$ while the left-hand size grows without bound.  Thus, for sufficiently large $\breve r$, the inequality will be satisfied.

By replacing the right-hand side of \eqref{eq:track1} with uniform bounds, we can compute a value $\breve r^*$ which guarantees that the tracking condition \eqref{eq:track1} is satisfied whenever $\breve r \geq \breve r^*$.  We find $\breve r^*$ by setting $y=\sqrt{\breve r^*}$ and solving
\begin{equation}
\label{quintic-tracking}
y=\frac{a+b/y+ 2\sqrt{(c+d/y)(e+f/y)}}{\breve\delta},
\end{equation}
where
\begin{equation}
\label{positive-coefficients}
\begin{aligned}
a &:= \sup_{\breve r} \norm{\Theta^{(0)}_{\sm\sm}(\breve r,\breve\xi)} + \norm{\Theta^{(0)}_{\sp\sp}(\breve r,\breve\xi)}, \\
b &:=\sup_{\breve r} \norm{\Theta^{(-1/2)}_{\sm\sm}(\breve r,\breve\xi)} + \norm{\Theta^{(-1/2)}_{\sp\sp}(\breve r,\breve\xi)}, \\
c &:=\sup_{\breve r} \norm{\Theta^{(0)}_{\sm\sp}(\breve r,\breve\xi)}, \\
d &:=\sup_{\breve r} \norm{\Theta^{(0)}_{\sp\sm}(\breve r,\breve\xi)}, \\
e &:=\sup_{\breve r} \norm{\Theta^{(-1/2)}_{\sm\sp}(\breve r,\breve\xi)}, \\
f &:=\sup_{\breve r} \norm{\Theta^{(-1/2)}_{\sp\sm}(\breve r,\breve\xi)}.
\end{aligned}
\end{equation}
Further simplifying \eqref{quintic-tracking} gives the quartic polynomial equation 
\begin{equation}
\label{eq:quart2}
\breve\delta^2 y^4-2a\breve\delta y^3+(a^2-b\breve\delta-ce)y^2+(2ab-de-cf)y+(b^2-df)=0.
\end{equation}
Thus, to obtain our uniform tracking bound $\breve r^* := \breve r^*(\breve\xi)$, we simply find the largest real root of \eqref{eq:quart2} and square it.  Therefore we have established the following theorem:
\begin{theorem}[Uniform Tracking Bounds]
\label{thm:strong-bounds}
Let $\breve\xi\in[0,1]$ be fixed.  The tracking bound, namely the largest value $\breve r^*$ of $\breve r$ above which there exist no imaginary eigenvalues, is bounded above by the square of the largest real root of \eqref{eq:quart2}. 
\end{theorem}
By inverting the transformation \eqref{prescale}--\eqref{pradius}, we have that $\lambda = i\breve r (1-\breve\xi^2)$ and $\xi = \breve r^{1/2} \breve\xi$.  From Theorem \ref{thm:strong-bounds} we find that there are no purely imaginary eigenvalues $\lambda$ of \eqref{linearized} satisfying
\begin{equation}
\label{hfbound}
|\lambda|\geq r^*(u_\sp,\breve\xi) := \breve r^*(\breve \xi) (1-\breve\xi^2),
\end{equation}
where $\xi = \breve r^*(u_\sp, \breve \xi)^{1/2} \breve\xi$.  
Thus, once the high-frequency bounds $\breve r^* := \breve r^*(\breve \xi)$ are established, 
we have reduced our investigations to a compact parameter domain in $\xi, \lambda$.

\begin{remark}
The caveat to Theorem \ref{thm:strong-bounds}, computationally speaking, is that one would need to prove that the coefficients in \eqref{positive-coefficients} are bounded and numerically tractable in order to securely claim high-frequency bounds.  When we plot the values for varying $\breve r$ (across a range of fixed $\breve \xi$), they seem like well-behaved functions that decay for large $\breve r$.  However, without further analysis, we cannot be certain that the functions do not blow up outside of our computational domain.  Therefore, below, we compactify the computational region by a more crude tracking bound $\breve r^*_{crude}$ that eliminates the need to investigate for $\breve r \geq \breve r^*_{crude}$.  In other words, we first show that admissible values of $\breve r$ are bounded above by $\breve r^*_{crude}$, then compute the coefficients in \eqref{positive-coefficients} by taking the suprema on the domain $0 \leq \breve r \leq \breve r^*_{crude}$ and use \eqref{eq:quart2} to establish the vastly sharper uniform tracking bound $\breve r^*$ defined above.
\end{remark}


\subsubsection{Crude Bounds}  

Using a combination of numerical evaluation and careful analysis, we are able to determine crude high-frequency bounds.  We begin with the realization that the norms on the $\Theta$-blocks in \eqref{tracksys} can be bounded by the norm on $\Theta$.  Specifically, we have that $\|\Theta^{(0)}_{kl} + \breve r^{-1/2} \Theta^{(-1/2)}_{kl}\| \leq \|\Theta\|$ for each $k,l\in \{-,+\}$.  Thus, the right-hand side of the tracking condition in \eqref{rootcond} satisfies
\[
\|\Theta_{\sm\sm}\|+\|\Theta_{\sp\sp}\| + 2\sqrt{\|\Theta_{\sm\sp}\|\|\Theta_{\sp\sm}\|}  \leq 4 \|\Theta\|.
\]
This results in the following corollary to Lemma \ref{tracklem}:
\begin{corollary}[Crude Bounds]
\label{cor:crude}
Let $\breve r_0 > 1$ and $\breve\delta\in[0,1]$ be given.  Then, the crude estimate
\begin{equation}
\label{crude-bounds}
\breve r^*_{crude}(\breve \xi) := \max\left\{ \breve r_0, \sup_{\breve r\geq \breve r_0} \left( \dfrac{4 \|\Theta(\breve r,\breve\xi)\|}{\breve\delta(\breve \xi)} \right)^2 \right\}
\end{equation}
is an upper bound for the modulus of any $\breve r$ corresponding to unstable imaginary eigenvalues.
\end{corollary}

Thus, our crude bounds are determined by finding uniform bounds on $\|\Theta\|$.  To begin, we observe that
\[
\|\Theta\| \leq \|L\| \|R\| \|\mat{U}\| \|\mat{U}^{-1}\| \left( \|\mat{C}_0\| + \breve r^{-1/2} \| \mat{C}_{-1/2}\|\right) + \| L R'\| + \|L\| \| R \| \| \mat{U}^{-1} \mat{U}'\|.
\]
Since $\|\mat{U}^{-1}\| = \|\mat{U}\|$ and $\| \mat{U}^{-1} \mat{U}'\| = \|\mat{U}'\|= \| \psi_{x_1}\|$, we have
\[
\|\Theta\| \leq \|L\| \|R\| \|\mat{U}\|^2 \left( \|\mat{C}_0\| + \breve r^{-1/2} \| \mat{C}_{-1/2}\|\right) + \| L R'\| + \|L\| \| R \| \| \psi_x\|.
\]
Note that $\|L\|\|R\|$, $\| L R'\|$, $\|\mat{C}_0\|$, $\|\mat{C}_{-1/2}\|$, and $\breve\delta^{-1}(\breve\xi)$ have no $\breve r$ dependence and are readily computable.  Thus, to compute uniform bounds for $\|\Theta\|$, we need the uniform bounds on $\| \mat{U}\|$ and $\|\psi_{x_1}\|$.  A tricky point in our analysis is that both $\| \mat{U}\|$ and $\|\psi_{x_1}\|$ depend on the matrix $( \hat \rho \breve \lambda \breve r^{1/2}\mat{I}_6 + \beta)^{-1}$, which has  $\breve r^{1/2}$ dependence.  Using Lemmas \ref{lem:matrix-norm}, \ref{lem:V}, and \ref{inhom}, we bound $\|\mat{U}\|$ and $\|\psi_{x_1}\|$ numerically by stepping through values of $s\in [0,(1+b)\|\beta\|]$ and $s\in [0,2\|\beta\|]$ in the computation of $M_0$ and $M_1$, respectively.  We note that the bound on $\|\Theta\|$ decays to a constant as $\breve r^{-1/2}$ and is thus uniform.  It remains to assign numerical values to the different terms above.

The following bounds are computed numerically for a monatomic gas ($\Gamma=2/3$) and taken to be uniform over $\breve\xi\in[0,1]$ and over all shock strengths $u_\sp\in[u_*,1]$, where $u_\sp\to 1$ is the transonic limit and $u_\sp\to u_*$ is the infinite Mach number limit (recall $u_* = 1/4$ is defined in \eqref{phys}).  The following $\breve r$-independent quantities are
\[
\|L\| \|R\| \leq 2.425,\quad
\|L R'\| \equiv 1,\quad
\|\mat{C}_0\| \leq 4.868,\quad
\|\mat{C}_{-1/2}\| \leq 8.5440,\quad
\breve\delta^{-1} \leq 1.72.
\]
From Lemma \ref{inhom}$(i)$-$(ii)$, we have $M_0 = 0.8130$ and $M_1 = 2.0057$.  From Lemma \ref{lem:V}, this gives
\[
\|\mat{U}\|^2 = 1 + \dfrac{M_0^2 + M_0 \sqrt{M_0^2 + 4}}{2} \leq 2.2081.
\]
Furthermore, from  Lemma \ref{inhom}$(iii)$, we find the bounds
\[
\sup \left| \dfrac{p_{x_1}}{p} \right| \leq 0.3027 \quad\mbox{and}\quad \sup \left| \dfrac{\rho_{x_1}}{\rho} \right| \leq 0.3370.
\]
Thus, for $\breve r_0 = 40,\!000$, we have that
\[
\|\psi_{x_1}\| \leq \left(0.3027 + (2.0057)(0.3370)\sqrt{1+\dfrac{2}{200}}\right) (0.8130) \leq 0.7984
\]
and
\[
\|\Theta\| \leq (2.425)(2.2081)\left(4.868 + \dfrac{1}{200}8.5440\right) + 1 + (2.425)(0.7984) \leq 29.2313.
\]
Therefore, for a monotomic gas with parameters as given in Section \ref{sec:equations}, we have that
\[
\breve r^*_{crude}(\breve\xi) \leq \max\left\{ 40,\!000,\left( 4 (29.2313)(1.72) \right)^2 \right\}\leq 40,\!446,\quad\forall\breve\xi\in[0,1].
\]
Thus, when $\breve r>\breve r^*_{crude}$ we have that no purely imaginary eigenvalues of \eqref{linearized} can exist.  Having crude bounds allows us to compactify the suprema computed in \eqref{positive-coefficients}, thus establishing the vastly smaller tracking bound in Theorem \ref{thm:strong-bounds}.

\subsubsection{Numerical results}

In Figure \ref{fig:hfb}, we display the high-frequency bounds for monatomic shocks computed as a function of $\breve\xi\in[0,1]$ and $u_\sp\in[u_*,1]$ as given by Theorem \ref{thm:strong-bounds} and \eqref{hfbound}.  These vary from approximately 32 for weak (small-amplitude) shocks to approximately $280$ in the strong-shock (maximal amplitude) limit, with the largest values  occurring for $\breve \xi=0$, corresponding to the one-dimensional case.  The diatomic case looks essentially the same, but with bounds ranging from roughly $32$ to $500$.  These values give an upper bound on our computational region.

\begin{figure}[t]
\begin{center}
\includegraphics[width=16cm]{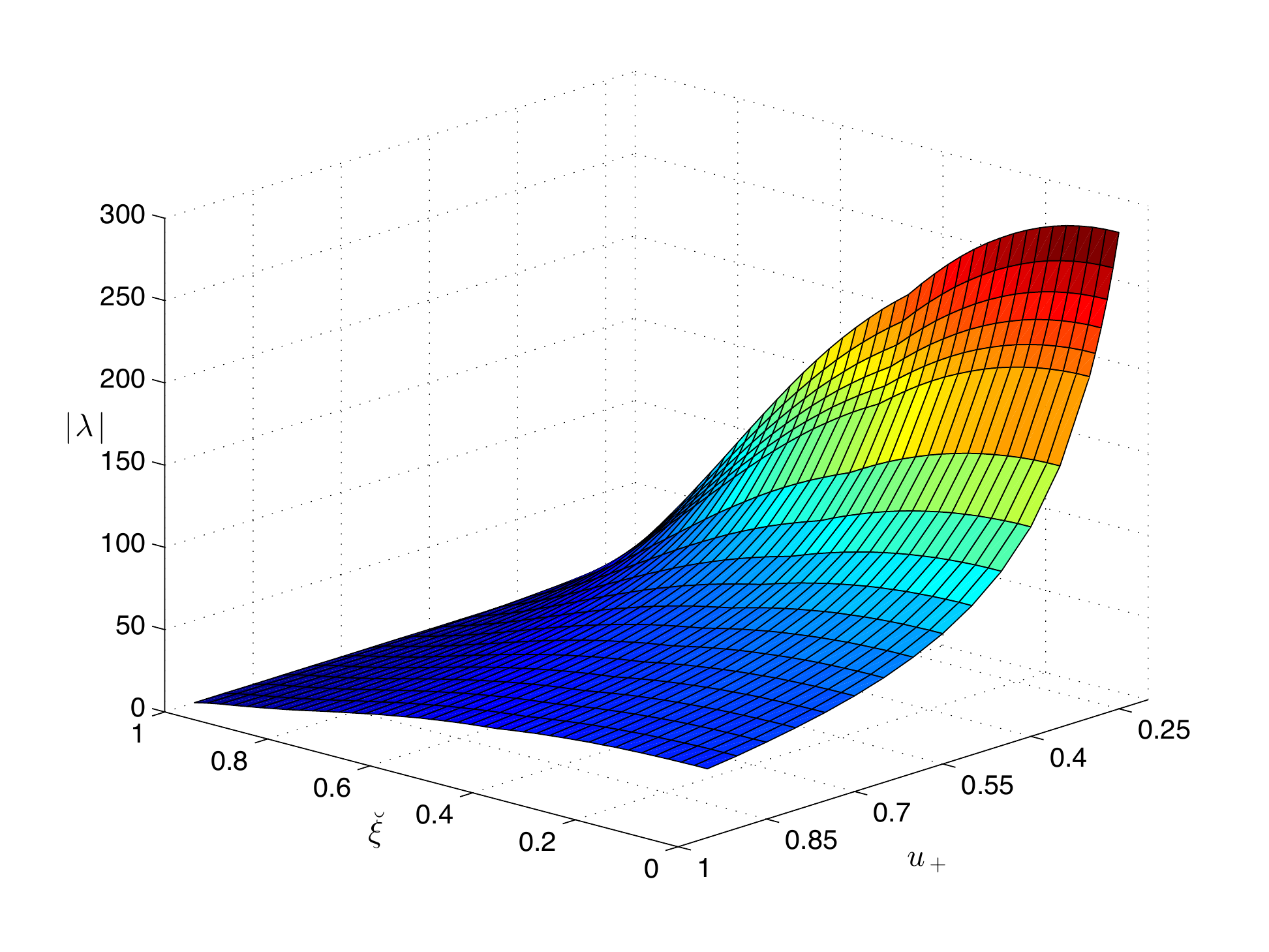}
\end{center}
\caption{Monotone high-frequency bounds $|\lambda|$ as a function of $\breve\xi$ and $u_\sp$.  Notice that the bounds are large for strong shocks when $\breve\xi$ is near zero and relatively small elsewhere, that is when shocks are weak or when $\breve\xi$ is near $1$.}
\label{fig:hfb}
\end{figure}

\begin{remark}
Here, we are blocking the imaginary axis for high frequencies, which will suffice (by a homotopy argument) to capture any unstable zeros as they cross the imaginary axis.  As described below, for bounded frequencies we not only check for zeros on the imaginary axis, but perform winding number computations about semicircular contours in the unstable half plane with bounding diameter along the imaginary axis, counting roots in the interior.  Thus, we detect any such crossing roots not only at the moment that they cross the imaginary axis, but afterward as well, making this approach particularly robust and convenient; in particular, we can be sure that crossings are not missed between mesh points in the computation. 
\end{remark}

\begin{remark}
There are many ways to carry out the block-diagonalization described above.  Here, we followed a naive approach and it turned out to be sufficient.  If necessary, one may use Lyapunov's Lemma (Lemma \ref{lem:lyapunov} below) to effect a further change of coordinates guaranteeing the hypotheses of the quantitative tracking lemma (Lemma \ref{tracklem}) in terms of $\mat{M}_\pm=\Re(\mat{P}_\spm^{1/2}\mat{N}_\spm\mat{P}_\spm^{-1/2}$).  We don't find this to be needed in the present case since $\Re\mat{N}_\spm$ already have a gap; thus we simply set $\mat{M}_\pm=\mat{N}_\pm$.
\end{remark}


\section{Low-frequency study}
\label{numerical}
\subsection{Theoretical background and implementation}

As noted in the introduction, the low-frequency limit of the (viscous) Evans function captures inviscid behavior described in terms of a Lopatinski determinant \cite{ZS}. Indeed, this is one of the regimes in which substantial analytic information can be extracted from the Evans function. In particular, combining the inviscid stability results of Majda \cite{Maj1} and Erpenbeck \cite{Er} for ideal, polytropic gases with the low-frequency analysis of Zumbrun \& Serre \cite{ZS} and Zumbrun \cite{Z3,Z5}, we may immediately conclude that there are no low-frequency zeros except at the origin. In this section, we show that our numerical implementation of the Evans function captures the expected behavior near the origin. This both provides a strong validation for our computations and, when combined with our high-frequency (Section \ref{sec:hfb}) and intermediate-frequency (Section \ref{sec:intermediate}) studies,  allows us to give a complete description of the behavior of the Evans function on its entire domain.  

To capture the Lopatinski determinant, itself a function of frequencies which is homogenous degree one, the low-frequency analysis of the Evans function is based on radial limits. Here, we numerically investigate such limits using the balanced flux formulation of the Evans function (see Section \ref{balanced} and \cite{BHLyZ2}); this alternative Evans function does not vanish at the origin but still detects unstable eigenvalues. Moreover, it has (nonvanishing) radial limits equal to the value of the Lopatinski determinant at the same frequency angle on the ball of radius one.  In principle, we could estimate the region of non-vanishing analytically, but instead we estimate the size of this region by performing a numerical convergence study along rays through the origin: at the point at which we estimate the relative error to be $0.05$, we assume that there are no zeros along this ray.

Finally, to implement this calculation, we initialize the Evans function with bases at $\infty$ by first evolving using Kato's ODE in the angle along a semicircle of radius $0.1$ for the monatomic case and $0.12$ for the diatomic case.  Then we evolve radially toward the origin to $0.001$ using Kato's ODE again (this time in the radial variable), verifying the relative error at the end to be below $0.05$.

\subsection{Estimation of relative error.} A detail of the algorithm is the means of estimating the relative error 
$$
\frac{
|\check D(\check \xi, \check \lambda, \check r)- \check D(\check \xi, \check \lambda, 0)|}
{| \check D(\check \xi, \check \lambda, \check r)|}.
$$
By the analysis of \cite{ZS,Z3}, we know that 
$\check D(\check \xi, \check \lambda, \check r)- \check D(\check \xi, \check \lambda, 0)\sim \check r^s$,
where $s=1$ at points of analyticity and $s=1/2$ at ``glancing points'' where $\check D(\cdot, \cdot, 0)$ exhibits a square-root
singularity.
Hence, for successive mesh points $\check r_{j+1}<\check r_j$, 
$$
\begin{aligned}
	|\Delta \check D_j|&:=
|\check D(\check \xi, \check \lambda, \check r_{j+1})- \check D(\check \xi, \check \lambda, r_j)\\
& \sim s\check r_j^{s-1}|\Delta \check r_j|\\
&\sim
s |\check D(\check \xi, \check \lambda, \check r_j)- \check D(\check \xi, \check \lambda, 0)| \check r_j^{-1} \Delta \check r_j|,
\end{aligned}
$$
where $\Delta \check r_j:=\check r_{j+1}-\check r_j$.
Thus, rearranging, we may estimate
%
%
\begin{equation}
\label{stopcrit}
\frac{
|\check D(\check \xi, \check \lambda, \check r)- \check D(\check \xi, \check \lambda, 0)|}
{| \check D(\check \xi, \check \lambda, \check r)|} \approx
s^{-1} \frac{|\Delta \check D_j|/|D_j|}
{|\Delta \check r_j|/|r_j|} \leq 2 \frac{|\Delta \check D_j|/|D_j|}
{|\Delta \check r_j|/|r_j|}.
\end{equation}
Thus, we can use as a stopping or sufficiency criterion in our numerical convergence study the requirement that the righthand side be sufficiently small, say $\leq 0.05$.

\begin{figure}[t]
\begin{center}$
\begin{array}{cc}
\includegraphics[width=8.25cm]{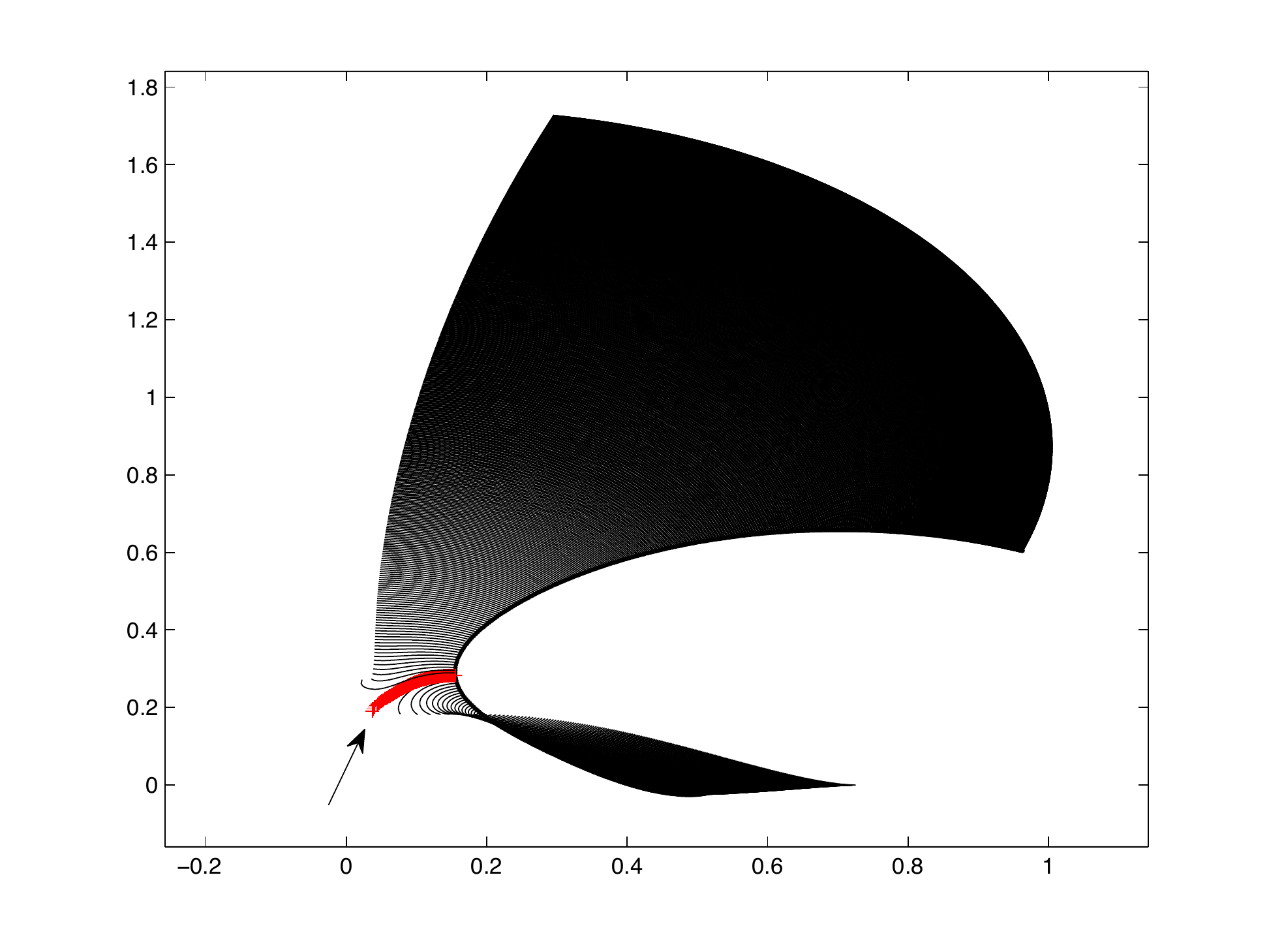} & \includegraphics[width=8.25cm]{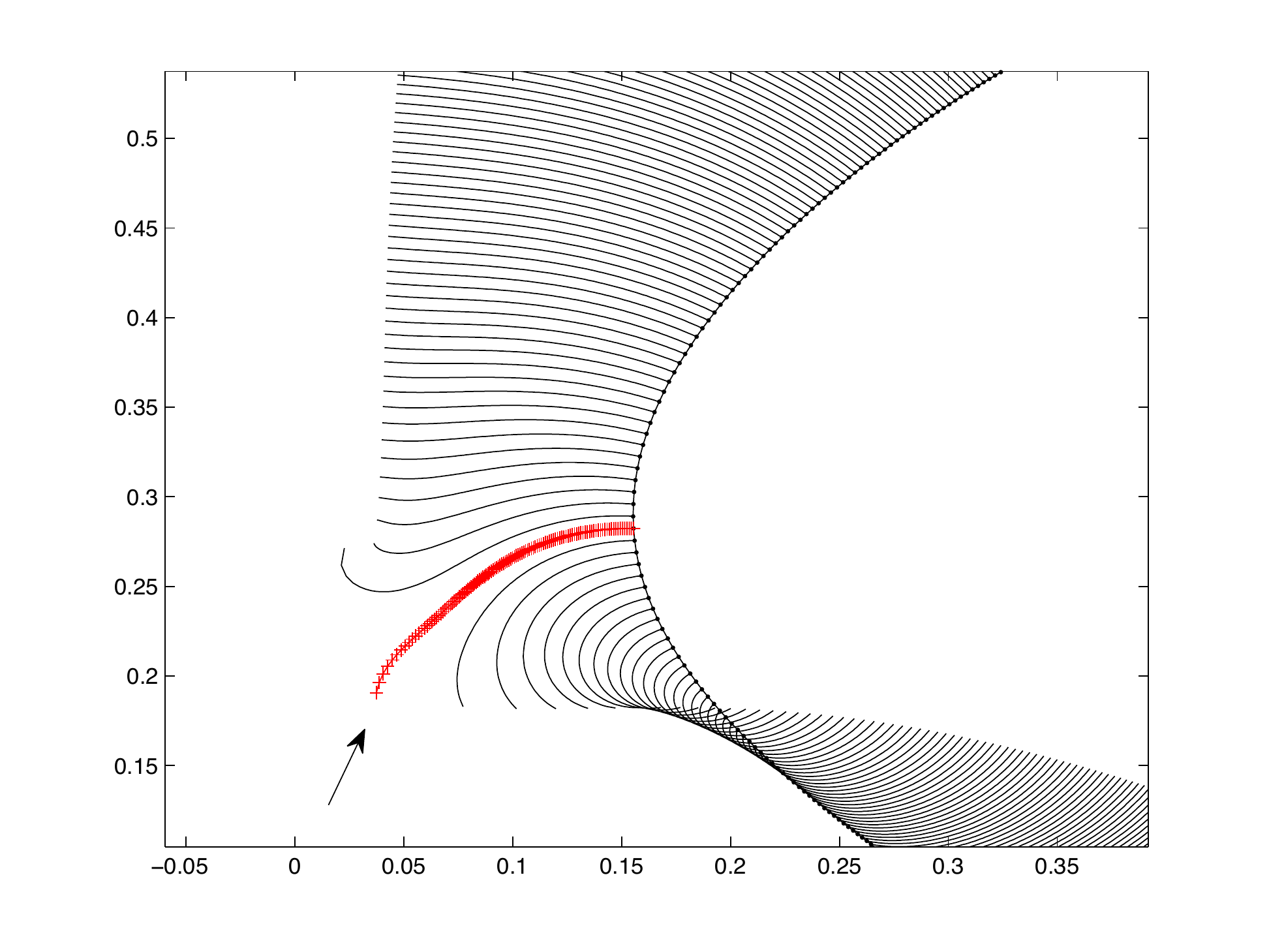} \\ 
$(a)$ & $(b)$
\end{array}$
\end{center}
\caption{Two pictures of the low-frequency ``rib roast'' showing radial limits of the Evans function for the case of a monotone gas.  The corner spoke is given by $\theta_{310} = \frac{310}{1000}\cdot \frac{\pi}{2}\approx 0.4869$.  It is close to the theoretical value of $\theta_* = 0.4867$ given in \eqref{theta-roast}. Note that all values are well bounded away from zero, indicating stability for $|\xi,\lambda|\leq 0.1$.}
\label{rib-roast}
\end{figure}

\subsection{Numerical results, and a useful check}


The limit (with respect to radial variable $r$) of the Evans function is a function of angle corresponding to the Lopatinski condition of the inviscid theory \cite{ZS,BHLyZ2}. This function is known to have a square-root singularity in the angle at certain ``glancing'' points (see, e.g., the appendix by Jenssen \& Lyng that appears in \cite{Z3}). As discussed in \S5.5.1 of \cite{BHLyZ2}, this singularity is preserved by our method of initialization/coordinatization.  

More precisely, consultation of Jenssen \& Lyng's appendix shows that the precise location is given by $ |\Im\lambda|^2= (c^2-|u|^2)|\xi|^2$.  (Note that, by the Lax shock inequalities, this can only occur on the righthand side $x=+\infty$.)  Thus, in polar coordinates, we obtain
\begin{equation}
\label{theta-roast}
\theta_*= \tan^{-1} \sqrt{c_\sp^2-u_\sp^2}\,.
\end{equation}
Then, using the following physical parameters,
\[
e_\sp = \dfrac{u_\sp(\Gamma + 2 - \Gamma u_\sp)}{2\Gamma(\Gamma+1)} = 0.333\,,
\quad 
c_\sp = \sqrt{\Gamma(\Gamma+1)e_\sp} = \sqrt{0.37}\,,
\quad 
u_\sp = 0.3\,, 
\quad 
\Gamma = 2/3\,;
\]
we computed the Evans function by first taking a contour of the form $(r, \cos\theta, i \sin\theta)$, where $\theta$ varied $0\leq\theta\leq \pi/2$.  Then, we computed a series of radial spokes by following along the fixed-angle contours $(r,\cos\theta_k,i \sin\theta_k)$, where $r_1\leq r \leq r_2$.  We use $r_1 = 0.001$ and $r_2=0.1$.  For our computation, we took $n+1$ radial spokes $\theta_k = \frac{k\pi}{2}$, where $k=0,1,2,\ldots,n$; in our case $n=1000$.  For the given parameters,
\[
\theta_* = \tan^{-1}\sqrt{0.28} = 0.4867\,, 
\]
so we examine the behavior near the spoke whose $\theta_k$ value is closest to $\theta_*$. Evidently, for this experiment, the relevant value is $\theta_{310} = \frac{310}{1000}\cdot\frac{\pi}{2}\approx 0.4869$. Figure \ref{rib-roast} shows the output of this experiment; the spoke corresponding to $k=310$ is highlighted in the figure, and the singularity is clearly visible, and its presence at the correct location gives a strong confirmation of the validity of our computations.  Note that all values of $\check D$ are clearly nonzero, indicating stability.

%
%
We carried out similar experiments for both the monatomic gas case $\Gamma=2/3$ and the diatomic gas case, letting $u_\sp$ vary across the strong and weak shock regimes.  Specifically, we computed \eqref{stopcrit} for \eqref{uprange_monatomic} and \eqref{uprange_diatomic}.  An a posteriori check reveals that the ratio in \eqref{stopcrit} does not exceed $0.0383$, verifying our convergence criterion.

\subsection{Expanded study for a single shock.} 
Recall that we are eliminating possible unstable imaginary roots, for purpose of a later homotopy argument.  For this later argument, we require a single, initial, shock wave for which there exist no low-frequency roots with $\Re \lambda\geq 0$.  For both the monatomic and diatomic gas cases, we investigate the case $u_\sp=0.6$, performing the same type of low-frequency analysis as previously, but now taking $\lambda= \sin\theta e^{i\phi}$, adding an additional angular loop $0\leq \phi\leq \pi/2$.  Specifically we chose $\Delta \phi=\pi/50$ and tested $21$ values ranging from $\phi=0$ to $\pi/2$.  As before, we find that all values are nonzero, with ratio \eqref{stopcrit} not exceeding $0.0038$, thus verifying nonexistence of any (not only pure imaginary) unstable low-frequency eigenvalues.  We note that the bound is an order of magnitude smaller than the $\phi=0$ bound performed for all values of $u_\sp$ because the upper bound was very small for $u_\sp=0.6$ relative to the other values tested.

\subsection{Conclusion} There are no unstable imaginary eigenvalues for $|\xi, \lambda|\leq 0.1$, for any shock strength, for either the monatomic or diatomic case.  For the special value $u_\sp=0.6$, there are no unstable eigenvalues with nonnegative real part, for monatomic or diatomic case.


\section{Intermediate-frequency study}
\label{sec:intermediate}

Having treated the high- and low-frequency regimes, $\breve r\geq r_*(\breve \xi)$ and $|\xi,\lambda|\leq 1.0$ in the coordinates of their respective sections, it remains to study the main, intermediate-frequency, regime consisting of their complement.  The first step is to identify a region containing this set in the natural coordinates $\xi$, $\lambda$ appropriate for this regime, and in a form suitable for winding number computations holding $\xi $ fixed and varying $\lambda$, that is, to bound the set $\breve r\leq r_*(\breve \xi)$ by one composed of slices $\xi=\xi_0$, $|\lambda| \leq r$.

\begin{lemma}
\label{nojordan}
The set $\Sigma_1:=\{(\xi, |\lambda|)=(\breve r^{1/2}\breve \xi, \breve r (1-\breve \xi^2)): \;
0\leq \breve r \leq r_*(\breve \xi), \; 0\leq \breve \xi\leq 1 \}$ defined in \eqref{hfbound}\footnote{
Here and below we suppress the $u_\sp$ argument for ease in writing.} is contained in 
$
\Sigma_2:=\cup_{0\leq \breve \xi\leq 1}
\{(\xi, |\lambda|): \; \xi= r^*(\breve \xi)^{1/2} \breve \xi, \, 0\leq |\lambda|\leq r^*(\breve \xi)(1-\breve \xi^2\}.
$
\end{lemma}

\begin{proof}
	Define $S:=\{\breve \xi, \breve r: 0\leq \breve r\leq r^*(\breve \xi)\}$ and 
	$
	\Psi:(\breve \xi, \breve r) \to (\xi, |\lambda|)=(\breve r^{1/2}\breve \xi, \breve r(1-\breve \xi^2)),
	$
	so that $\Sigma_1=\Psi(S)$.  
	It is readily verified that $\Psi$ is one-to-one on $S$, except along the boundary $\breve r=0$, which maps to
	the point $(\xi,|\lambda|)=(0,0)$.
	For, equating 
	$$
	(\breve r_1^{1/2}\breve \xi_1, \breve r_1(1-\breve \xi_1^2)) =
	(\breve r_2^{1/2}\breve \xi_2, \breve r_2(1-\breve \xi_2^2)) ,
	$$
	we find, squaring the first coordinates and adding to the second, that $\breve r_1=\breve r_2$, whence, comparing first coordinates, $\breve \xi_1=\breve \xi_2$ unless $\breve r_j= 0$.
	Thus, $\Psi(\partial S)$ defines a simple closed curve in the $\xi$-$\lambda$ plane, and 
	(by continuity of $\Psi$ together with the Jordan Curve Theorem) therefore
	$\Psi(S)^{int}=\Psi(\partial S^{int})$ and
	$\partial \Sigma_1=\partial \Psi(S)=\Psi(\partial S)$.
	Rewriting 
	$$
	\Sigma_2=\cup_{(\xi_0,|\lambda_0|)\in \Psi(\partial S)} \{\xi_0 \} \times [0, |\lambda_0|]
=\cup_{(\xi_0,|\lambda_0|)\in \partial \Sigma_1} \{\xi_0 \}\times [0, |\lambda_0|], 
$$
we thus have evidently $\Sigma_1\subset \Sigma_2$.
\end{proof}

Our protocol, based on Lemma \ref{nojordan}, is to step through $\breve \xi\in [0,1]$, 
testing for roots of the Evans function on 
the fixed-$\xi$ slices $ \xi= r^*(\breve \xi)^{1/2}\breve \xi$, $ 0\leq |\lambda|\leq r^*(\breve \xi)\}$
using a winding number computation in $\lambda$ around the semicircle $\Gamma:=\partial \{\lambda: \; \Re \lambda \geq 0, \, |\lambda|\leq r^*(\breve \xi)\}$.
For this analysis, it is necessary to use a version of the Evans function (i) is analytic in $\lambda$, and (ii) has no roots 
on, or overly near, the curve $\Gamma$, in particular at $\lambda=0$.
The standard Evans function $D(\xi,\lambda)$ defined in \eqref{evanseq} has these properties for $\xi$ bounded away from $0$,
but for $\xi=0$ has a root at $\lambda=0$, hence gives numerical difficulty for $\xi$ too near $0$.
We find it useful, therefore, to supplement this standard Evans function in our investigations
with a modification removing the zero at $(\xi,\lambda)=(0,0)$, particularly near small values of $\xi, \lambda$.

\subsection{Modified balanced flux coordinates}
Specifically, we construct For intermediate frequencies a modified version of the balanced flux Evans function of Section \ref{balanced} possessing the property of analyticity in $\lambda$.  The key idea is to note that the Evans system \eqref{bevans_ode} is independent of the definition of $r(\xi,\lambda)$.  Meanwhile, consistent splitting is inherited in $r,\xi^\sharp, \lambda^\sharp$ coordinates wherever it holds in the original $\xi$, $\lambda$ coordinates, where $\xi^\sharp:=\xi/r, \lambda^\sharp:= \lambda/r$.  Thus, defining 
\beq
\label{Dsharp}
D_\sharp(r, \xi^\sharp, \lambda^\sharp)
\eeq
to be the Evans function associated with \eqref{bevans_ode}, we obtain an arbitrary number of different 
versions of the Evans function, with different properties, through the prescription
\beq
\label{Drecovery}
D_{r(\cdot)}(\xi,\lambda) := D_{\sharp}(r(\xi, \lambda), \xi/r(\xi,\lambda), \lambda/r(\xi,\lambda))
\eeq
generalizing \eqref{evanseq}, one for each choice of $r(\cdot)$.

The choice $r(\xi, \lambda)=|\xi,\lambda|$ recovers the balanced flux version described in Section \ref{balanced}.  As described further in \cite{BHLyZ2}, this has the advantage of removing roots at the origin, similarly as for integrated coordinates in 1D.  However, it has the disadvantage that the resulting Evans function $D_{r(\cdot)}$ is not analytic in either of the variables $\xi$ or $\lambda$, hence we are not able to carry out winding number computations, thus losing a key dimensional advantage in carrying out large-scale Evans computations in the intermediate frequency regime.

For intermediate frequencies, we therefore introduce the modified radial function
\[
r=r_2(\xi,\lambda):=|\xi|+\lambda
\]
with the twin properties that $(i)$ $r_2\sim r$ for low frequencies, so that this definition likewise removes vanishing at the origin, and $(ii)$ $r_2$ unlike $r$ is analytic in $\lambda$ for each fixed $\xi$.  The associated Evans function $D_{r_2}$ inherits analyticity in $\lambda$ away from $(\xi,\lambda)=(0,0)$, and is continuous (indeed, analytic) along rays through the origin.

This allows us, using the Evans function $D_{r_2}$, to step through various values of $\xi$ and check for each value by a winding number computation, whether there are any zeros of the Evans function within the bounded computational domain given by the complement of \eqref{hfbound}.  See \cite{BHLyZ2} for further discussion.

\subsection{Advantages and disadvantages}

In some sense the modified balanced flux version of the Evans function is the multidimensional generalization of the integrated Evans function of one-dimensional theory, removing undesirable roots at the origin while preserving analyticity in $\lambda$ if not $\xi$.  However, whereas the integrated Evans function has similar qualitative behavior to the standard (unintegrated) one, the modified balanced flux version is not quite so well-behaved.  A delicate point for this version  is behavior for $|\xi|\ll 1$ but not zero, since there is a rapid transition at $|\lambda|=|\xi|$ from angle $(\check \xi, \check \lambda)=(0,1)$ to $(\check \xi, \check \lambda)=(i,0)$, with associated rapid variation in $D_{r_2}$ corresponding to the limiting values along radii of different angles through the origin.  This leads to substantial winding of the image contour for values of $|\lambda|$ on the order of $|\xi|$, requiring additional refinement in case of small $\xi$ in order to obtain a smooth-looking output.   Even for intermediate values of $\xi$, the image contours for the modified balanced flux Evans function $D_{r_2}$ exhibit substantial additional winding near $\lambda=0$ as compared to image contours of the standard Evans function $D$ and degrading the attractiveness (and readability) of figures.  See Figures \ref{fig:large-small-shock-xi}--\ref{fig:xi-zero-modified}.

However, the property that $D_{r_2}$ does not vanish at the origin ensures that winding number computations remain well-conditioned even for small $\xi$.  Indeed, it should be noted that, for both standard and modified balanced flux Evans functions, it is not necessary to compute winding numbers for contours passing arbitrarily close to $(\xi,\lambda)=(0,0)$, since we have already treated this region by a separate low-frequency analysis.  In the end, the results from both (standard, and modified balanced flux) types of winding number computation are satisfactory,  giving a useful double check on stability conclusions.


\begin{figure}[t]
\begin{center}$
\begin{array}{cc}
\includegraphics[width=8.25cm]{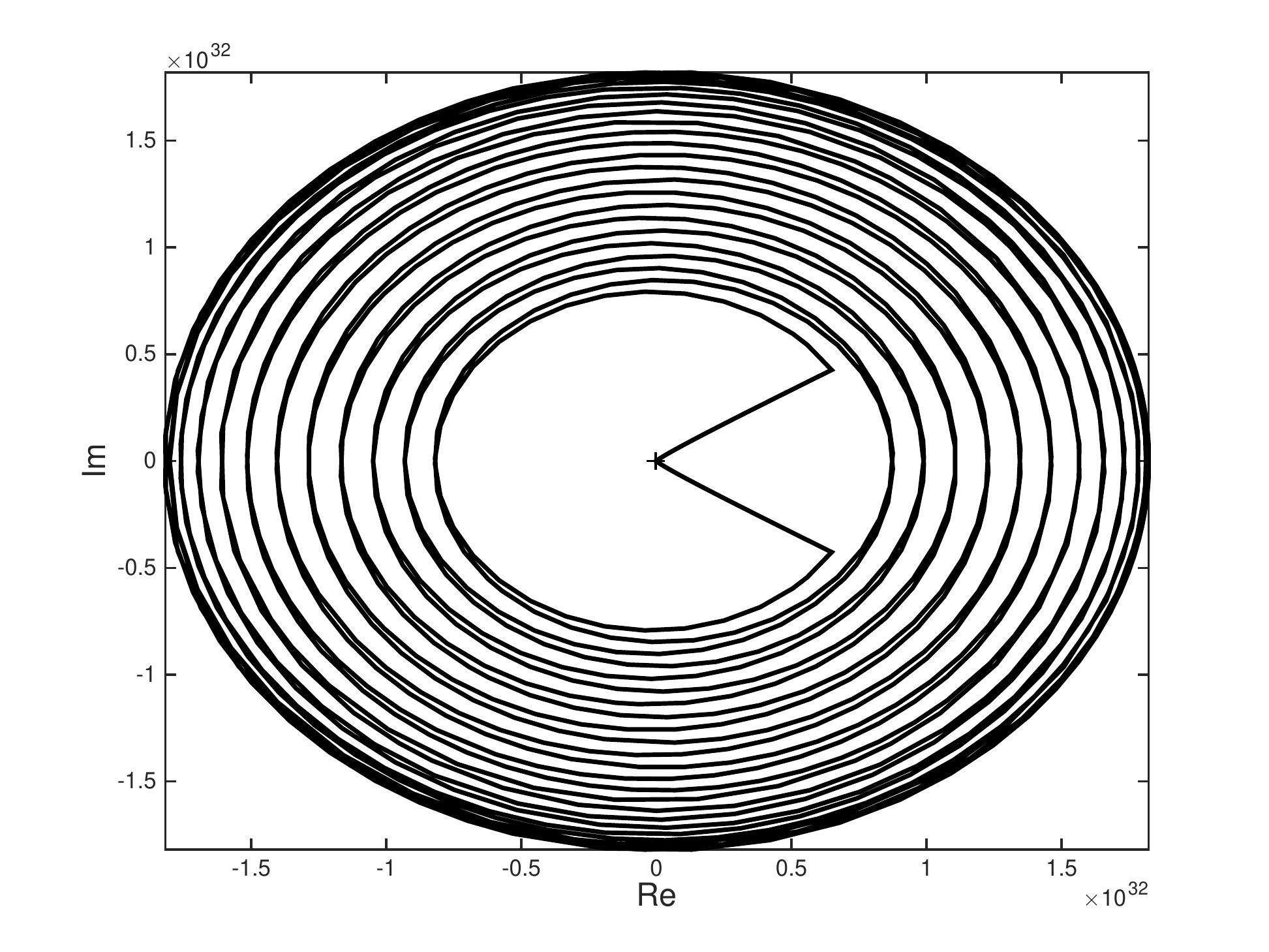} & \includegraphics[width=8.25cm]{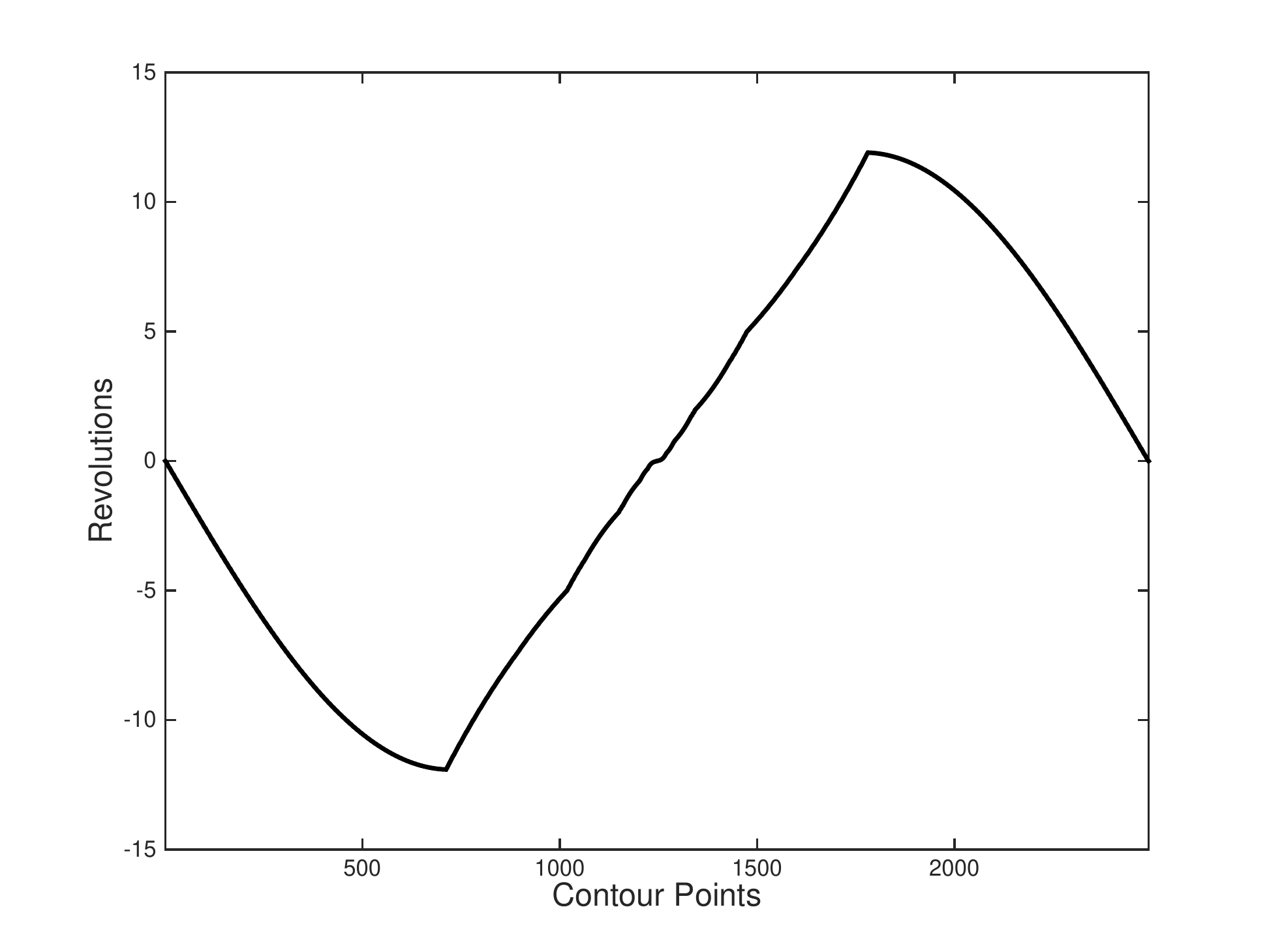} \\ 
(a) & (b)
\end{array}$
\end{center}
\caption{Evans function output along the semi-circular contour in Eulerian coordinates.  In $(a)$ we see the output of the weak-shock high-$\breve\xi$ case with $u_\sp=0.27$ and $\breve\xi=0.95$.  This corresponds to $r^*=16.3$ (for the high frequency bounds in $|\lambda|$) and $\xi=12.2778$. In $(b)$, we compute the angle of the winding.  Notice that the outer circles wrap 24 times around the origin.  By zooming into the origin one sees an unwinding of the contour the other way, so that the net winding is zero.  In other words, this contour has winding number zero and therefore demonstrates that shocks in this parameter regime are spectrally stable, yet it is not obvious via inspection that this is the case.  We remark that over 2492 contour points were needed to produce this graph.}
\label{fig:winding}
\end{figure}

\subsection{Pseudo-Lagrangian coordinates}
\label{ssec:pseudolag}
Besides the change in dependent variables in the Evans system just discussed, we find it necessary also to make a ``pseudo-Lagrangian'' change of coordinates $x_1\to y_1$ in the independent variable $x_1$, determined by
\beq
\label{pl}
dy_1/dx_1=\hat \rho(x_1).
\eeq
This has the effect of converting the upper lefthand diagonal element $-r\check\lambda\hat \rho $ in the coefficient matrix given in \eqref{bevans_ode} to a constant:  $(dx_1/dy_1)(-r\check\lambda\hat \rho )= -r\check\lambda$.  As this entry, similarly as in the high-frequency analysis of Section \ref{sec:hfb}, dominates order $r$ high-frequency behavior, and as the Evans function by its construction factors out exponential growth predicted by the endstates of the coefficient matrices at $x_1=\pm \infty$, the result is to eliminate factors $\sim e^{C\lambda}$ that would otherwise occur, leaving more manageable $e^{D\sqrt{\lambda}}$ asymptotics for large $\lambda$.  This has the effect of greatly reducing winding about the origin in our winding number computations, thus reducing the number of mesh points needed and thereby computation time; indeed, in the original coordinates, we were not able to compute past $|\lambda|$ values of order more than one, whereas in pseudo-Lagrangian coordinates we are able to compute out the two orders of magnitude farther that we need to complete the analysis.  See \cite{BHLyZ1} for further discussion of this important point.

To illustrate the amount of winding in basic coordinates, we consider the shock $u_\sp=0.75$ with $\breve\xi=0.95$.  This corresponds to $r^*=10.4$ (for the high frequency bounds in $|\lambda|$) and $\xi=6.2770$.  In Figure \ref{fig:winding}, we see that the Evans function winds around nearly $5$ times before unwinding for a net winding number of zero.  Note however that a radius of $r^*=10.4$ is relatively small for the high-frequency bounds as given in Figure \ref{fig:hfb}.  Indeed when the shock is strong (close to $u_\sp=0.25$) and the value of $\breve\xi$ is small (close to $\breve\xi=0$), we see that the radius $r^*$ is around $280$.  This results in so much winding that it is prohibitive to resolve the winding number using the basic form of the Evans function given in \eqref{firstorder}--\eqref{evans_ode}.  Thus, the pseudo-Lagrangian coordinates are essential to this study.


\subsection{Numerical implementation}

Our approach to Evans function computation is on its way to becoming a mature technology.  All computations in this study were performed using {\sc STABLAB}, which is a {\sc MATLAB}-based toolbox developed by Barker, Humpherys, Lytle, and Zumbrun \cite{STABLAB}.

For our intermediate-frequency study, we begin by computing the two-variable traveling wave profile, which is formulated as a three-point boundary-value problem on the infinite domain $\mathbb{R}$, that is, in addition to conditions at the end states, a third condition.  To simplify this, we halve the domain and double the size of the system, thus expressing the system as a two-point boundary-value problem in four variables over a half-line domain $[0,\infty)$.  We use a sixth-order solver {\tt bvp6c} \cite{bvp6c}, which is a generalization of the built-in lower-order solvers {\tt bvp4c} and {\tt bvp5c} included in {\sc MATLAB}.  We set the absolute and relative tolerances to be $10^{-8}$ and $10^{-6}$, respectively.

Using two matching conditions, the centering condition, and a projective boundary condition as $x_1\to\infty$, we are able to solve the profiles over the entire parameter regime $u_\sp\in [u_*,1]$ with relative ease.  We follow the process described in \cite{HLyZ} to make sure the computational domains $[-L,L]$ are sufficiently large to get good approximations of the Evans function.  For a monatomic gas, the domain varied from $L=40$ in the strong shock case to $L=300$ in the weak shock case and for the diatomic case, the domain varied from $L=40$ in the strong shock case to $L=400$ for the weak shock case.  Note that the weak shock case has slower exponential decay to the end states and therefore requires a larger domain.

Following the computation of the profile, we are able to set up the batch jobs for Evans function computation.  The {\sc STABLAB} package uses the continuous orthogonalization approach to Evans function computation as described in \cite{HuZ2}.  Specifically, we compute the Evans function by solving the system
\begin{subequations}
\label{abel-polar}
\begin{align}
\Omega' &= (I -  \Omega \Omega^* ) A \Omega,  \label{abel-polar:1}\\
\gamma' &= \tr(\Omega^* A \Omega)\gamma.  \label{abel-polar:2}
\end{align}
\end{subequations}
The Evans function then takes the form
\begin{equation}
\label{evans-final}
D(\lambda) = \gamma_-(0)\,\, \gamma_+(0) \,\, \det\left(\begin{bmatrix} \Omega_- &  \Omega_+\end{bmatrix}\right) \Bigr|_{x_1=0}.
\end{equation}
We remark that there are important pre-processing steps built into \cite{STABLAB}, such as making an analytic choice of initial conditions via Kato's method \cite{Kato}; see \cite{BHRZ, HLZ, HLyZ, BHZ1} for discussions on these details.

Because of the excessive winding in the case where the shock is strong ($u_\sp$ near $1/4$ for monatomic and $u_\sp$ near $1/6$ for diatomic) and $\breve\xi$ is small, we also compute the Evans function using the no-radial option where we throw out the contribution from the scalar ODE \eqref{abel-polar:2} that us used to restore analyticity.  Thus we simply compute the quantity
\begin{equation}
\label{evans-no-radial}
D_{nr}(\lambda) = \det\left(\begin{bmatrix} \Omega_- &  \Omega_+\end{bmatrix}\right) \Bigr|_{x_1=0},
\end{equation}
This no-radial formulation leaves us with a topological index that still contains the same winding number properties (a Poincar\'e index) needed to determine the number of eigenvalues inside the contour; for more details see \cite{AB}.  All of these are options built into {\sc STABLAB}.


\begin{figure}[t]
\begin{center}$
\begin{array}{cc}
\includegraphics[width=8.25cm]{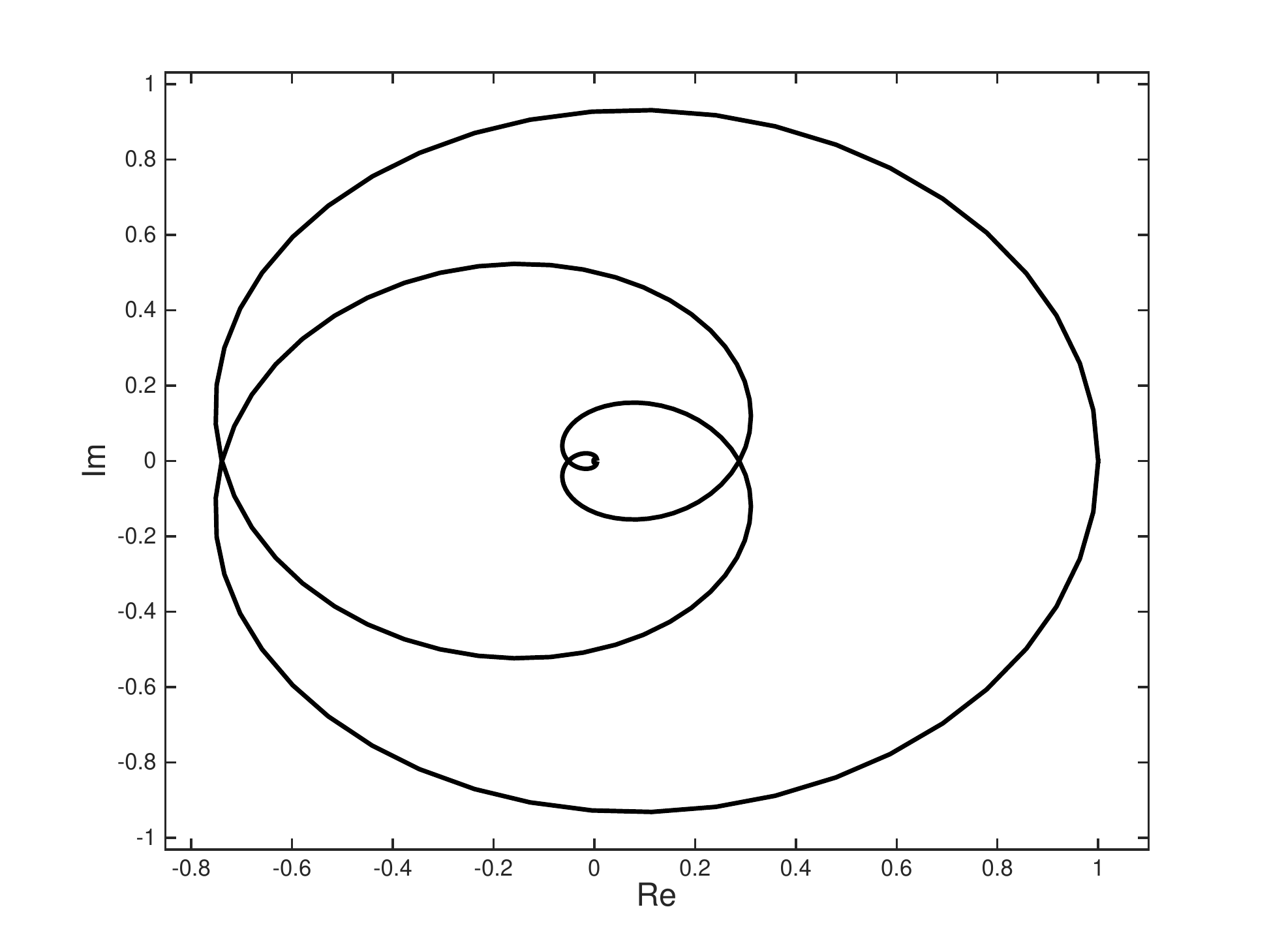} & \includegraphics[width=8.25cm]{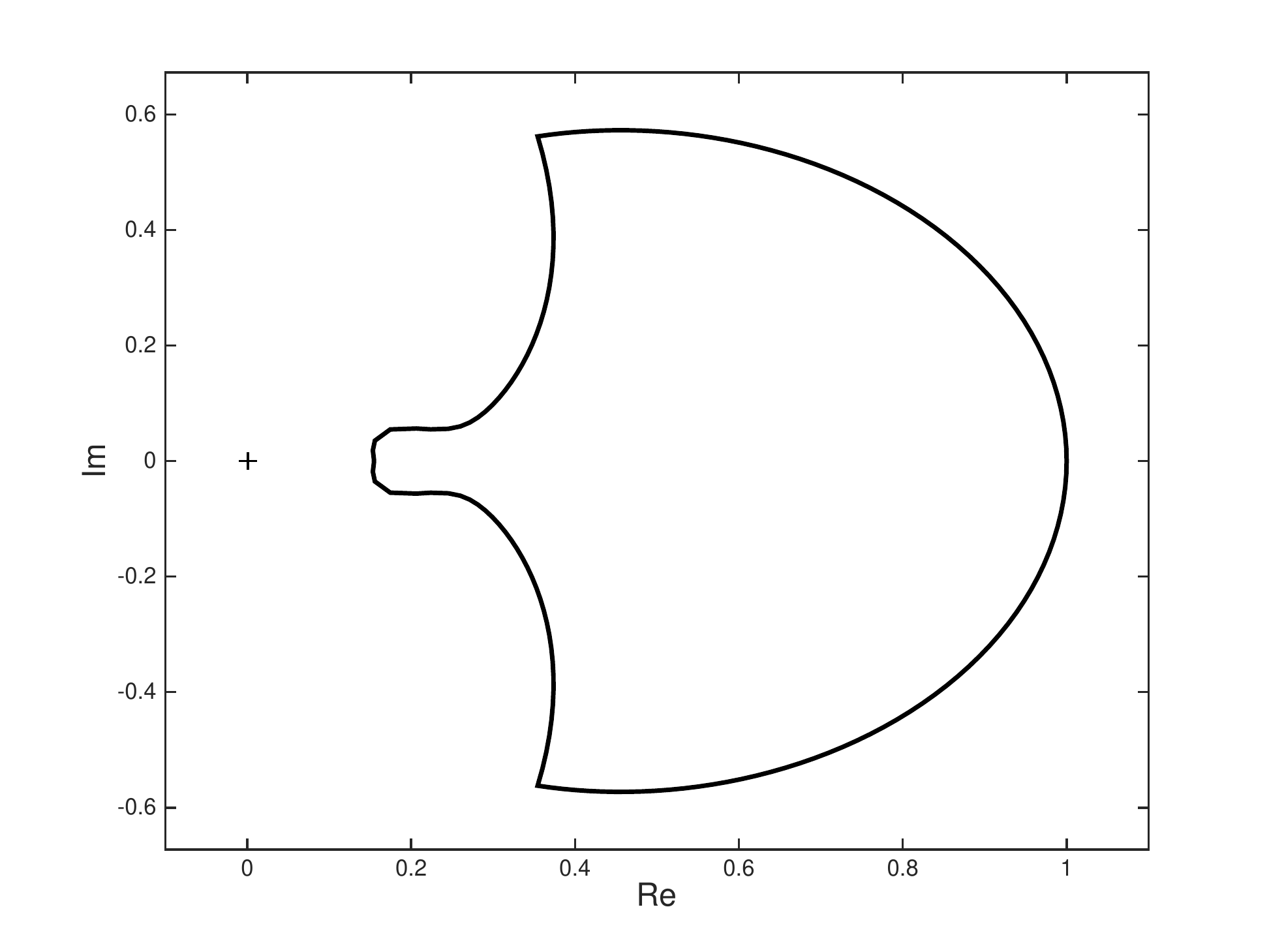} \\ 
(a) & (b)\\
\includegraphics[width=8.25cm]{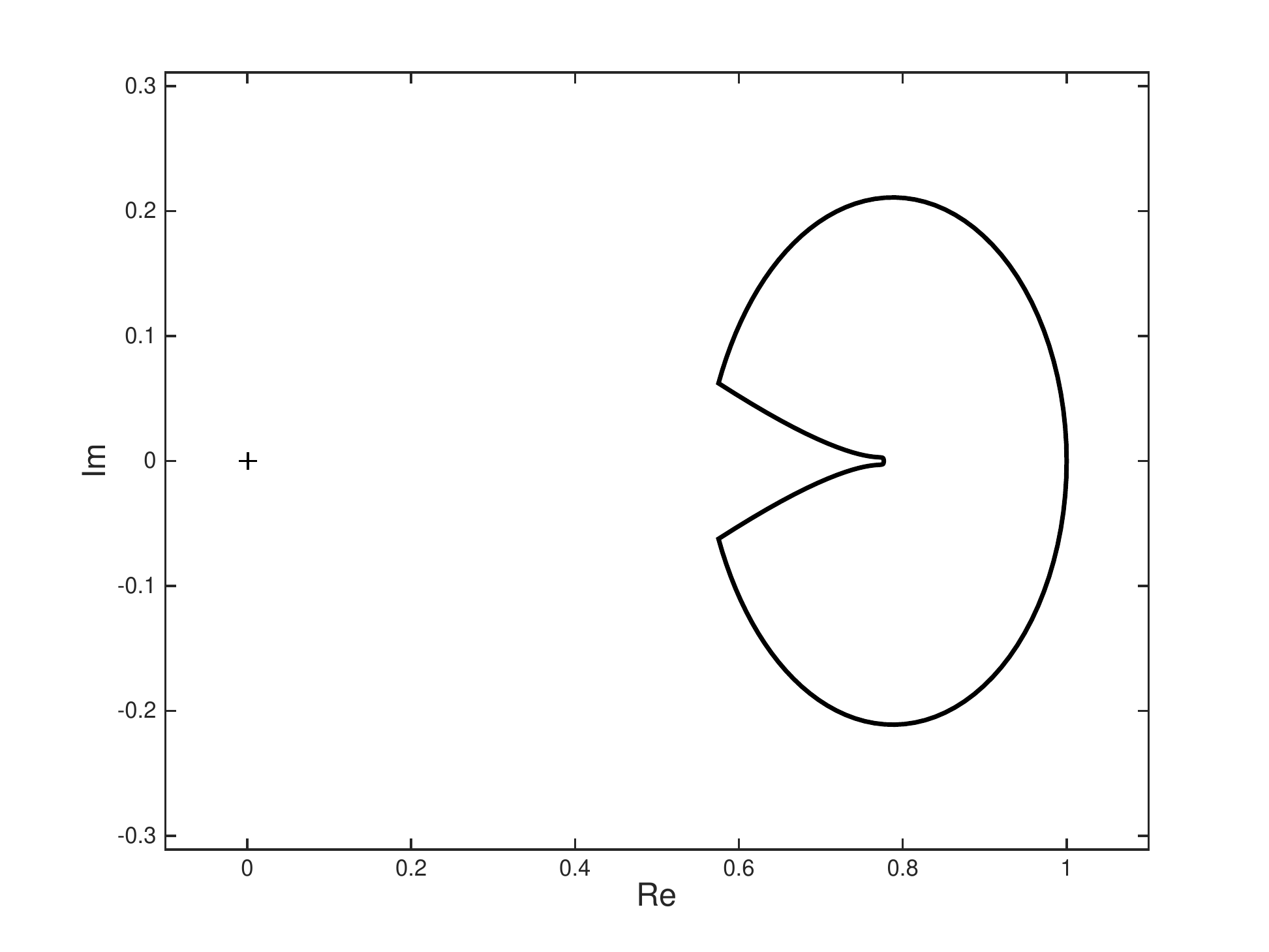} & \includegraphics[width=8.25cm]{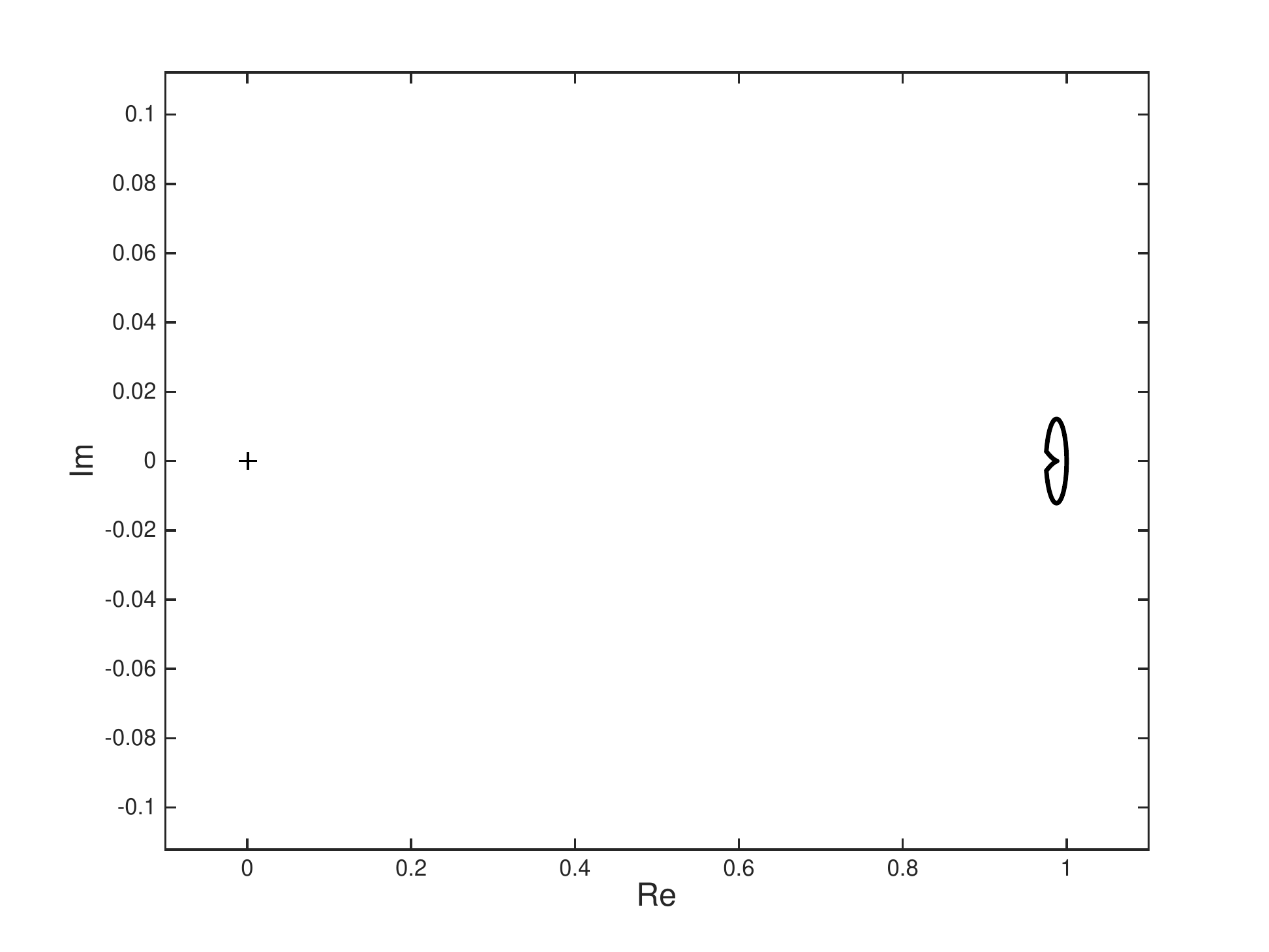} \\ 
(c) & (d)
\end{array}$
\end{center}
\caption{Evans function output for a monatomic gas along the semi-circular contour.  Toggling high and low values of $\breve\xi$ together with strong and weak shocks strengths.  In $(a)$, we have the strong-shock low-$\breve\xi$ case with $u_\sp=0.27$ and $\breve\xi= 0.025$.  This corresponds to a radius of $r^*=250.8$ (for the high frequency bounds in $|\lambda|$) and a $\xi=0.3960$.  In $(b)$, we have the weak-shock high-$\breve\xi$ case with $u_\sp=0.75$ and $\breve\xi= 0.025$.  This corresponds to $r^*=53.6$ and a $\xi=0.1830$.   In $(c)$, we have the strong-shock high-$\breve\xi$ case with $u_\sp=0.27$ and $\breve\xi= 0.95$.  This corresponds to $r^*=16.3$ and a $\xi=12.2778$.  Finally in $(d)$, we have the weak-shock high-$\breve\xi$ case with $u_\sp=0.75$ and $\breve\xi= 0.95$.  This corresponds to $r^*=4.3$ and $\xi=6.2770$.  The winding numbers in $(b)$--$(d)$ are zero by inspection.  The case $(a)$ also has winding number zero, but this is not obvious given the amount of winding and unwinding around the origin.}
\label{fig:large-small-shock-xi}
\end{figure}

\subsection{Numerical results}

We performed three Evans function batch jobs for the monatomic gas case, and then repeated our efforts for the diatomic case (six batch jobs in total).  Each job was with pseudo-Lagrange coordinates but with different Evans function formulations, namely standard balanced flux formulation, the no-radial balanced flux formulation, and the modified balanced flux formulation.  Each job was carried out via the continuous orthogonalization method described above for parameter pairs $(u_\sp,\breve\xi)$, where

\begin{equation}
\label{uprange_monatomic}
\begin{aligned}
u_\sp &\in \{0.25, 0.26, 0.27, 0.28, 0.29, 0.30, 0.35, 0.40, 0.45, 0.50, \\
 & \qquad 0.55, 0.60, 0.65, 0.70, 0.75, 0.80, 0.85, 0.90, 0.95\}
\end{aligned}
\end{equation}
for the monatomic case,
\begin{equation}
\label{uprange_diatomic}
\begin{aligned}
u_\sp &\in \{0.167, 0.17, 0.18, 0.19, 0.20, 0.25, 0.30, 0.35, 0.40, 0.45, \\
 & \qquad 0.50, 0.55, 0.60, 0.65, 0.70, 0.75, 0.80, 0.85, 0.90, 0.95\}
\end{aligned}
\end{equation}
for the diatomic case, and for both cases, we used
\[
\breve\xi \in \{0.005,0.01,0.15,0.2\} \cup \{0.025,0.05,\ldots,0.95,0.975,1.0\}.
\]
This corresponds to 40 equally spaced values in $\breve\xi$ on the domain $[0.025,1]$, together with an additional 4 equally spaced values for a refined grid on $[0.005,0.025]$ to better study the behavior for small values of $\breve\xi$ closer to zero.  For each of the 44 values of $\breve\xi$ and $19-20$ values of $v_+$ (or $45\times 19 = 836$ contours for each of the  three monatomic batch jobs and $45\times 20=900$ contours for each of the three diatomic batch jobs), the contour radius was set taken to be $1.1$ times the high-frequency bound $r^* := r^*(u_\sp, \xi)$ as defined in \eqref{hfbound}.  This extra $10\%$ was taken for good measure.  For a physical measure of the associated shock strengths, we match in Tables \ref{tbl:mach-numbers}--\ref{tbl:mach-numbers2} of Appendix \ref{mach} these values of $u_\sp$ with the Mach numbers of the corresponding shocks.  Note that $u_\sp=1/4$ corresponds to an infinite Mach number for the monatomic case and $u_\sp=1/6$ for the diatomic case.

For each of the six batch jobs, the contours consisted of at least $50$ radial points in the first quadrant and $50$ points along the positive imaginary axis.  Exploiting conjugate symmetry, we reflected the output to capture the half of the contour in the fourth quadrant and to produce the figures and winding number calculations.  We note that for the first and third batch jobs (for both monatomic and diatomic cases), we used a contour-refinement feature in {\sc STABLAB} that adds contour points as needed to achieve a pre-determined tolerance of
\[
\dfrac{\Delta D(\lambda)}{D(\lambda)} \leq \dfrac{1}{5} = 0.20.
\]
This guarantees, by Rouch\'e's theorem that our winding number calculations are accurate.  We note that in the monatomic case, that there were contours with as many as 230 points required to achieve the desired tolerance.  In the diatomic case, there were contours with as many as 474 points.

\begin{figure}[t]
\begin{center}$
\begin{array}{cc}
\includegraphics[width=8.25cm]{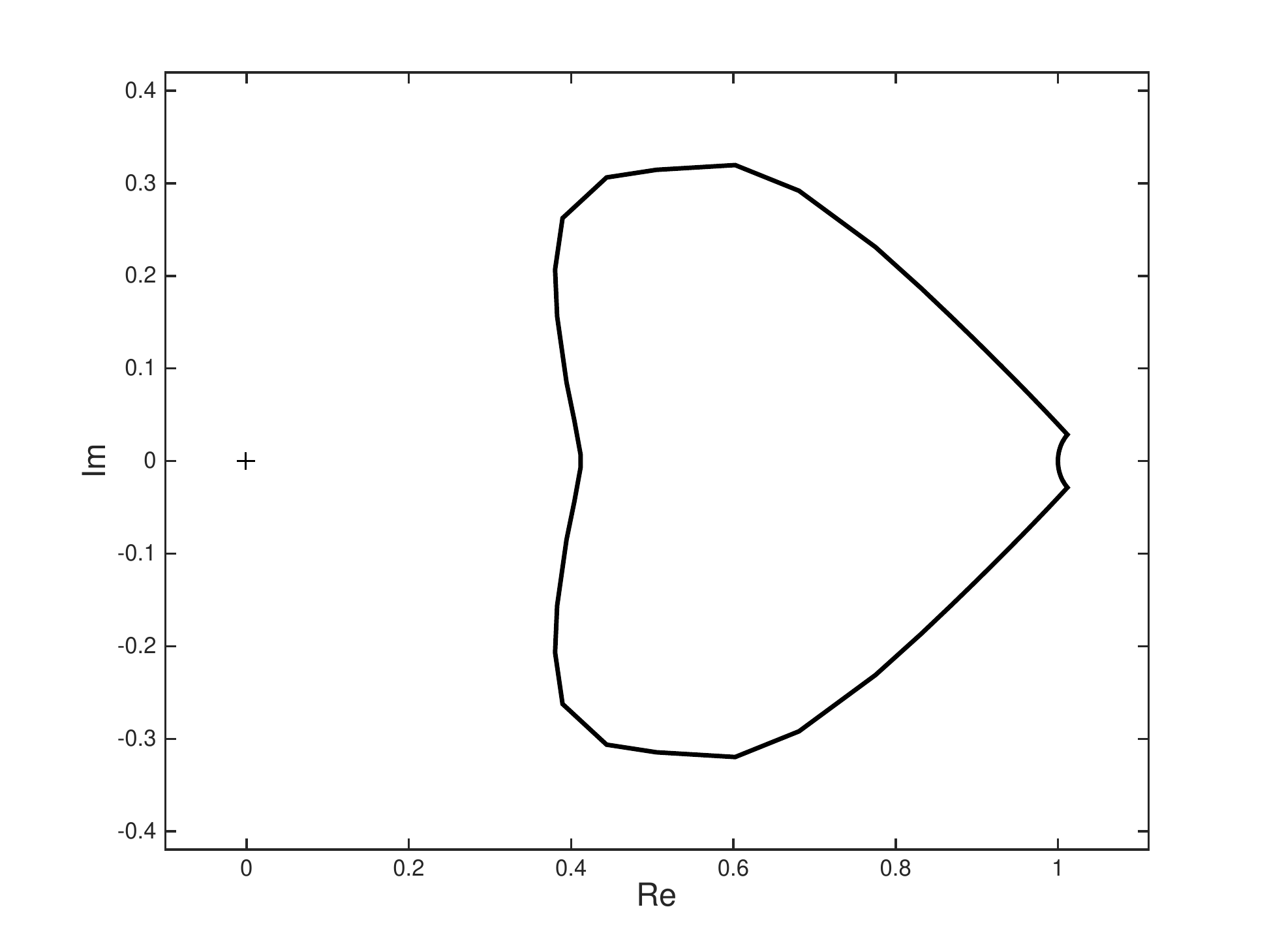} & \includegraphics[width=8.25cm]{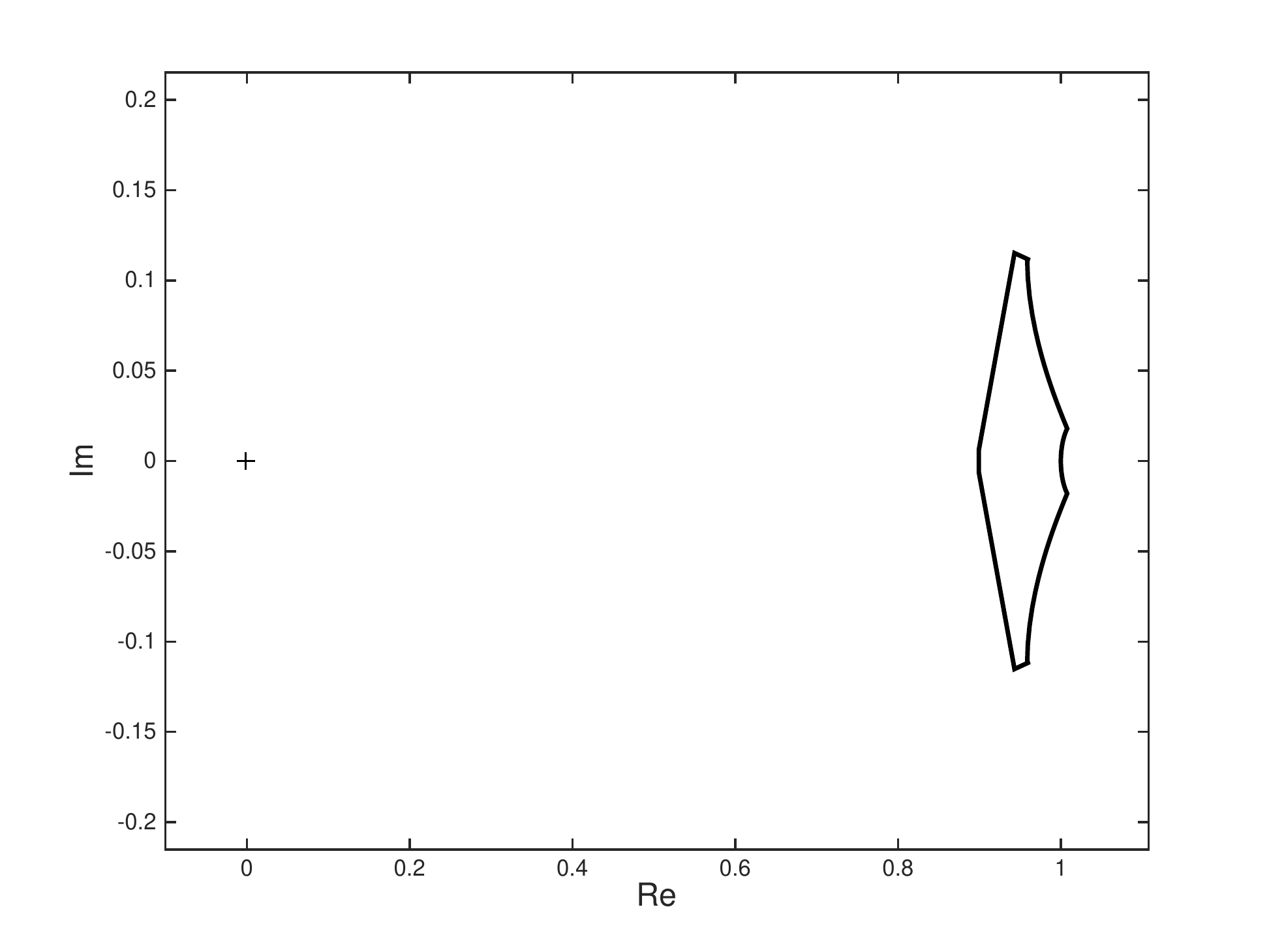} \\ 
(a) & (b)\\
\includegraphics[width=8.25cm]{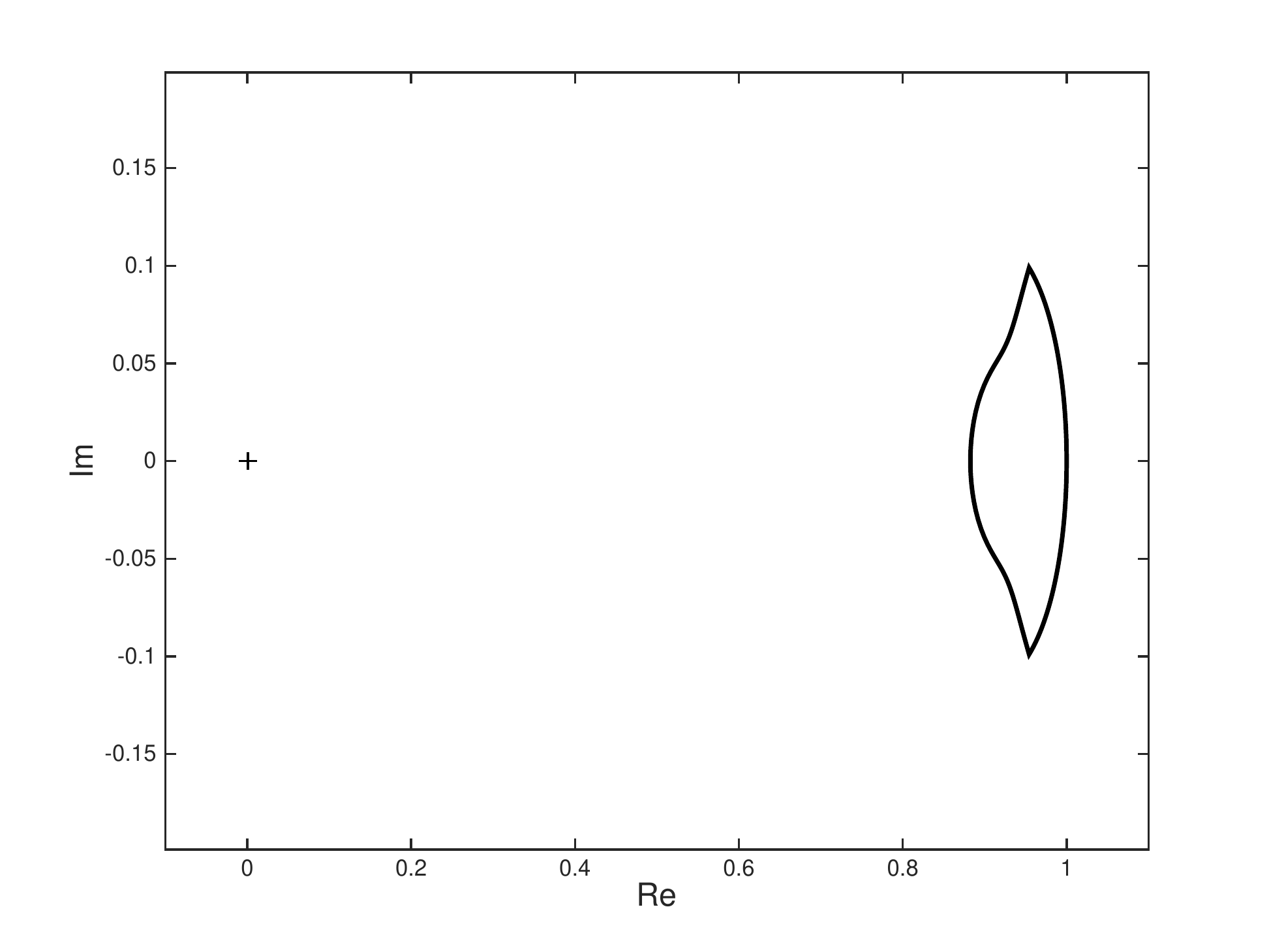} & \includegraphics[width=8.25cm]{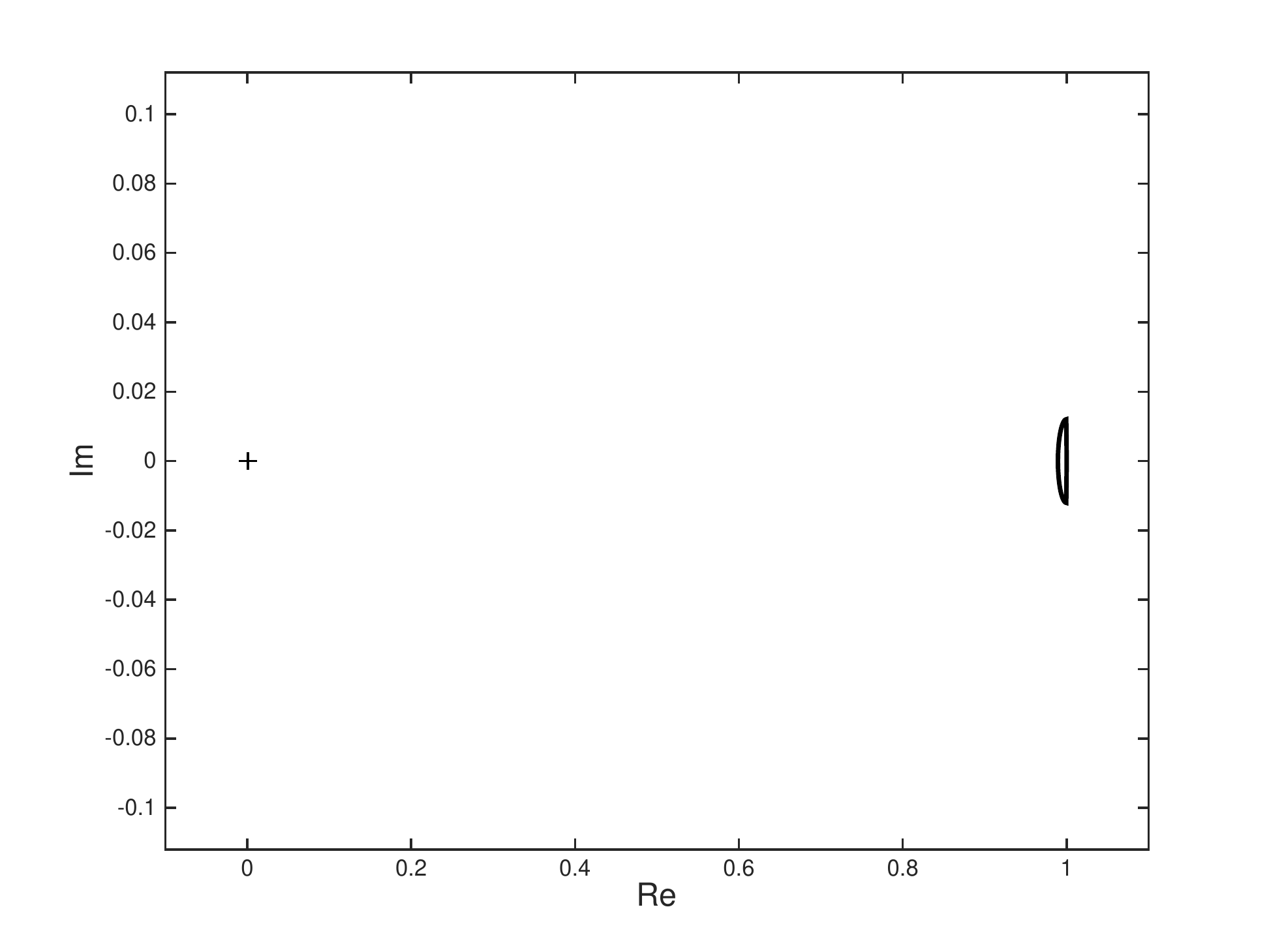} \\ 
(c) & (d)
\end{array}$
\end{center}
\caption{The no-radial Evans function output for the same four contours and parameters as Figure \ref{fig:large-small-shock-xi}.  We see visually that all of these have winding number zero.}
\label{fig:no-radial-winding}
\end{figure}

One of the overarching themes observed in the results of our first batch job(s) (for each of the monatomic and diatomic cases) is that the larger the radius of the semi-circular contour, the more the Evans function winds and unwinds around the origin. (Note that the parameter values for which there is a large amount of winding are precisely those for which mesh refinement was used.)  Therefore, in the strong-shock regime(s), when $\breve\xi$ is close to zero, the high-frequency bounds are also large and thus do a lot of winding; see Figure  \ref{fig:hfb}.  If the shock is small or $\breve\xi$ is larger, then the radius is smaller and the winding is minimal.  As an example, we examine the output for four test cases in Figure \ref{fig:large-small-shock-xi}.   Specifically we compute the Evans function in standard pseudo-Lagrangian balanced flux coordinates toggling between strong and weak shocks as well as near-zero and near-unity values of $\breve\xi$.  Note that with the exception of the strong shock case, where $\breve\xi$ is close to zero, that the winding numbers are zero by inspection.  In all cases, but especially when shock is strong and $\breve\xi$ is close to zero, we computed the winding numbers numerically to confirm that they are zero as expected.

\begin{figure}[t]
\begin{center}$
\begin{array}{cc}
\includegraphics[width=8.25cm]{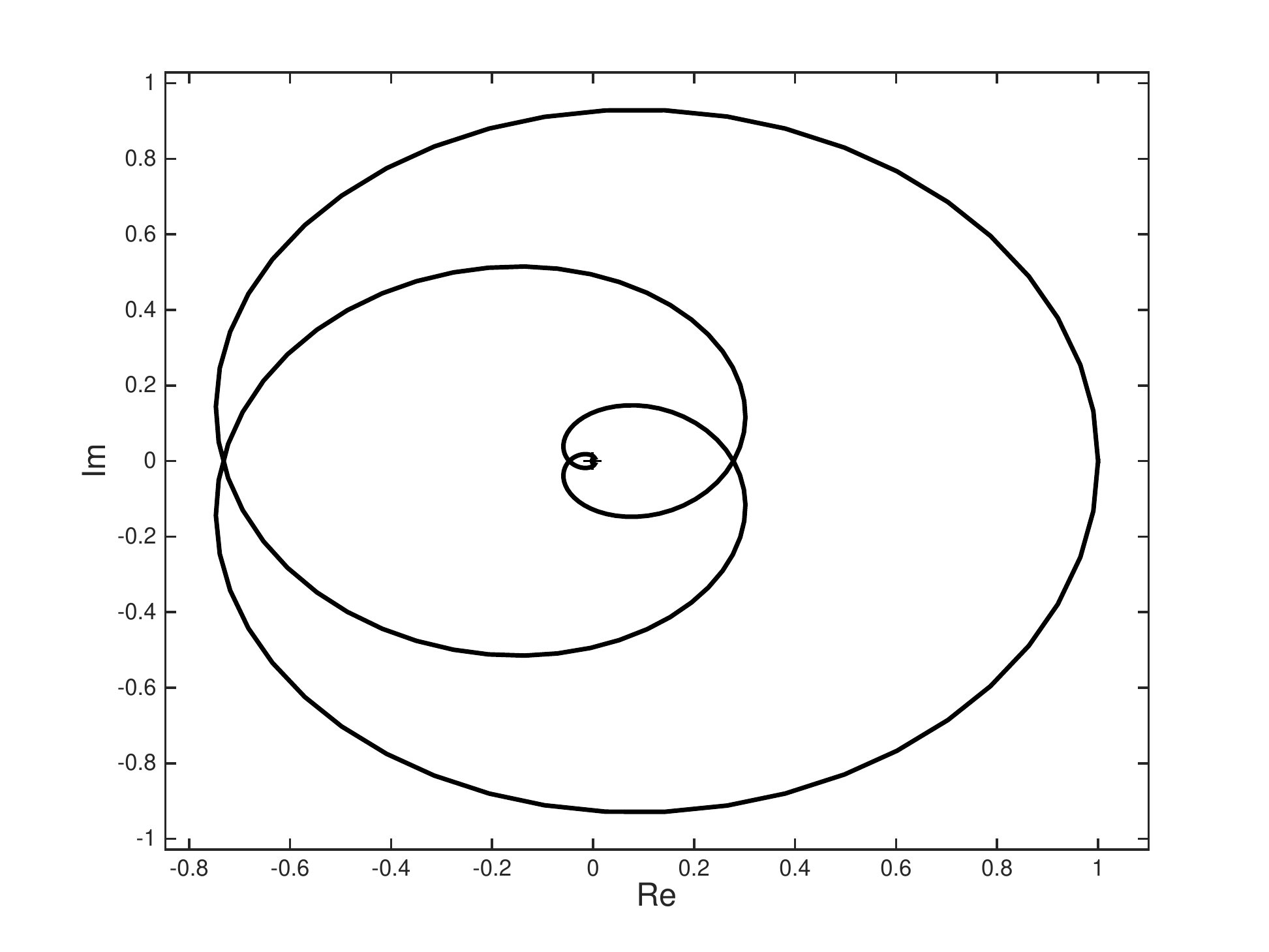} & \includegraphics[width=8.25cm]{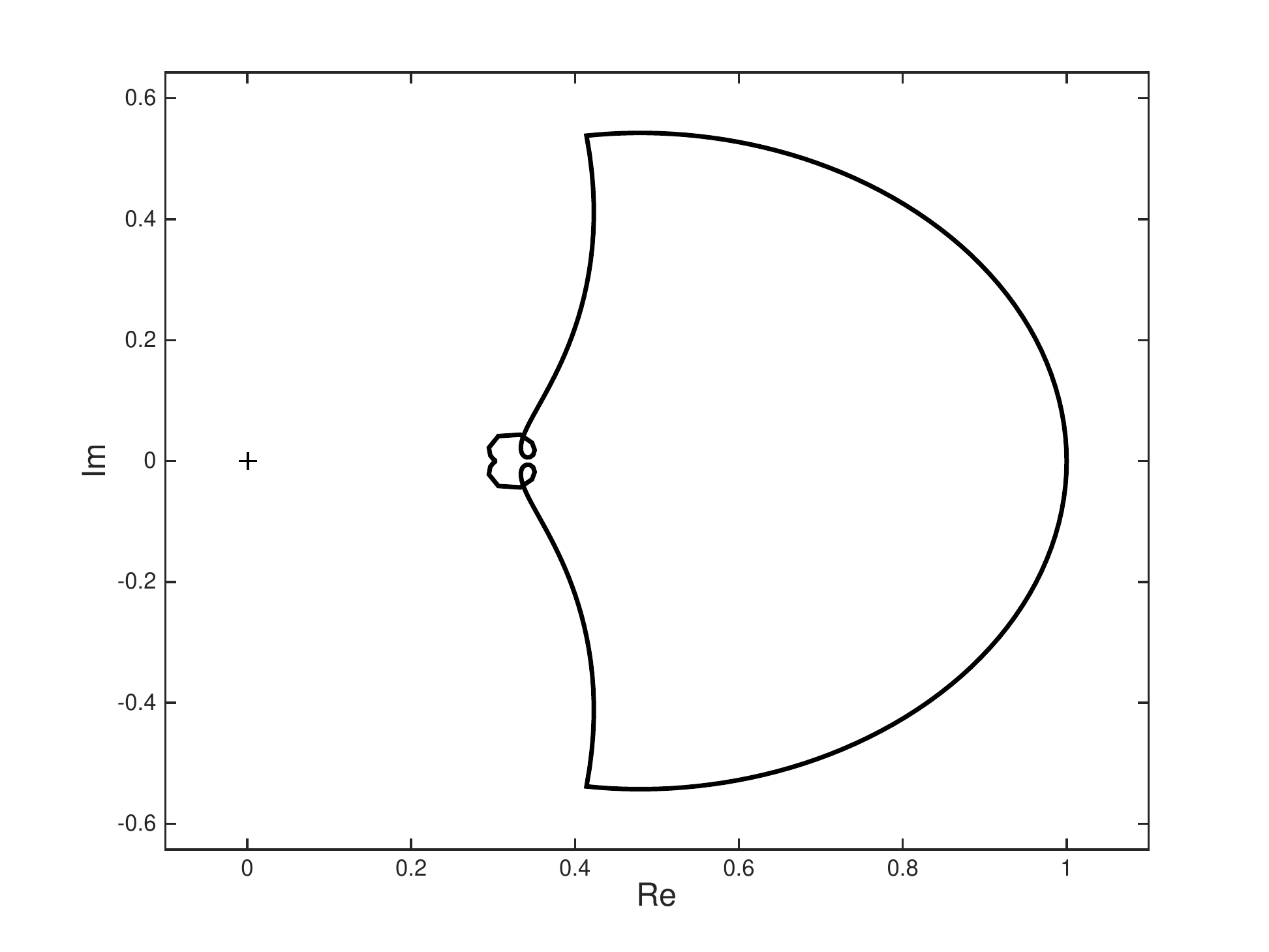} \\ 
(a) & (b)\\
\includegraphics[width=8.25cm]{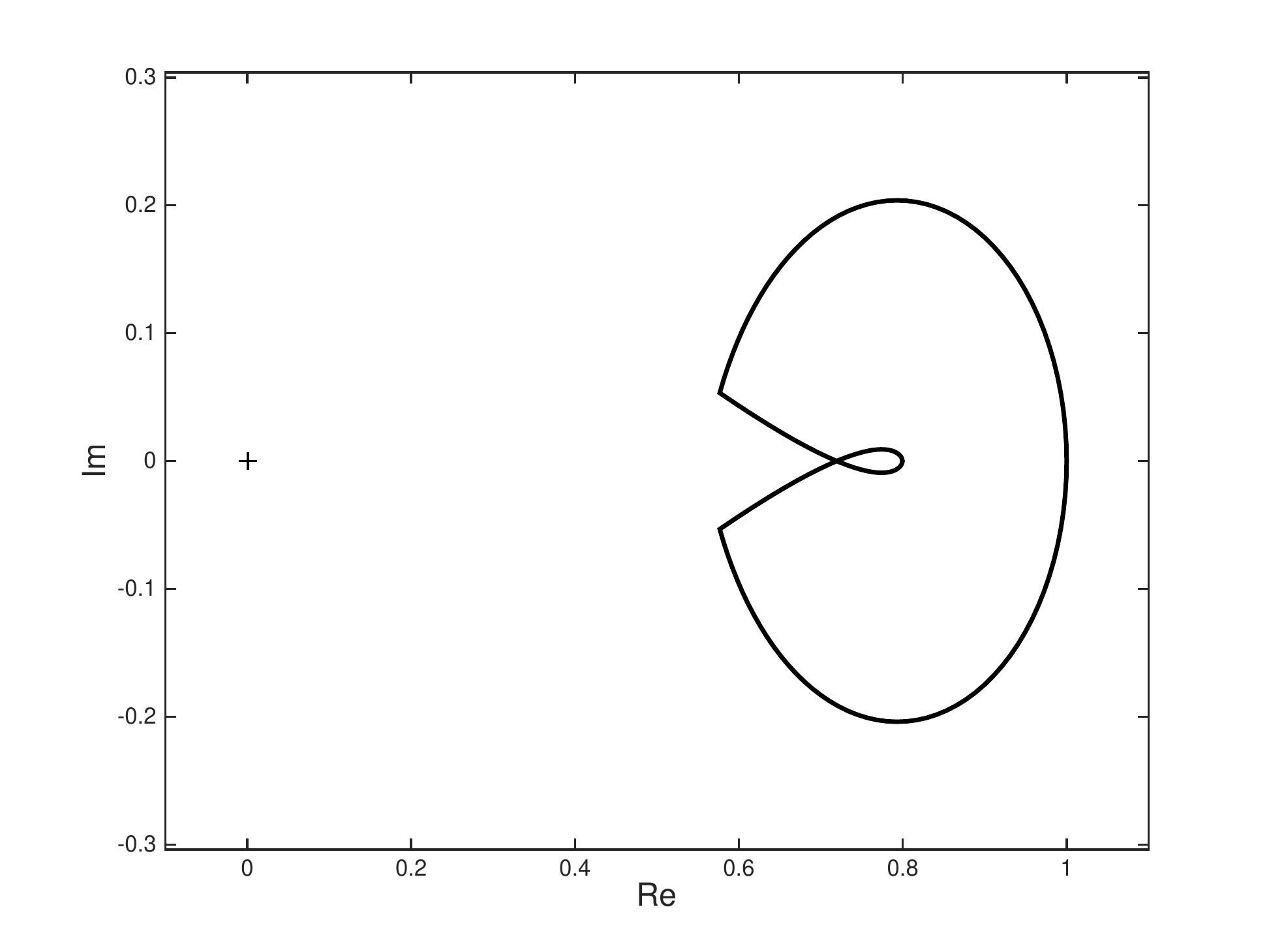} & \includegraphics[width=8.25cm]{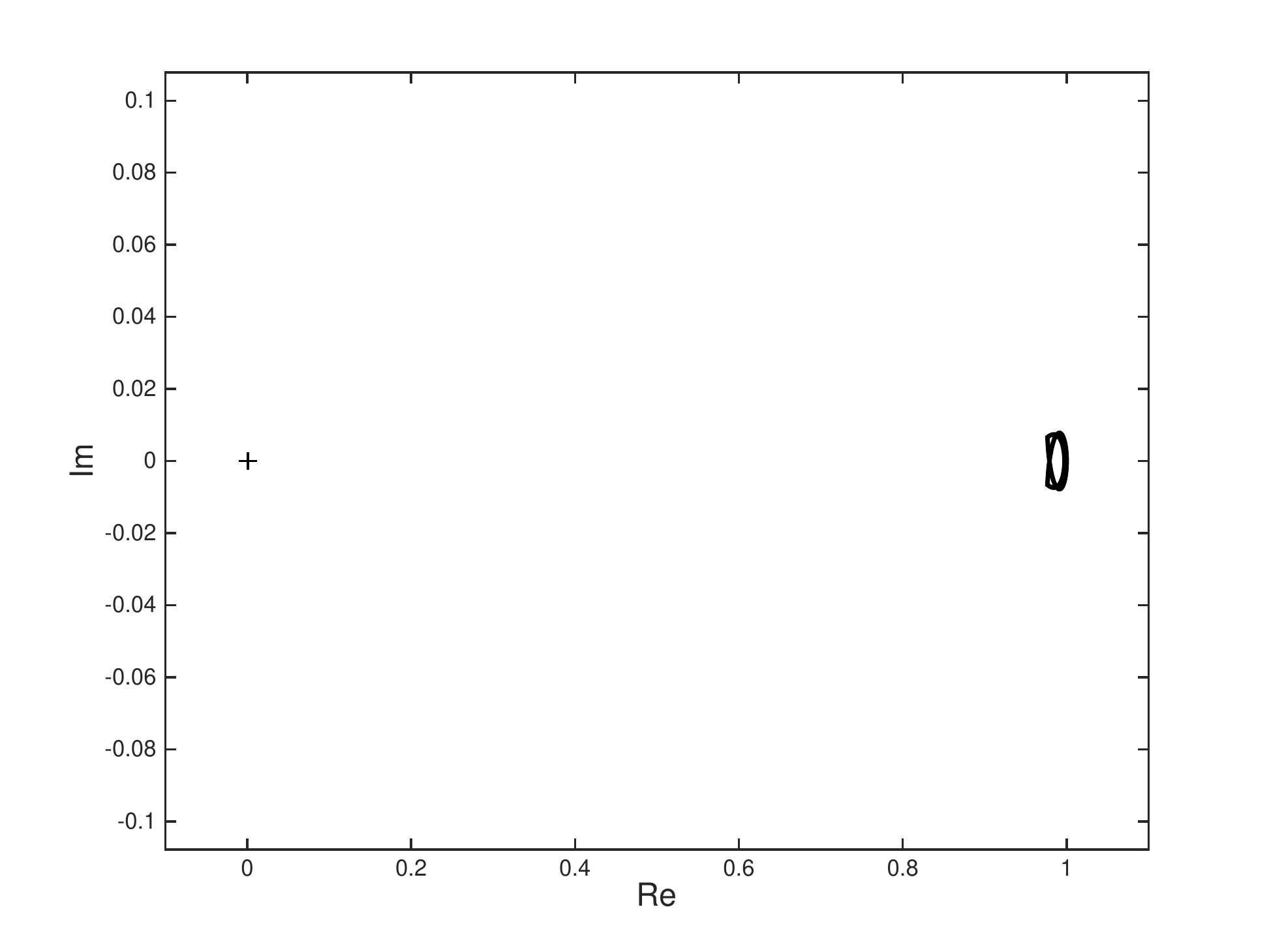} \\ 
(c) & (d)
\end{array}$
\end{center}
\caption{The Evans function output using the modified balanced flux formulation for the same four contours and parameters as Figures \ref{fig:large-small-shock-xi}--\ref{fig:no-radial-winding}.}
\label{fig:modified-balanced}
\end{figure}

As an additional verification, the second Evans function batch job (computed for each of the monatomic and diatomic cases) was performed using the no-radial formulation\footnote{In this formulation of the method, the radial factor $\gamma$ is not evolved with $x_1$ but kept constant at its initializing value at $\pm\infty$; see \cite{STABLAB} for further details.}.  These runs produced non-windy albeit non-analytic output for which all runs again have a zero winding number, by both immediate inspection and direct computation.  In particular, in Figure \ref{fig:no-radial-winding}, we compute the same Evans function contours as in Figure \ref{fig:large-small-shock-xi} except that we use the no-radial Evans function computation to demonstrate the difference between the winding of the standard balanced flux case and the no-radial case.

The third batch job(s) (again for each of the monatomic and diatomic cases) were carried out using the modified balanced pseudo-Lagrange formulation.  This is the only truly analytic formulation that we carried out: the first batch job are not analytic on the imaginary axis because $|\lambda|$ changes along the contour and thus \eqref{eq:scale_factor} changes as well.  In Figure \ref{fig:modified-balanced}, we see examples of output from this method.  Notice the similarity between this output and that of the standard balanced flux method in Figure \ref{fig:large-small-shock-xi}.

\begin{figure}[t]
\begin{center}$
\begin{array}{cc}
\includegraphics[width=8.25cm]{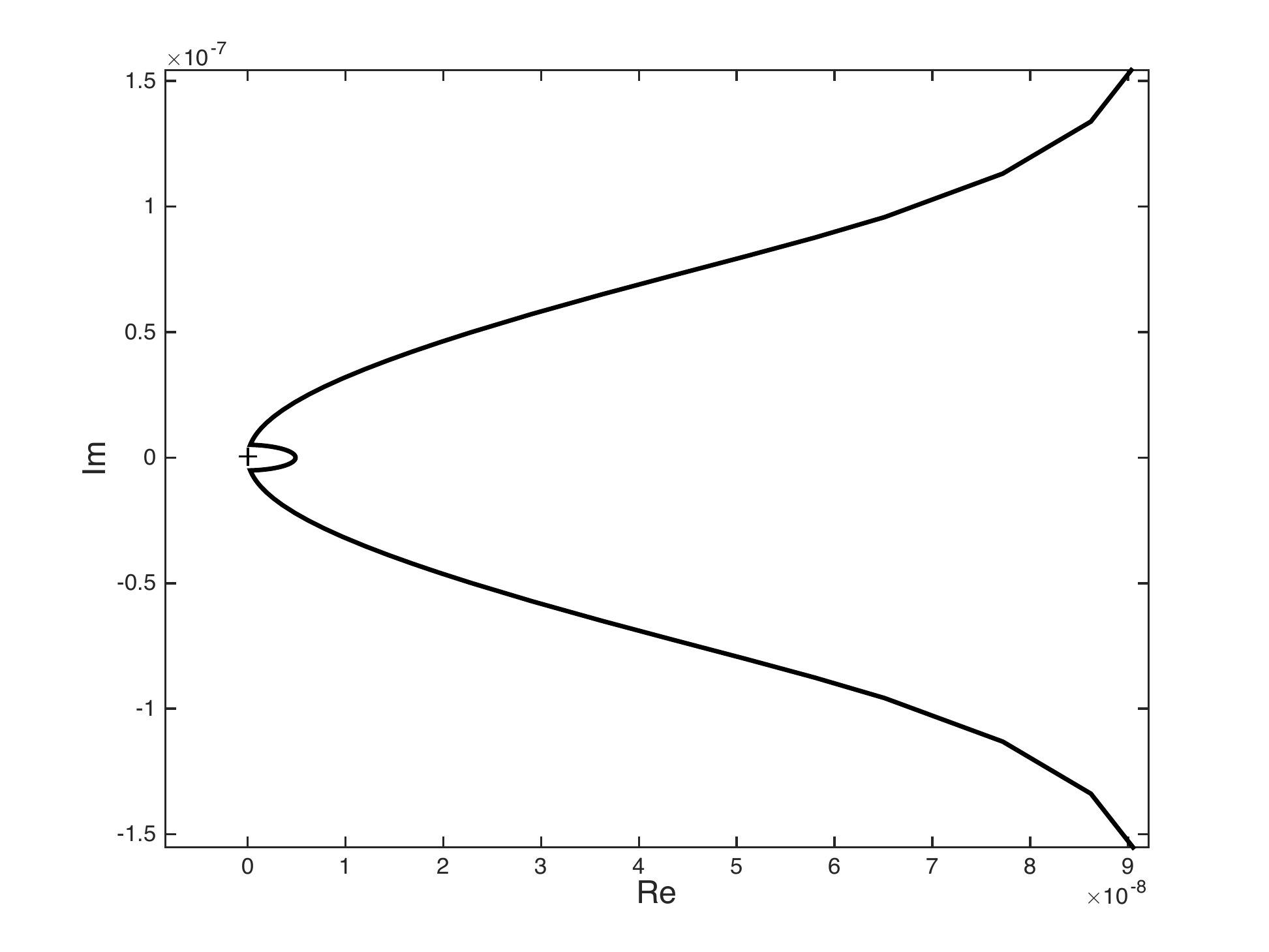} & \includegraphics[width=8.25cm]{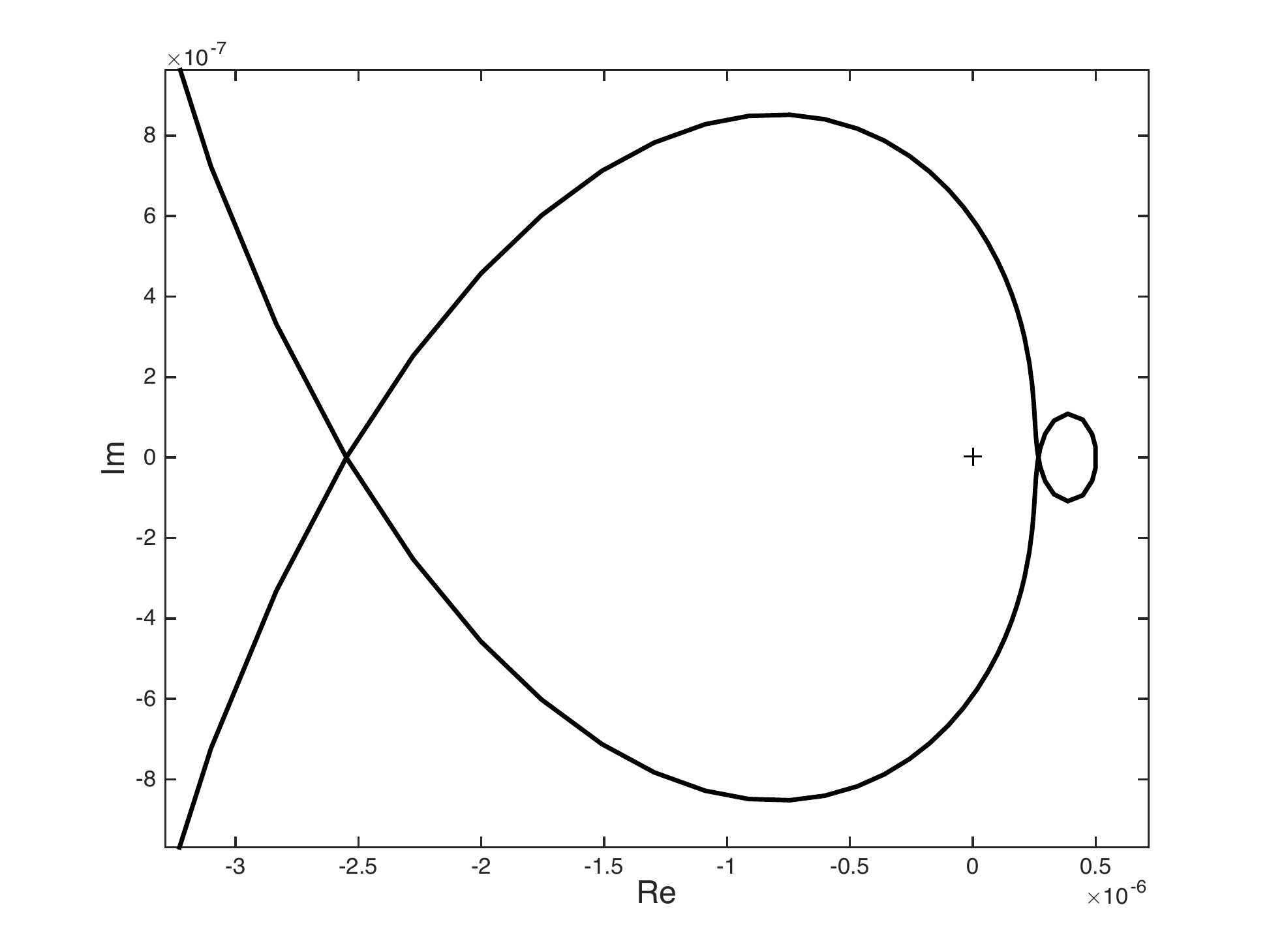} \\ 
(a) & (b)
\end{array}$
\end{center}
\caption{Examination of the output generated for the $\breve\xi=0$ case when zoomed near the origin.  In the $(a)$ standard balanced flux formulation and the $(b)$ modified balanced flux formulation.  Notice that the notch removed near the origin was necessary in the first case to prevent the output from crossing through zero, whereas no such adaptation was needed in the second case.}
\label{fig:xi-zero-modified}
\end{figure}

We remark that in the first and third batch jobs (for both monotone and diatomic cases) we also computed contours for $\xi = 0$ (for all the shock strengths).  With the standard balanced flux formulation, we have a root at $\lambda=0$ and thus winding numbers can't be computed for exactly semicircular contours.  Thus, we removed a notch on the contour around the origin so that the output wouldn't pass through it.  Since in the low-frequency limit we already have demonstrated stability, we are able to do this without affecting the integrity of our results.  With the modified balanced flux formulation, the notch need not be removed since $\lambda=0$ is not a root in this case.  In Figure \ref{fig:xi-zero-modified}, we see the $\breve\xi=0$ output zoomed near the origin for a strong shock satisfying $u_\sp=0.27$.  Notice that in the standard balanced flux version, the winding number is zero since the removed notch prevents the contour from touching the origin.  We note that in all cases, for both formulations, the winding number was found to be zero, signifying spectral stability.  We remark that spectral stability is known for $\breve\xi=0$, since this case corresponds to the 1D case studied in \cite{HLyZ}.

\subsection{Conclusions}  

There are no unstable imaginary eigenvalues for $(\xi,\lambda)$ in the complement of region \eqref{hfbound}, except possibly on a ball of radius $0.08$ about the origin for monatomic and $0.11$ for diatomic, where these values are the minimum value of our $\xi$-mesh, below which $|\lambda|$ of a comparable order require additional refinement for good resolution.  Each ``unchecked'' region is contained in a ball of sufficiently small radius ($0.1$ for monatomic and $0.12$ for diatomic) about the origin on which we have already excluded imaginary roots via our low-frequency analysis.


\section{Summary/final conclusions}
\label{sec:summary}
Combining the numerical results of Sections \ref{sec:hfb}--\ref{sec:intermediate}, we find for both monatomic and diatomic $\gamma$-law equations of state that there are no pure imaginary roots $\lambda$ of the balanced flux Evans function for any $\xi$ and any Mach number.  Thus, the number of unstable roots for any fixed $\xi$ is independent of Mach number, so that either all shock profiles are stable, or else all are unstable for a given model.  But, on the other hand, we have also shown nonexistence of unstable eigenvalues of any kind for $(\xi, \lambda)$ within the computational domain consisting of the complement of \eqref{hfbound} minus a small ball.  
Moreover, for the special value $u_\sp=.6$, we have shown nonexistence of unstable eigenvalues also for this small deleted ball, leaving only the possibility of high-frequency instability. As the Evans function varies analytically with $\xi$ outside any finite ball, it is thus sufficient to find a single $\xi$ for which there exist no unstable eigenvalues, nonexistence for all $\xi$ then following by homotopy.  Recalling that nonexistence for $\xi=0$ follows from the one-dimensional study of \cite{HLyZ}, we obtain stability for $u_\sp=.6$. Alternatively, one might appeal to an asymptotic argument like that  of \cite{FS2} in the artificial viscosity case to conclude stability in the small-amplitude, or constant-state limit $M\to1$.
That some shocks are stable implies that all shocks are stable for Mach number $M\geq 1.035$, completing the analysis.  As noted in the introduction, the $M\to 1$ limit should be treatable analytically by methods similar to those of \cite{FS2,PZ}.

\appendix

\section{Numerical tracking algorithm}\label{tracking}

In this appendix, we collect the tools used to track invariant subspaces in the high-frequency regime.  The basic approach is by now standard \cite{GZ,ZH, MZ1,Z3}; the point here is to provide an efficient numerical realization with minimal required analytical preparation.


\subsection{Quantitative tracking lemma}
\label{quant}
Consider an asymptotically constant approximately block-diagonal system of the form 
\beq
\label{Msys}
W'=(\mat{N}+\Theta)W\,,
\eeq
where 
\begin{equation}
\label{tracksys}
 W=\begin{pmatrix}W_\sm\\ W_\sp\end{pmatrix}, \quad
\mat{N}= \begin{pmatrix}\mat{N}_\sm& 0\\ 0 & \mat{N}_\sp \end{pmatrix},
\quad
\Theta= \begin{pmatrix}\Theta_{\sm\sm}& \Theta_{\sm\sp}\\ 
\Theta_{\sp\sm} & \Theta_{\sp\sp} \end{pmatrix},
\end{equation}
and there exist constants $c_\spm$ such that 
$\Re \mat{N}_-\le c_\sm\mat{I}$ and $\Re \mat{N}_\sp\ge c_\sp\mat{I}$ and  
\begin{equation}
\label{numrange}
c_\sp-c_\sm \ge \delta(x)>0\,,
\end{equation}
where $\Re \mat{N}:=\dfrac{\mat{N}+\mat{N}^*}{2}$ denotes the symmetric part of a matrix $\mat{N}$.  Denote by
\begin{equation}
\label{zeta+-}
\begin{aligned}
\zeta_\spm(x)&:= \frac{ \delta -\|\Theta_{\sm\sm}\|-\|\Theta_{\sp\sp}\|}{2\|\Theta_{\sp\sm}\|}
\pm \sqrt{ \left(\frac{ \delta -\|\Theta_{\sm\sm}\|-\|\Theta_{\sp\sp}\|}{2\|\Theta_{\sp\sm}\|}\right)^2 -\frac{\|\Theta_{\sm\sp}\|}{\|\Theta_{\sp\sm}\|} }
\end{aligned}
\end{equation}
the roots of
\begin{equation}
\label{quad}
P(\zeta,x):= \Big(-\delta +\|\Theta_{\sm\sm}\|+\|\Theta_{\sp\sp}\|\Big)\zeta+\|\Theta_{\sm\sp}\|+ \|\Theta_{\sp\sm}\|\zeta^2=0.
\end{equation}

\begin{lemma}
\label{tracklem}
Suppose that
\begin{equation}
\label{rootcond}
\delta> \|\Theta_{\sm\sm}\|+\|\Theta_{\sp\sp}\| + 2\sqrt{\|\Theta_{\sm\sp}\|\|\Theta_{\sp\sm}\|}
\end{equation}
for a system \eqref{Msys}--\eqref{tracksys}, satisfying \eqref{numrange}.
Then,
\begin{enumerate}
\item[(i)] $0<\zeta_\sm<\zeta_\sp$;
\item[(ii)] the invariant subspaces of the limiting coefficient
matrices $(\mat{N}+\Theta)(\pm \infty)$ are contained in distinct cones
\[
\Omega_\sm=\{\|W_\sm\|/\|W_\sp\|\le \zeta_\sm\}, \quad \Omega_\sp=\{\|W_\sm\|/\|W_\sp\|\ge \zeta_\sp\}\,;
\]
\item[(iii)] denoting by $S^+$ the total eigenspace of $(\mat{N}+\Theta)(+\infty)$ contained in $\Omega_\sp(+\infty)$ and $U_\sm$ the total eigenspace of $(\mat{N}+\Theta)(-\infty)$ contained in $\Omega_\sm(-\infty)$, the manifolds of solutions of \eqref{tracksys} asymptotic to $S_\sp$ at $x=+\infty$ $U_\sm$ at $x=-\infty$ are separated for all $x$, lying in $\Omega_\sp$ and $\Omega_\sm$ respectively.
\end{enumerate}
In particular, there exist no solutions of \eqref{tracksys} asymptotic to $S_\sp$ at $+\infty$ and to $U_\sm$ at $- \infty$.
\end{lemma}

\begin{remark}
\label{noevans}
In the case $c_\sm<0$, $c_\sp>0$, $S_\sp$ and $U_\sm$ correspond to the stable subspace at $+\infty$ and the unstable subspace at $-\infty$ of the limiting constant-coefficient matrices, and we may conclude nonexistence of decaying solutions of \eqref{tracksys}, or nonvanishing of the associated Evans function.
\end{remark}

\begin{proof}
From \eqref{tracksys}, we obtain readily
\begin{subequations}
\begin{align}
\|W_\sm\|'&\leq  c_\sm\|W_\sm\| + \|\Theta_{\sm\sm}\|\|W_\sm\|+ \|\Theta_{\sm\sp}\|\|W_\sp\|,\\
\|W_\sp\|'&\geq  c_\sp\|W_\sp\| - \|\Theta_{\sp\sm}\|\|W_\sm\| - \|\Theta_{\sp\sp}\|\|W_\sp\|\,,
\end{align}
\end{subequations}
from which, defining $\zeta:= |W_\sm|/|W_\sp|$, 
we obtain by a straightforward computation the Riccati equation
$\zeta'\le P(\zeta,x)$.

Consulting \eqref{quad}, we see that $\zeta'<0$ on the interval
$\zeta_-<\zeta<\zeta_+$,
whence $\Omega_-:=\{\zeta\le \zeta_-\}$ is an invariant region
under the forward flow of \eqref{tracksys}; moreover, this region
is exponentially attracting for $\zeta < \zeta_+$.
A symmetric argument yields that $\Omega_+:=\{\zeta \ge \zeta_+\}$ is
invariant under the backward flow of \eqref{tracksys}, and exponentially
attracting for $\zeta >\zeta_-$.
Specializing these observations to the constant-coefficient limiting
systems at $x=-\infty$ and $x=+\infty$, we find that the invariant 
subspaces of the limiting coefficient matrices must lie in 
$\Omega_-$ or $\Omega_+$.  (This is immediate in the diagonalizable
case; the general case follows by a limiting argument.)

By  forward (resp. backward) invariance of $\Omega_-$ (resp. $\Omega_+$),
under the full, variable-coefficient flow, we thus find that the manifold
of solutions initiated along $U^+$ at $x=-\infty$ 
lies for all $x$ in $\Omega_+$ while the manifold of
solutions initiated in $S^+$ at $x=+\infty$ lies for all $x$ in $\Omega_-$.
Since $\Omega_-$ and $\Omega_+$ are distinct,
we may conclude that under condition \eqref{rootcond}
there are no solutions asymptotic to both $U_\sm$ and $S_\sp$.
\end{proof}

\subsection{Numerical block-diagonalization}
\label{blockdiag}
We now show how to compute the numerical block diagonalization given in \eqref{eq:block-form}.  We need to choose $R$ to be smooth in $x_1$ so that the $LR'$ term is well defined and well behaved.

Assume that $\mat{D}_{1/2}$ has complementary eigenspaces $\Sigma_\pm$, corresponding to its stable and unstable subspaces, with spectral gap
\begin{equation}
\label{gapc}
\min \sigma(\Sigma_+)\ge d_+, \quad \max \sigma(\Sigma_-)\ge d_-,
\quad d_+-d_->0.
\end{equation}
We reduce numerically to block-diagonal form.  Denoting by $\Pi_\pm$ the associated analytic spectral projections of $\mat{D}_{1/2}$, we may define bases $R_\pm$ of basis vectors, arranged in columns, $R = \begin{pmatrix} R_-& R_+\end{pmatrix}$, as solutions of Kato's ODE \cite{Kato}
\begin{equation}
\label{kato}
R_\pm'=(\Pi \, \Pi'-\Pi' \,\Pi)_\pm R_\pm,
\end{equation}
initializing at $x=-\infty$ by some choice of orthonormal bases for each subspace.  The solution of \eqref{kato} may be carried out numerically as described in \cite{BrZ, HuZ2, BHRZ}.  Setting $L:=R^{-1}$ completes the computation and yields  \eqref{eq:block-form}.

\subsection{Numerical symmetrization}
We conclude by prescribing a ``symmetrizing'' transformation converting the spectral gap condition 
\eqref{gapc} into a numerical range condition \eqref{numrange} if this is necessary: that is, if \eqref{numrange} does not hold already after block-diagonalization.\footnote{As noted earlier, in the numerical investigations of this paper, this step is not needed.}

\begin{lemma}[Lyapunov's Lemma \cite{BS}]
\label{lem:lyapunov}
Suppose that $\mat{N}$ is a stable (unstable) matrix, i.e., 
\[
\Re\sigma(\mat{N})<0\,,\;\;(>0)\,.
\]
Then, there exists a $\mat{P}$ positive and symmetric such that 
\beq
\Re(\mat{P}\mat{N})=\frac{1}{2}(\mat{P}\mat{N}+(\mat{P}\mat{N})^*)<0\,,\;\;(>0)\,.
\eeq
The matrix $\mat{P}$ may be characterized as a solution of the Sylvester equation
\[
\mat{P}\mat{N}+\mat{N}^*\mat{P}=-\mat{Q}\,,
\]
for some positive (negative) symmetric matrix $\mat{Q}$. 
\end{lemma}

Now, suppose that $\mat{N}$ is block diagonal with the form 
\beq
\mat{N}=\bpm \mat{N}_\sp & 0 \\ 0 & \mat{N}_\sm \epm\,,
\eeq
and that there exist $\delta_\spm$ describing a spectral gap between the blocks: 
\beq
\Re\sigma(\mat{N}_\sp)\geq\delta_\sp>\delta_\sm\geq\Re\sigma(\mat{N}_\sm)\,.
\eeq
Then we shall use the Lyapunov lemma to identify $\mat{P}_\spm$ positive and symmetric so that
\beq\label{eq:newgap}
\Re(\mat{P}_\sp^{1/2}\mat{N}_\sp\mat{P}_\sp^{-1/2})
\geq c_\sp>c_\sm\geq 
\Re(\mat{P}_\sm^{1/2}\mat{N}_\sm\mat{P}_\sm^{-1/2})
\eeq
for $c_\spm$ chosen to satisfy $\delta_\sp\geq c_\sp>c_\sm\geq\delta_\sm$. To see this, simply observe that $\Re\sigma(\mat{N}_\sp-c_\sp\mat{I})>0$, whence Lemma \ref{lem:lyapunov} implies the existence of a positive matrix $\mat{P}_\sp$ such that 
\beq\label{eq:coord}
\Re(\mat{P}_\sp(\mat{N}_\sp-c_\sp\mat{I}))>0\,.
\eeq
But, evidently, \eqref{eq:coord} is equivalent to $\mat{P}_\sp^{1/2}(\mat{N}_\sp-c_\sp\mat{I})\mat{P}^{-1/2}>0$, and this implies the first inequality in \eqref{eq:newgap}.
Defining now
\beq\label{Mdef}
M_+:= \Re(\mat{P}_\sp^{1/2}\mat{N}_\sp\mat{P}_\sp^{-1/2}),
\quad
M_-:=\Re(\mat{P}_\sm^{1/2}\mat{N}_\sm\mat{P}_\sm^{-1/2}),
\eeq
$ \Phi:= \mat{P}^{1/2}\Psi \mat{P}^{-1/2}- \mat{P}^{1/2}(\mat{P}^{-1/2})_{x_1}$, and $\mat{P}:={\rm blockdiag}\{\mat{P}_+,\mat{P}_-\}$,
we find that the coordinate change $\mat{P}^{1/2}Y= W$ yields, finally, a system
of the desired form \eqref{Msys}--\eqref{tracksys} satisfying \eqref{numrange}.
The solution of Lyapunov's equation is a standard numerical linear algebra problem \cite{BS}.

\subsection{Norm Bounds for the Crude Tracking Estimate}

\begin{lemma}
\label{lem:matrix-norm}
For $\breve r>0$, we have that $\left\| \begin{pmatrix} I & \breve r^{-1/2} I \\ 0 & I \end{pmatrix} \right\|^2 = 1 + \dfrac{1+\sqrt{1+4\breve r}}{2 \breve r}$.  Moreover, when $\breve r\geq 1$, we have $\left\| \begin{pmatrix} I & \breve r^{-1/2} I \\ 0 & I \end{pmatrix} \right\|^2 \leq 1 + \dfrac{2}{\sqrt{\breve r}}$.
\end{lemma}

\begin{lemma}
\label{lem:V}
For $\mat{U}$ given in \eqref{psi}, we have that
\[
\| \mat{U} \|^2 = 1 + \dfrac{\|\psi\|^2 + \|\psi\| \sqrt{\|\psi\|^2 + 4}}{2}.
\]
\end{lemma}

\begin{lemma}
\label{inhom}
For all values of $\hat \rho$, $\breve\lambda$, $\breve r \geq 1$, let $b = \sup \dfrac{\|\beta\|^{-1}\|q\|}{\|\beta^{-1}q\|}$.  We have:
\begin{enumerate}
\item $M_0 := \sup \|( \hat \rho \breve \lambda \breve r^{1/2}\mat{I}_6 + \beta)^{-1}q\| \leq \displaystyle\sup_{0 \leq s \leq (1+b)\|\beta\|} \|(s-i\beta)^{-1}q\|$.
\item $M_1 := \sup \|\hat \rho \breve \lambda \breve r^{1/2} ( \hat \rho \breve \lambda \breve r^{1/2}\mat{I}_6 + \beta)^{-1}\| \leq \displaystyle\min\left\{ 2,\displaystyle\sup_{0\leq s\leq 2\|\beta\|} \|s(s-i\beta)^{-1}\|\right\}$.
\item $\| \psi_{x_1}\| \leq \left( \displaystyle\sup_x\left|\dfrac{p_{x_1}}{p}\right| + M_1 \displaystyle\sup_x\left|\dfrac{\rho_{x_1}}{\rho}\right| \sqrt{1+\dfrac{2}{\sqrt{r}}} \right) M_0$.
\end{enumerate}
\end{lemma}

\begin{proof}\hfill
\begin{enumerate}
\item Expressing $\hat \rho \breve \lambda \breve r^{1/2}=is$, $s$ nonnegative real,
we reduce the problem to that of bounding $\|\psi\| = \|(sI-i\beta)^{-1}q\|$ for $s\in [0,+\infty)$.  For $s\geq (1+b)\|\beta\|$, we have
\[
\|(isI+\beta)^{-1}q\|=\|s^{-1}(I- i\beta/s)^{-1}q\| \leq \frac{\|q\|}{s-\|\beta\|} \leq \dfrac{\|q\|}{b\|\beta\|} \leq \|\beta^{-1}q\|,
\]
whence the first inequality follows by comparison with the value at $s=0$.  The second inequality follows by invertibility for each $s\geq 0$ of $(s-i\beta)$, a consequence of the uniform spectral gap of $\beta(\breve \xi, \breve \lambda)$, as established in Lemma \ref{betalem} below.
\item Following $(i)$, we have for $s\geq 2\|\beta\|$
\[
\|s (isI+\beta)^{-1}\| \leq \frac{1}{1-\|\beta\|/s} \leq \dfrac{1}{1-\frac{1}{2}} = 2.
\]
\item Since $q_{x_1} = q \cdot \dfrac{p_{x_1}}{p}$ and $\hat \rho \breve \lambda \breve r^{1/2} I + \beta_{x_1} = \dfrac{\hat \rho_{x_1}}{\hat \rho} \hat \rho \breve \lambda \breve r^{1/2} \begin{pmatrix} I & \breve r^{-1/2}I\\ 0 & I \end{pmatrix}$, we have that
\begin{align*}
\|\psi_{x_1}\| &=\| ( \hat \rho \breve \lambda \breve r^{1/2}\mat{I}_6 + \beta)^{-1}q_{x_1} +(\hat \rho \breve \lambda \breve r^{1/2}\mat{I}_6 + \beta)^{-1} (\hat \rho_{x_1} \breve \lambda \breve r^{1/2}\mat{I}_6 + \beta_{x_1}) (\hat \rho \breve \lambda \breve r^{1/2}\mat{I}_6 + \beta)^{-1} q\| \\
&\leq \sup_x \left| \dfrac{p_{x_1}}{p}\right| \| \psi \| + \|\hat \rho \breve \lambda \breve r^{1/2} (\hat \rho \breve \lambda \breve r^{1/2}\mat{I}_6 + \beta)^{-1}\| \sup_x \left| \dfrac{\hat \rho_{x_1}}{\hat \rho}\right|  \left\| \begin{pmatrix} I & \breve r^{-1/2}I\\ 0 & I \end{pmatrix}\right\| \|\psi \| \\
&\leq \left( \sup_x \left| \dfrac{p_{x_1}}{p}\right| + M_1 \sup_x\left| \dfrac{\hat \rho_{x_1}}{\hat \rho}\right|  \left\| \begin{pmatrix} I & \breve r^{-1/2}I\\ 0 & I \end{pmatrix}\right\| \right) M_0.
\end{align*}
The result follows from Lemma \ref{lem:matrix-norm}.
\end{enumerate}
\end{proof}

\section{Symbolic analysis of $\beta$}
\label{symbolic}
In this appendix, we verify the spectral information asserted for \eqref{beta} appearing in the high-frequency analysis of Section \ref{sec:hfb}.  This is not logically needed for our numerical study, where we anyway verify hyperbolicity parameter by parameter in the course of numerical block-diagonalization.  However, it is important in understanding why the method works.

\begin{lemma}[Gu\`es et al.\ \cite{GMWZ}
]\label{betalem}
Fix $\mu>|\eta|\ge 0$, $\nu>0$. For each fixed $x_1$, and on the range $-1\le  \xi\le 1$ and $ \lambda=i \tau$, $\tau=1- \xi^2$, $\beta$ is uniformly hyperbolic, i.e., $\Re \sigma(\beta)$ is uniformly bounded from zero. This hyperbolicity holds over the entire range of shock strengths $0\leq e_\sm\leq e_\mathrm{max}$.
\end{lemma}

\begin{proof}
By continuity/compactness, it is sufficient to verify that $\beta$ is hyperbolic for each $ \xi$, $ \lambda=\mi \tau$, $ \xi,  \tau \in \mathbb{R}$.  Suppose that $-\mi \zeta$ is an eigenvalue of $\beta$, or $(\beta+\mi \zeta)\vec{w}=0$, $\vec{w}=(\vec{w}_1, \vec{w}_2)^T$, $\vec{w}_j\in \mathbb{C}^3$.  First, observe that the diagonal $\beta_{12}$ block has eigenvalues $\mi\hat \rho  \tau$ plus $\mu  \xi^2, \tilde \mu  \xi^2, \nu \xi^2$, with modulus bounded by $(|\tau|+| \xi|^2)/C>0$, hence is invertible on the entire parameter range.
Thus, we may solve the $\vec{w}$ equation $ \beta_{12}\vec{w}_2+ \mi\zeta \vec{w}_1=0$ to obtain
$\vec{w}_2=-\mi \zeta \beta_{12}^{-1}\vec{w}_1$.  Substituting the above expression for $\vec{w}_2$ into the $\vec{w}_2$ equation yields
\[
0= \beta_{21}\vec{w}_1 + (\mi \zeta \mat{I}_3 + \beta_{22})\vec{w}_2 = \big(\beta_{21} + (\mi \zeta \mat{I}_3 + \beta_{22})(-\mi \zeta \beta_{12}^{-1})\big)\vec{w}_1,
\]
and thus, rearranging, we find $\det \Big( \beta_{12} + \zeta^2 \beta_{21}^{-1} - \mi \zeta \beta_{21}^{-1} \beta_{22} \Big)=0$, or
\begin{equation}
\label{quadsym}
\hat \rho \mi\tau \in \sigma
\begin{pmatrix}
\mu  \xi^2+\tilde \mu  \zeta^2 & 
 \xi  \zeta \tilde \eta  & 0\\
 \xi  \zeta \tilde \eta  & 
\tilde \mu  \xi^2+\mu  \zeta^2  & 0\\
0 & 0&  \nu ( \xi^2+ \zeta^2) \\
\end{pmatrix}.
\end{equation}
Expanding the symmetric matrix in the righthand side of \eqref{quadsym} as
\begin{equation}
\label{matsym}
\begin{pmatrix}
\mu( \xi^2+ \zeta^2)& 0&0\\
0& \mu( \xi^2+ \zeta^2)&0 \\
0 & 0 &\nu( \xi^2+ \zeta^2)\\
\end{pmatrix}
+
\tilde \eta \begin{pmatrix}
 \zeta^2&  \xi  \zeta & 0\\
 \xi  \zeta & 
 \xi^2 & 0\\
0 & 0&  0\\
\end{pmatrix},
\end{equation}
$\tilde \eta=\mu+\eta\ge 0$, we see that it is $\ge ( \xi^2+ \zeta^2)\min\{\mu,\nu\}$, hence positive definite for $| \xi,  \zeta|\ne 0$, which is a contradiction unless $ \zeta= \xi=0$.  But, $ \zeta= \xi=0$ would imply $ \tau=0$ by \eqref{quadsym}, contradicting $\tau=1- \xi^2=1$.
\end{proof}

\begin{remark}
\label{explanation}
The quadratic form \eqref{matsym} may be recognized as the Fourier symbol of the diffusion matrices for the compressible Navier--Stokes equations written in $(u, v, e)$ coordinates.  Thus, the argument above relies only on parabolicity in $(u,v,e)$ components of the original equations, an instance of a more general principle presented in \cite{GMWZ}.
\end{remark}

\section{Computation of the Mach number}
\label{mach}
A dimensionless quantity measuring shock strength is the {\it Mach number}  $M = \frac{u_\sm - s}{c_\sm}$ (for a left-moving shock), where $u_\sm$ is the downwind velocity, $s$ is the shock speed, and $c_\sm$ is the downwind sound speed.  This gives a convenient measure of comparison between standard versus rescaled coordinates \eqref{rescaled}.
Under our normalizations $s=0$, $u_\sm=1$, this becomes simply $M=c_\sm^{-1}=(\Gamma (1+\Gamma) e_\sm)^{-1/2}$ or, using $e_\sm=\frac{(\Gamma +2)(u_\sp-u_*)}{2\Gamma(\Gamma +1)}$, $M=\frac{\sqrt{2}}{\sqrt{(\Gamma +2)(u_\sp-u_*)}}$.  In the strong shock limit $u_\sp\to u_*$, $M\sim (u_\sp-u_*)^{-1/2}$; in the weak shock limit $u_\sp\to 1$,  $M \to 1$.  In Tables \ref{tbl:mach-numbers}--\ref{tbl:mach-numbers2}, we provide a list of $u_\sp$ values and their corresponding Mach numbers for the monatomic and diatomic cases, respectively.

\begin{table}
\begin{center}$
\begin{array}{|ccc|ccc|ccc|}
u_\sp & M & u_\sp/e_\sm & u_\sp & M & u_\sp/e_\sm & u_\sp & M & u_\sp/e_\sm \\
\hline
0.26 & 8.58 & 21.24 & 0.40 & 2.23 & 2.22 & 0.70 & 1.29 & 1.29 \\
0.27 & 6.10 & 11.13 & 0.45 & 1.94 & 1.87 & 0.75 & 1.22 & 1.25 \\
0.28 & 4.98 & 7.72 & 0.50 & 1.73 & 1.66 & 0.80 & 1.17 & 1.21 \\
0.29 & 4.32 & 6.01 & 0.55 & 1.58 & 1.53 & 0.85 & 1.12 & 1.18 \\
0.30 & 3.87 & 4.98 & 0.60 & 1.46 & 1.43 & 0.90 & 1.07 & 1.15 \\
0.35 & 2.74 & 2.91 & 0.65 & 1.37 & 1.35 & 0.95 & 1.03 & 1.13 \\
\end{array}$
\end{center}
\caption{A list of values of $u_\sp$ with their corresponding Mach numbers for the case of a monatomic gas ($\Gamma=2/3$).  We also include the values of $u_\sp/e_\sm$ as another way of describing the shock strength. Note that both $M\to\infty$ and $u_\sp/e_\sm\to\infty$ when $u_\sp\to 1/4$.}
\label{tbl:mach-numbers}
\end{table}

\begin{table}
\begin{center}$
\begin{array}{|ccc|ccc|ccc|ccc|}
u_\sp & M & u_\sp/e_\sm & u_\sp & M & u_\sp/e_\sm & u_\sp & M & u_\sp/e_\sm & u_\sp & M & u_\sp/e_\sm \\
\hline
0.167 & 50.00 & 233.80 & 0.25 & 3.16 & 1.40 & 0.50 & 1.58 & 0.70 & 0.75 & 1.20 & 0.60 \\
0.17 & 15.81 & 23.80 & 0.30 & 2.50 & 1.05 & 0.55 & 1.47 & 0.67 & 0.80 & 1.15 & 0.59 \\
0.18 & 7.91 & 6.30 & 0.35 & 2.13 & 0.89 & 0.60 & 1.39 & 0.65 & 0.85 & 1.10 & 0.58 \\
0.19 & 5.98 & 3.80 & 0.40 & 1.89 & 0.80 & 0.65 & 1.31 & 0.63 & 0.90 & 1.07 & 0.57 \\
0.20 & 5.00 & 2.80 & 0.45 & 1.71 & 0.74 & 0.70 & 1.25 & 0.61 & 0.95 & 1.03 & 0.57
\end{array}$
\end{center}
\caption{A list of values of $u_\sp$ with its corresponding Mach number $M$ for the case of a diatomic gas ($\Gamma=2/5$).  We also include the values of $u_\sp/e_\sm$ as another way of describing the shock strength. Note that both $M\to\infty$ and $u_\sp/e_\sm\to\infty$ when $u_\sp\to 1/6$.}
\label{tbl:mach-numbers2}
\end{table}
\section{Computational environment and effort}\label{s:repro}

\subsection{Environment}
All numerical computations were performed on one of two Mac Pro workstations.  One has 2 Quad-Core
Intel Xeon processors with a 2.8 GHz rated CPU and the other has two 6-Core Intel Xeon processors with a 2.4 GHz rated CPU.  Computations were carried out in the MATLAB-based package STABLAB \cite{STABLAB}, an open source package available in the GitHub repository, nonlinear-waves/stablab.  Further information is available on request.

\subsection{Effort}
{\it Profiles}:  Computation of profiles is fast.  It takes about 3 seconds to produce the 19-20 profiles used in each of the monatomic and diatomic batch jobs.  The high frequency bounds take roughly 4 minutes of computational time per $u_\sp$, with roughly 80 minutes of computational time for an entire batch of 19-20 profiles.  The low-frequency computations (``rib-roast'') only run on a single core (not optimized for multi-core processing) and take roughly 100 minutes for a typical case.  Weak shocks take considerably longer given the length of the domain, and so an entire batch of 19-20 profiles took around 50 hours.  Finally, an average Evans function contour takes between 1-3 minutes for 100 contour points.  For an entire batch which includes the adaptive mesh for the particularly winding contours, it takes roughly 88 hours of computational time for the 900 contours.

\end{document}